 \documentclass[a4paper,11pt]{article}
\usepackage{pdfsync}

 \usepackage[plainpages=false]{hyperref}
 \usepackage{amsfonts,amsmath,amssymb,amsthm}
 \usepackage{latexsym,lscape,rawfonts,mathrsfs}
\usepackage{enumitem}
\usepackage[overload]{empheq}
\usepackage{cases}

 \usepackage[dvips]{color}
 \usepackage{multicol}
 \definecolor{orange}{rgb}{1,0.5,0}
 \definecolor{brown}{rgb}{0.48,0.33,0.19}
 \definecolor{magenta}{rgb}{1,0,1}
 \definecolor{miao}{cmyk}{0.5,0,0.2,0.2}
 \definecolor{qiao}{gray}{0.96}

 \usepackage[all]{xy}
 \usepackage{eufrak}
 \usepackage{makeidx}
 \usepackage{graphicx,psfrag}
 \usepackage{pstool}

 \usepackage{array,tabularx}

 \usepackage{setspace}

 \usepackage[titletoc,title]{appendix}

\usepackage{txfonts}

 \newcommand{\R}{\ensuremath{\mathbb{R}}}

 \newcommand{\ba}{\begin{align*}}
 \newcommand{\ea}{\end{align*}}
 \newcommand{\na}{\nabla}
\newcommand{\la}{\langle}
\newcommand{\ra}{\rangle}
\newcommand{\lc}{\left(}
\newcommand{\rc}{\right)}
 \newcommand{\KN}{\mathbin{\bigcirc\mspace{-14mu}\wedge\mspace{3mu}}}
\newcommand{\ep}{\epsilon}

 \newcommand{\norm}[2]{{ \ensuremath{\left\|} #1 \ensuremath{\right\|}}_{#2}}
 \newcommand{\snorm}[2]{{ \ensuremath{\left |} #1 \ensuremath{\right |}}_{#2}}

 \makeatletter
 \def\ExtendSymbol#1#2#3#4#5{\ext@arrow 0099{\arrowfill@#1#2#3}{#4}{#5}}
 
 \makeatother

 \makeatletter
 \def\ExtendSymbol#1#2#3#4#5{\ext@arrow 0099{\arrowfill@#1#2#3}{#4}{#5}}
 \newcommand\longright[2][]{\ExtendSymbol{-}{-}{\rightarrow}{#1}{#2}}
 \makeatother

\numberwithin{equation}{section}

\newtheorem{thm}{Theorem}[section]
\newtheorem{theorem}[thm]{Theorem}
\newtheorem{cor}[thm]{Corollary}
\newtheorem{corollary}[thm]{Corollary}
\newtheorem{prop}[thm]{Proposition}
\newtheorem{proposition}[thm]{Proposition}
\newtheorem{lem}[thm]{Lemma}
\newtheorem{lemma}[thm]{Lemma}

\newtheorem{rem}[thm]{Remark}
\newtheorem{remark}[thm]{Remark}

\newtheorem{defn}[thm]{Definition}

\newtheorem{claim}[thm]{Claim}

\title{Heat kernel on Ricci shrinkers}
\author{Yu Li \footnote{Partially supported by research fund from SUNY Stony Brook.} \quad and \quad Bing Wang \footnote{Partially supported by NSF grant DMS-1510401 and research funds from USTC and UW-Madison.}}
\date{\today}

\begin{document}
\maketitle

\begin{abstract}
In this paper, we systematically study the heat kernel of the Ricci flows induced by Ricci shrinkers.
We develop several estimates which are much sharper than their counterparts in general closed Ricci flows. 
Many classical results, including the optimal Logarithmic Sobolev constant estimate,  the Sobolev constant estimate, the no-local-collapsing theorem, the pseudo-locality theorem and the strong maximum principle for curvature tensors, 
are essentially improved for Ricci flows induced by Ricci shrinkers.  
Our results provide many necessary tools to analyze short time singularities of the Ricci flows of general dimension. 
\end{abstract}

\tableofcontents

\section{Introduction}

A Ricci shrinker is a triple $(M^n, g, f)$ of smooth manifold $M^n$, Riemannian metric $g$ and a smooth function $f$ satisfying 
\begin{align} 
Rc+\text{Hess}\,f=\frac{1}{2}g. 
\label{E100}
\end{align}
By a normalization of $f$,  we can assume that 
\begin{align} 
R+|\nabla f|^2&=f,  \label{E101}\\
\int_{M} e^{-f} (4\pi)^{-\frac{n}{2}} dV&=e^{\boldsymbol{\mu}},   \label{eqn:PK27_1}
\end{align}
where $\boldsymbol{\mu}$ is the functional of Perelman.   As usual, we define
\begin{align}
    \mathcal{M}_{n}(A) \coloneqq \left\{ (M^n, g,f) \left|\, \boldsymbol{\mu} \geq -A \right.\right\}.    \label{eqn:PL24_1}
\end{align}

Lying on the intersection of critical metrics and geometric flows, the study of Ricci shrinkers has already become a very important topic in geometric analysis. 
Up to dimension 3, all Ricci shrinkers are classified. In dimension 2, the only Ricci shrinkers are $\R^2, \,S^2$ and $\R P^2$ with standard metrics, due to the classification of Hamilton \cite{Ham95}. In dimension 3, we know that $\R^3,\,S^2 \times \R,\,S^3$ and their quotients are all possible Ricci shrinkers, based on the work of Perelman \cite{Pe1}, Petersen-Wylie \cite{PW10}, Naber \cite{Naber}, Ni-Wallach \cite{NW10} and Cao-Chen-Zhu \cite{CCZ08}.
If we assume the curvature operator to be nonnegative, then the Ricci shrinkers are also classified, see Munteanu-Wang~\cite{MW17}.
However, an important motivation for the study of the Ricci shrinkers is that the Ricci shrinkers are models for short time singularities of the Ricci flows.
In dimension 3, by the Hamilton-Ivey pinch \cite{Ham95}\cite{Ham99}\cite{Ivey93}, one may naturally assume that the Ricci shrinker has nonnegative curvature operator.  
If the dimension is strictly greater than $3$,  the loss of pinch estimate makes the nonnegativity of curvature operator an unsatisfactory condition and should be dropped. 
Also,  it is well known (cf. Haslhofer-M\"uller~\cite{HM11}) that most interesting singularity models are non-compact.
Therefore, to prepare for the singularity analysis of high dimensional Ricci flow,  we shall focus only on the study of \textit{non-compact Ricci shrinkers without any curvature assumption}. 
Since $M$ is non-compact,  the inequality
\begin{align}
      \sup_{M} |Rm|< \infty      \label{eqn:PL24_4}
\end{align}
may fail.  The failure of Riemannian curvature bound causes serious consequences. 
Many fundamental analysis tools, e.g., maximum principle and integration by parts,  cannot be applied directly without estimates of the manifold at infinity.

In this paper, we shall provide a solid foundation for many fundamental analysis tools in the Ricci shrinkers.  We shall  mostly take the point of view that Ricci shrinkers are
time slices of self-similar Ricci flow solutions.  After a delicate choice of cutoff functions and calculations, 
we show that most of the fundamental tools, including maximum principle, existence of (conjugate) heat solutions, uniqueness and stochastic completeness, integration by parts, etc.,  work well on the Ricci shrinker spacetime.  
Then we use these fundamental tools to study the geometric properties of the Ricci flows induced by the Ricci shrinkers.
Therefore, we are able to check that most known important properties of the compact Ricci flows, including monotonicity of Perelman's functional, no-local-collapsing and pseudo-locality theorem of Perelman, curvature tensor strong maximum principle of Hamilton,  do apply
on noncompact Ricci shrinkers. Furthermore, since the Ricci flows induced by the Ricci shrinkers are self-similar,  we obtain many special properties of the Ricci shrinkers.

The first property is the estimate of sharp Logarithmic Sobolev constant, which can be regarded as an improvement of the fact that Perelman's functional is monotone along each Ricci flow.

\begin{theorem}[\textbf{Optimal Logarithmic Sobolev constant}]
Let $(M^n,p,g,f)$ be a Ricci shrinker.
Then $\boldsymbol{\mu}(g,\tau)$ is a continuous function for $\tau>0$ such that $\boldsymbol{\mu}(g,\tau)$ is decreasing for $\tau \le 1$ and increasing for $\tau \ge 1$.
In particular,  we have 
  \begin{align}
     \boldsymbol{\nu}(g) \coloneqq \inf_{\tau >0}\boldsymbol{\mu}(g,\tau)=\boldsymbol{\mu}(g).
  \label{eqn:PK20_1}   
  \end{align}
Consequently, the following properties hold.

\begin{itemize}
  \item Logarithmic Sobolev  inequality. In other words, for each compactly supported locally Lipschitz function  $u$ and each $\tau>0$, we have
        \begin{align}
	  &\int u^2 \log u^2 dV - \left(\int u^2 dV \right) \log \left( \int u^2 dV \right)   + \left( \boldsymbol{\mu} +n +\frac{n}{2} \log (4\pi \tau) \right) \int u^2 dV \notag\\
	  &\leq \tau \int \left\{ 4|\nabla u|^2 +Ru^2 \right\} dV.        \label{eqn:PK20_2}
        \end{align}
        
  \item Sobolev inequality.   Namely,  for each compactly supported locally Lipschitz function $u$, we have
        \begin{align}
            \left(\int u^{\frac{2n}{n-2}}\,dV\right)^{\frac{n-2}{n}} \le C e^{-\frac{2 \boldsymbol{\mu}}{n}} \int \left\{ 4|\nabla u|^2+Ru^2  \right\} dV
          \label{eqn:PK20_3}
        \end{align}
	for some dimensional constant $C=C(n)$. 
\end{itemize}
  \label{thm:A}
\end{theorem}

In geometric analysis, it is a fundamental problem to estimate uniform Sobolev constant.   When the underlying manifold is noncompact, the uniform Sobolev constant in general does not exist.
However, (\ref{eqn:PK20_3}) says that there is a uniform (Scalar-)Sobolev constant, depending only on $n$ and $\boldsymbol{\mu}$.    In particular, if the scalar curvature is bounded, i.e., $\sup_{M} R<\infty$, 
then there exists a classical Sobolev constant. Namely, for each $u \in C_{c}^{\infty}(M)$, we have
\begin{align*}
          \left(\int u^{\frac{2n}{n-2}}\,dV\right)^{\frac{n-2}{n}} \le C e^{-\frac{2 \boldsymbol{\mu}}{n}} \int \left\{ |\nabla u|^2+u^2  \right\} dV
\end{align*}
for some $C=C(n, \sup_{M} R)$.    Note that the term $e^{-\frac{2 \boldsymbol{\mu}}{n}}$ is almost $|B(p, 1)|^{-\frac{2}{n}}$ by Lemma 2.5 of~\cite{LLW18}.

The proof of Theorem~\ref{thm:A} follows a similar route as done in Proposition 9.5 of~\cite{LLW18}, by using the monotonicity of Perelman's functional
along Ricci flow and the invariance of Perelman's functional under diffeomorphism actions.

Secondly, we can improve the no-local-collapsing theorem of Perelman on the Ricci shrinker Ricci flow. 
By the fundamental work of Perelman~\cite{Pe1}, the Ricci flow spacetime can be regarded as a ``Ricci-flat" spacetime in terms of reduced volume and reduced distance. 
Now we can regard Ricci shrinker as a special time slice of the induced Ricci flow.  
On a Ricci flat manifold,  an elementary  comparison argument shows that $\frac{|B(x, r)|}{|B(x,1)|}$ grows at most Euclideanly and at least linearly (cf.~\cite{Yau76}, ~\cite{SZhu}, and Theorem 2.5 of~\cite{Peter12}). 
This comparison geometry picture has a spacetime version which is used to illustrate the no-local-collapsing (cf. ~\cite{Pe1} and~\cite{BW17A}). 
Although the comparison argument (even the space-time version) does not apply directly in the Ricci shrinker case,  we can still show that similar phenomena hold for Ricci shrinkers.

\begin{theorem}[\textbf{Improved no-local-collapsing theorem}]
 Suppose $(M^n,p,g,f)$ is a Ricci shrinker, $r>1$. Then
 \begin{subequations}
\begin{align}[left = \empheqlbrace \,]
&\frac{1}{C} r \leq \frac{|B(p,r)|}{|B(p,1)|} \leq C r^{n}, \label{eqn:PL05_2a}\\
&\inf_{\rho \in (0, r^{-1})} \rho^{-n}|B(q,\rho)| \geq \frac{1}{C} |B(p,1)|.  \label{eqn:PL05_2b}
\end{align}
\label{eqn:PL05_2}
\end{subequations}
\\
\noindent
Here $q$ is any point on $\partial B(p, r)$, and $C$ is a dimensional constant. 
  \label{thm:B}
\end{theorem}

Although the volume estimate (\ref{eqn:PL05_2a}) behaves like the Ricci-flat case,  its proof is totally different and much more involved. 
The proof builds on the the Sobolev inequality (\ref{eqn:PK20_3}) and an improvement (cf. Remark~\ref{rmk:PK29_1}) of the induction argument due to Munteanu and Wang~\cite{MW12}. 
The non-collapsing estimate (\ref{eqn:PL05_2b}) in general does not hold for Ricci-flat manifold.   
This indicates that  Ricci shrinkers are more rigid than Ricci-flat manifold.  See Figure~\ref{fig:comparisonA} for intuition.

 \begin{figure}[h]
 \begin{center}
 \psfrag{A}[c][c]{\color{brown}{$B(p, 1)$}}
 \psfrag{B}[c][c]{\color{red}{$B(q,\rho)$}}
 \psfrag{C}[c][c]{\color{blue}{$B(p, r)$}}
 \includegraphics[width=0.4 \columnwidth]{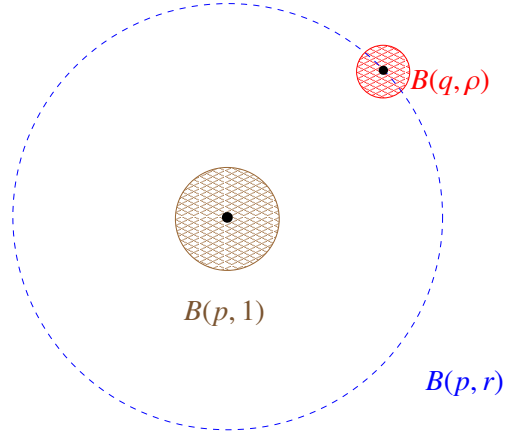}
 \caption{Propagation of non-collapsing on Ricci-shrinkers}
 \label{fig:comparisonA}
 \end{center}
 \end{figure}

 The proof of (\ref{eqn:PL05_2b}) relies on (\ref{eqn:PK20_1}) and  an effective volume estimate in~\cite{BW17A}. 
 The scale $\rho \in (0, r^{-1})$ is chosen such that $R\rho^2 \leq C(n)$ inside $B(q,r)$. If we further assume scalar curvature is uniformly bounded on $M$, then
 we shall obtain that every unit ball on the Ricci shrinker $M$ is uniformly non-collapsed. 
 Theorem~\ref{thm:B} can be regard as a special case of  Theorem~\ref{thm:PL06_2} and Theorem~\ref{thm:PL06_2}, which are more general versions of the no-local-collapsing.  
 The proof of Theorem~\ref{thm:B}, Theorem~\ref{thm:PL06_1} and Theorem~\ref{thm:PL06_2} can be  found in Section~\ref{sec:nolocal}.   
 Note that Theorem~\ref{thm:B} indicates that the Ricci shrinkers are similar to the Ricci-flat manifolds.
 Actually,  there exist many other similarities between the Ricci-flat manifolds and the Ricci Shrinkers. 
 For example, in~\cite{LLW18} and~\cite{LHW18}, it is proved that each sequence of non-collapsed Ricci shrinkers sub-converges  to a limit Riemannian conifold Ricci shrinker. 
 Such results are analogue of the weak compactness theorem of non-collapsed Ricci-flat manifolds, by the deep work of Cheeger, Colding and Naber (cf. ~\cite{CC97},~\cite{CN2},~\cite{ColdNa}).

Thirdly, the pseudo-locality theorem of Perelman has an elegant version on the Ricci shrinker Ricci flow.
The  pseudo-locality theorem of Perelman~\cite{Pe1} is a fundamental tool in the study of Ricci flow. It claims that the Ricci flow cannot turn an almost Euclidean domain
to a very curved region in a short time period.    In the literature, it is known that the pseudo-locality theorem hold for Ricci flow with bounded Riemannian curvature, 
which condition is clearly not available in the current setting.   However, using the existence of special cutoff function, we can show maximum principle and
stochastic completeness for conjugate heat kernel. By carefully checking the integration by parts, we obtain that the traditional pseudo-locality theorem
holds on the Ricci flow spacetime induced by the Ricci shrinker.     Furthermore,  the pseudo-locality has the following special version for Ricci shrinkers. 

\begin{theorem}[\textbf{Improved pseudo-locality theorem}]
Suppose that $(M^n,p,g,f)$ is a non-flat Ricci shrinker. Then we have
  \begin{align}
    \boldsymbol{\mu} < -\delta_0        \label{eqn:PK20_4}
  \end{align}
  for some small positive constant $\delta_0=\delta_0(n)$. 
  Furthermore, the following properties are equivalent.
  \begin{itemize}
    \item[(a)] $M$ has bounded geometry. Namely, the norm of Riemannian curvature tensor is bounded from above and the injectivity radius is bounded from below. 
    \item[(b)] The infinitesimal functional satisfies
                   \begin{align}
                      \lim_{\tau \to 0^{+}} \boldsymbol{\mu}(g,\tau)=0.   \label{eqn:PL05_3}
                   \end{align}
    \item[(c)] The infinitesimal functional satisfies the gap
                   \begin{align}
                      \lim_{\tau \to 0^{+}} \boldsymbol{\mu}(g,\tau)>-\delta_0.  \label{eqn:PL05_4}
                   \end{align}
  \end{itemize}  
  If one of the above conditions hold, we can define 
  \begin{align}
     \tau_0 \coloneqq \sup \left\{ \tau | \,\boldsymbol{\mu}(g,s) \geq -\delta_0, \quad \forall\; s \in (0, \tau)\right\}.     \label{eqn:PL05_0}
  \end{align}
  Then for some positive constant $C=C(n)$, we have the following explicit estimates
\begin{subequations}
\begin{align}[left = \empheqlbrace \,]
& \sup_{x \in M} |Rm|(x) \leq C \tau_0^{-1}, \label{eqn:PL05_1a}\\
&\inf_{x \in M} inj(x) \geq \frac{1}{C} \sqrt{\tau_0}. \label{eqn:PL05_1b}
\end{align}
\label{eqn:PL05_1}
\end{subequations}
\\
\noindent
\label{thm:C}
\end{theorem}

We remark that the gap inequality (\ref{eqn:PK20_4}) is not new. It was first proved by Yokota in~\cite{Yo09} and \cite{Yo12}.
However, our proof of (\ref{eqn:PK20_4}) is completely different and is the base for the proof of (\ref{eqn:PL05_3}), (\ref{eqn:PL05_4}) and (\ref{eqn:PL05_1}). Theorem~\ref{thm:C} also indicates that the bounded geometry for Ricci shrinkers is equivalent to the gap inequality \eqref{eqn:PL05_4}. This criterion has divided all Ricci shrinkers into two categories characterized by their graphs of entropies, which are illustrated by Figure~\ref{fig:mu1} and Figure~\ref{fig:mu2}. 
Note that Figure~\ref{fig:mu1} represents the functional behavior of a typical Ricci shrinker, for example, the cylinder $S^{k} \times \R^{n-k}$ for $k \geq 2$. 
Figure~\ref{fig:mu2} represents the functional behavior of a Ricci shrinker with unbounded geometry. 
However, it is not clear whether such Ricci shrinker exists. 
For Ricci shrinkers with bounded geometry, it follows from (\ref{eqn:PL05_0}) and (\ref{eqn:PL05_1}) that the number $\sqrt{\tau_0}$ can be understood as the regularity scale. 
Actually, under the scale $\sqrt{\tau_0}$, all the higher curvature derivatives norm $|\nabla^k Rm|$ are bounded by $C(n,k) \tau_0^{-1-\frac{k}{2}}$, in light of the estimates of Shi~\cite{Shi89A}.

 \begin{figure}[h]
    \begin{center}
      \psfrag{A}[c][c]{$\boldsymbol{\mu}(g,\tau)$}
      \psfrag{B}[c][c]{$\tau$}
      \psfrag{C}[c][c]{$\boldsymbol{\mu}(g,\tau)=\boldsymbol{\mu}$}
      \psfrag{D}[c][c]{$\tau=1$}
      \psfrag{E}[c][c]{$-\delta_0$}
      \psfrag{F}[c][c]{$\tau=\tau_0$}
      \includegraphics[width=0.5\columnwidth]{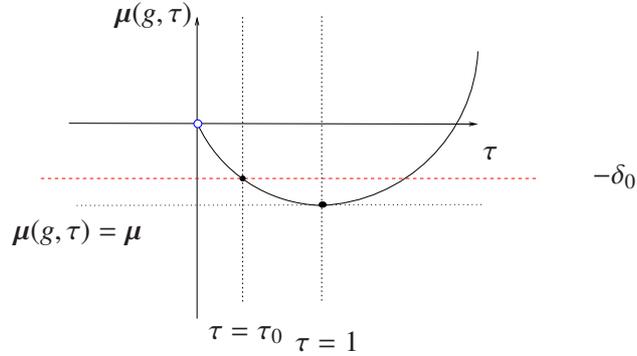}
    \end{center}
    \caption{$\boldsymbol{\mu}(g,\tau)$ of a Ricci shrinker with bounded geometry}
    \label{fig:mu1}
  \end{figure}

\begin{figure}[h]
    \begin{center}
      \psfrag{A}[c][c]{$\boldsymbol{\mu}(g,\tau)$}
      \psfrag{B}[c][c]{$\tau$}
      \psfrag{C}[c][c]{$\boldsymbol{\mu}(g,\tau)=\boldsymbol{\mu}$}
      \psfrag{D}[c][c]{$\tau=1$}
      \psfrag{E}[c][c]{$-\delta_0$}
      \includegraphics[width=0.5\columnwidth]{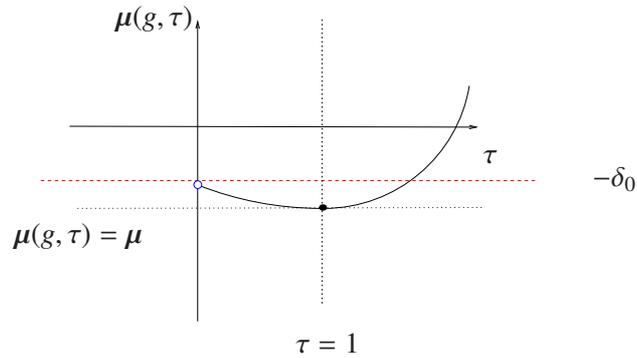}
    \end{center}
    \caption{$\boldsymbol{\mu}(g,\tau)$ of a Ricci shrinker with unbounded geometry}
    \label{fig:mu2}
  \end{figure}

 There exist several other special versions and consequences of the pseudo-locality theorems.  The proof of all of them, including the proof of Theorem~\ref{thm:C}, can be found in Section~\ref{sec:pseudo}.

Fourthly, the curvature tensor strong maximum principle, developed by R. Hamilton, works on Ricci shrinker Ricci flows and also
has an improved version.   Using the curvature tensor maximum principle,  Hamilton shows that the nonnegativity of curvature operator is preserved under the Ricci flow
and the kernel space is parallel. Therefore, the manifold splits as product when kernel space is nontrivial. 
Since different time slices of a Ricci shrinker Ricci flow are the same up to scaling and diffeomorphism, the preservation of curvature conditions is automatic.  
The interesting problem on Ricci shrinker is to show the strong maximum principle, i.e., the splitting of the manifold when eigenvalues of curvature operator satisfy some nonnegativity condition.
On this perspective, we can improve the traditional strong maximum principle of curvature operator to the following format. 

\begin{theorem}[\textbf{Improved strong maximum principle of curvature tensor}]

Suppose $(M^n,g,f)$ is a Ricci shrinker and $\lambda_1 \leq \lambda_2 \leq \cdots$ are the eigenvalue functions of the curvature operator $Rm$. 
Then the following properties hold.
\begin{itemize}
\item If $\lambda_2 \geq 0$ as a function, then there is a $k \in \{0,1, 2, \cdots, n\}$ and a closed symmetric space $N^k$ such that $(M^n, g)$ is isometric to a quotient of  $N^k \times \R^{n-k}$.
\item If $\lambda_2 \geq 0$ as a function and $\lambda_2>0$ at one point,  then $(M^n,g)$ is isometric to a quotient of round sphere $S^n$. 
\end{itemize}
\label{thm:D}
\end{theorem}

 The statement in Theorem~\ref{thm:D} should be well known to experts in Ricci flow if we replace $\lambda_2$ by $\lambda_1$. 
 In fact, by the work of Munteanu-Wang~\cite{MW17} and Petersen-Wylie~\cite{PW10}, we know that the same geometry conclusion hold if we replace $\lambda_2$ in Theorem~\ref{thm:D} by $\lambda_1+\lambda_2$.
 Their proof  builds on the celebrated work of B\"ohm-Wilking~\cite{BW08} on the closed Ricci flow satisfying $\lambda_1+\lambda_2>0$ and also relies on a weighted Riemannian curvature integral estimate $\int_{M} |Rm|^2 e^{-f} dV<\infty$.
 If $\lambda_1+\lambda_2 \geq 0$,  the Riemannian curvature integral estimate can be deduced from the Ricci curvature integral bound $\int_{M} |Rc|^2 e^{-f} dV<\infty$, which follows from a clever integration-by-parts.
 In Theorem~\ref{thm:D}, with only condition $\lambda_2 \geq 0$,   Riemannian curvature integral estimate $\int_{M} |Rm|^2 e^{-f} dV<\infty$ becomes nontrivial. 
 As done in~\cite{LLW18},  we apply  local conformal transformations  and the classical Cheeger-Colding theory to study the local structure of Ricci shrinkers. 
 Combining the $L^2$-curvature estimate of Jiang-Naber~\cite{JN16} with the improved no-local-collapsing Theorem~\ref{thm:B}, we are able to show that $\int_{M} |Rm|^2 e^{-f} dV<\infty$ always holds true (i.e., Theorem~\ref{thm:PL21_1}).    
 Consequently,  the work of Petersen-Wylie~\cite{PW10} applies and the curvature tensor strong maximum principle holds for Ricci shrinkers.  Then we are able to obtain $\lambda_1 \geq 0$ from the condition $\lambda_2 \geq 0$.
 Clearly, the condition $\lambda_2 \geq 0$ is weaker than $\lambda_1 + \lambda_2 \geq 0$ and Theorem~\ref{thm:D} is an improvement of the results of Munteanu-Wang~\cite{MW17} and Petersen-Wylie~\cite{PW10}. 
 Note that $\lambda_2 \geq 0$ is a novel condition in the Ricci flow literature. It is not clear whether $\lambda_2 \geq 0$ is preserved by the Ricci flow on a closed manifold. 
 Actually,  in Theorem~\ref{thm:D}, the same conclusion holds if one replace the condition $\lambda_2 \geq 0$  by an even weak condition
  \begin{align*}
     \lambda_2 \ge -\ep \dfrac{\lambda_1^2}{|R-2\lambda_1|}
 \end{align*}
 for some $\epsilon=\epsilon(n)$.  The details can be found in Theorem~\ref{T:rigidity}.    The proof of Theorem~\ref{thm:D} and Theorem~\ref{T:rigidity} appear in Section~\ref{sec:L2}.\\

The proof of the previous  four theorems requires some elementary, but delicate,  geometric and analytic facts on Ricci shrinkers. 
\begin{itemize}
\item The level sets of $f$ are comparable with geodesic balls. 
\item A special cutoff function.
\item Special heat solution and conjugate heat solution on the Ricci shrinker Ricci flow.
\item The existence of heat kernel and stochastic completeness of the backward heat solution.
\item The existence and uniqueness of bounded (conjugate) heat solutions.
\end{itemize}
After the above estimates are developed, we check that the entropy of Perelman is monotone along the Ricci flow induced by the Ricci shrinker, whose proof needs more delicate integration by parts. 
Then the proof of Theorem~\ref{thm:A} follows a similar route as the one in Proposition 9.5 of~\cite{LLW18}, with more involved technique. 
From Theorem~\ref{thm:A}, we can obtain Theorem~\ref{thm:B} by repeatedly choosing proper test function $u$. 
When integration by parts are assumed, one can formally follows the routine of Perelman to obtain the differential Harnack inequality (i.e., Theorem~\ref{T401}), and then the  traditional pseudo-locality theorem.   
Combining with a standard localization technique, one can deduce Theorem~\ref{thm:C}.    
However, as the functional derivatives contain quadratic Ricci curvature term, many terms concerning high order derivatives need to be carefully handled to verify the integration by parts. 
This causes many technical difficulties.  One key difficulty is the delicate heat kernel estimate to derive the differential Harnack inequality. 
Therefore, the following heat kernel estimate is in the central position for developing fundamental analytic estimates on Ricci shrinker.

\begin{thm}[\textbf{Heat Kernel estimate}]
\label{TI:4}
Let $(M^n,g,f)$ be a Ricci shrinker in $\mathcal M_n(A)$. Then the following properties hold.
\begin{enumerate}[label=(\roman*)]
\item (Heat kernel upper bound)\begin{align*} 
H(x,t,y,s) \le \frac{e^{-\boldsymbol{\mu}}}{(4\pi(t-s))^{\frac{n}{2}}}.
\end{align*}

\item (Heat kernel lower bound) 
For any $0<\delta<1$, $D>1$ and $0<\epsilon <4$, there exists a constant $C=C(n,\delta,D)>0$ such that
\begin{align*} 
H(x,t,y,s) \ge \frac{C^{\frac{4}{\ep}}e^{\boldsymbol{\mu}(\frac{4}{\ep}-1)}}{(4\pi(t-s))^{n/2}} \exp{\lc-\frac{d_t^2(x,y)}{(4-\epsilon)(t-s)}\rc}
\end{align*}
for any $t \in [-\delta^{-1},1-\delta]$ and $d_t(p,y)+\sqrt{t-s}\le D$.

\item (Heat kernel integral bound) 
For any $0<\delta<1$, $D>1$ and $\epsilon>0$, there exists a constant $C=C(n,A,\delta,D,\epsilon)>1$ such that
\begin{align*} 
\int_{M \backslash B_s(x,r\sqrt{t-s})} H(x,t,y,s)\,dV_s(y) \le C\exp{\lc-\frac{(r-1)^2}{4(1+\epsilon)}\rc}
\end{align*}
for any $t \in [-\delta^{-1},1-\delta]$, $d_t(p,x)+\sqrt{t-s}\le D$ and $r \ge 1$.
\end{enumerate}
\end{thm}

We briefly discuss the proof of Theorem \ref{TI:4}. Notice that the Logarithmic Sobolev inequality for all scales implies the ultracontractivity of the heat kernel by Davies' methods (see Chapter 2 of~\cite{Dav89}). 
We prove that the same result (i) holds for Ricci shrinkers. The lower bound of the heat kernel can be estimated by considering the reduced distance (i.e., Theorem~\ref{T302}). 
We first obtain an on-diagonal lower bound of the heat kernel, in which case the estimate of the reduced distance is straightforward. 
Then we derive the general off-diagonal lower bound by exploiting a Harnack property (i.e., \eqref{E311}). 
To prove the integral upper bound, we consider the probability measure $v_s(y)\coloneqq H(x,t,y,s)\,dV_s(y)$. Following the work of Hein-Naber~\cite{HN13}, 
we show that $v_s$ satisfies a type of Logarithmic Sobolev inequality (i.e., Theorem~\ref{T305}). 
The equivalence of the Logarithmic Sobolev inequality and the Gaussian concentration (i.e., Theorem~\ref{T306}) then shows that we can estimate the integral upper bound of the heat kernel by its pointwise lower bound.\\

\textbf{Organization of the paper}:
In section 2, we review the definition of the Ricci flows induced by the Ricci shrinkers.  We also present the estimates of the potential function and volume upper bound. In section 3, we introduce a family of cutoff functions and prove a maximum principle (i.e., Theorem \ref{T101}) on Ricci shrinker spacetime. Moreover, we prove the existence and other basic properties of the heat kernel on spacetime. In section 4, we prove the monotonicity of Perelman's entropy (i.e., Theorem \ref{T201}). In section 5, we prove Theorem \ref{thm:A}. In section 6, we prove the logarithmic Sobolev inequality (i.e., Theorem \ref{T305}) and the Gaussian concentration (i.e., Theorem \ref{T306}) of the probability measure induced by the heat kernel. In section 7, Theorem \ref{TI:4} is proved. In section 8, we prove the differential Harnack inequality (i.e., Theorem \ref{T401}) by using the heat kernel estimates. In section 9, we provide the proof of Theorem \ref{thm:B}. In section 10, we prove the pseudo-locality theorem (i.e., Theorem \ref{thm:PL03_1}) and Theorem \ref{thm:C}. In the last section, we obtain an $L^2$-integral bound of the Riemannian curvature (i.e., Theorem \ref{thm:PL21_1}). As a consequence, we prove Theorem \ref{thm:D}.\\

{\bf Acknowledgements}:  
Yu Li would like to thank Jiyuan Han and Shaosai Huang for helpful comments. Bing Wang would like to thank Haozhao Li and Lu Wang for their interests in this work. Part of this work was done while both authors were visiting IMS (Institute of Mathematical Sciences) at ShanghaiTech University during the summer of 2018. They wish to thank IMS for their hospitality.

\section{Preliminaries}
For any Ricci shrinker $(M^n,g,f)$, let ${\psi^t}: M \to M$ be a $1$-parameter family of diffeomorphisms generated by $X(t)=\dfrac{1}{1-t}\nabla_gf$. That is  
\begin{align} 
\frac{\partial}{\partial t} {\psi^t}(x)=\frac{1}{1-t}\nabla_g f\left({\psi^t}(x)\right).
\label{E102}
\end{align}

By a direct calculation, see \cite[Chapter $4$]{CLN06}, the rescaled pull-back metric $g(t)\coloneqq (1-t) (\psi^t)^*g$ and the pull-back function $f(t)\coloneqq (\psi^t)^*f$ satisfy the equation
\begin{align} 
Rc(g(t))+\text{Hess}_{g(t)}f(t)=\frac{1}{2(1-t)}g(t),
\label{E103}
\end{align}
where $\left\{(M,g(t)), -\infty <t<1 \right\}$ is a Ricci flow solution with $g(0)=g$, that is,
\begin{align} 
\partial_t g=-2Rc(g(t)).
\label{E104}
\end{align}

For notational simplicity, we will omit the subscript $g(t)$ if there is no confusion. From \eqref{E103} and \eqref{E104}, it is easy to show that
\begin{align} 
&\partial_tf=|\na f|^2 \label{E105}, \\
&R+\Delta f=\frac{n}{2(1-t)}\label{E106}, \\
&R+|\nabla f|^2=\frac{f}{1-t}.\label{E107}
\end{align}

Now we define
\begin{align} 
\bar \tau=1-t ,\quad F(x,t)=\bar \tau f(x,t)  \quad \text{and} \quad \bar v(x,t)=(4\pi \bar \tau)^{-n/2}e^{-f(x,t)}.\label{EX101} 
\end{align}

It follows from \eqref{E105}, \eqref{E106} and \eqref{E107} that
\begin{align} 
&\partial_tF=\bar \tau |\nabla f|^2-f=-\bar \tau R,
\label{E108} \\
&\bar \tau R+\Delta F=\frac{n}{2},
\label{E109}\\
&\bar \tau^2R+|\nabla F|^2=F,
\label{E110}
\end{align}

Now we define
\begin{align} 
&\square\coloneqq \partial_t-\Delta_t,  \label{def:heat}\\ 
&\square^*\coloneqq-\partial_t-\Delta_t+R. \label{def:conjheat}
\end{align}

We have special heat solution and conjugate heat solution:
\begin{align}
&\square \left(F  + \frac{n}{2} t\right)=0, \label{eqn:PK15_1}\\
&\square^* \bar v=0.  \label{E110Y}
\end{align}
Note that (\ref{eqn:PK15_1}) is equivalent to
\begin{align}
 \square F=-\frac{n}{2}.
\label{E110X}
\end{align}

Now we have the following estimate of $F$ by using the same method as \cite{CZ10} and \cite{HM11}.  
\begin{lem}
\label{L100}
There exists a point $p \in M$ where $F$ attains its infimum and $F$ satisfies the quadratic growth estimate
\begin{align}
\frac{1}{4}\left(d_t(x,p)-5n\bar \tau-4 \right)^2_+ \le F(x,t) \le \frac{1}{4} \left(d_t(x,p)+\sqrt{2n\bar \tau} \right)^2
\label{eqn:PL20_1}
\end{align}
for all $x\in M$ and $t < 1$, where $a_+ :=\max\{0,a\}$.
\end{lem}

\begin{proof}
This originates from the work of Cao-Zhou \cite[Theorem $1.1$]{CZ10}. We follow the argument of Haslhofer-M\"uller \cite{HM11}. It follows from \cite{CBL07} that for any Ricci shrinker $R \ge 0$ since its corresponding Ricci flow solution is ancient. So from \eqref{E110}, we have
\begin{align} 
  |\nabla F|^2 \le F.      \label{eqn:PK11_9}
\end{align}
It implies that $\sqrt{F}$ is $\frac{1}{2}$-Lipschitz, since 
\begin{align*} 
|\nabla \sqrt{F}|=\frac{1}{2}\left|\frac{\nabla F}{\sqrt{F}}\right | \le \frac{1}{2}.
\end{align*}

On the other hand, for any $x,y \in M$, we choose a minimizing geodesic $\gamma (s), 0\le s \le d=d_t(x,y)$ joining $x=\gamma (0)$ and $y=\gamma (d)$. Assume that $d >2$, we construct a function
\begin{align*} 
\phi(s)=
\begin{cases} 
      s, & s\le 1 \\
    1, & 1\le x\le d-1 \\
      d-s, & d-1 \le x\le d. 
   \end{cases}
\end{align*}
The second variation formula for shortest geodesic implies that
\begin{align} 
\int_0^d \phi^2 Rc(\gamma',\gamma') \, ds \le (n-1)\int_0^d \phi'^2 \, ds=2(n-1).
\label{E111}
\end{align}
Note that from the equation \eqref{E103},
\begin{align} 
\bar \tau\text{Rc}(\gamma',\gamma')=\frac{1}{2}-\text{Hess}F(\gamma',\gamma').
\end{align}
Therefore from \eqref{E110} we have
\begin{align} 
\frac{d}{2}-\frac{2}{3}-2\bar\tau(n-1) &\le \int_0^d \phi^2 \text{Hess}F(\gamma',\gamma') \, ds \notag \\
& \le -2\int_0^1 \phi \nabla_{\gamma'}F \, ds+2\int_{d-1}^d \phi \nabla_{\gamma'}F \, ds \notag \\
& \le \sup_{s\in[0,1]}|\nabla_{\gamma'}F|+\sup_{s\in[d,d-1]}|\nabla_{\gamma'}F| \notag \\
&\le \sqrt{F(x)}+\sqrt{F(y)}+1,
\label{E112}
\end{align}
where we used (\ref{eqn:PK11_9}) in the last inequality.  It is now immediate from \eqref{E112} 
that $F$ has a minimum point $p$. It is clear that $|\nabla F|=0$ and $\Delta F \ge 0$ at the point $p$ by the minimum principle. 
Hence from \eqref{E109} and \eqref{E110} we have 
\begin{align*} 
F(p)=\bar \tau^2 R \le \bar \tau(\bar \tau R+\Delta F) =\frac{\bar \tau n}{2}.
\end{align*}

For any $q \in M$ such that $d_t(p,q)=d$, it is straightforward from (\ref{eqn:PK11_9}) and \eqref{E112} that
\begin{align*} 
\frac{1}{4}\left(d-5n\bar \tau-4  \right)^2_+ \le \frac{1}{4}\left(d-\frac{10}{3}-4\bar \tau(n-1)-\sqrt{2n\bar \tau} \right)^2_+ \le F(q) \le \frac{1}{4} \left(d+\sqrt{2n\bar \tau} \right)^2.
\end{align*}
\end{proof}

Note that $F(\cdot,t)$ is a pull-back function of $f(\cdot,0)$ up to the scale $\bar \tau$, we can choose a base point $p \in M$ such that $p$ is a minimum point for all $F(\cdot,t)$. Now from Lemma \ref{L100}, $F(\cdot,t)$ can be regarded as an approximation of $\frac{d_t^2}{4}$.

With Lemma \ref{L100}, we have the following volume estimate whose proof follows from \cite[Theorem $1.2$]{CZ10}.

\begin{lem}
\label{L101}
There exists a constant $C=C(n)>0$ such that for any Ricci shrinker $(M^n,g,f)$ with $p \in M$ a minimum point of $f$,
\begin{align*}
|B_t(p,r)|_t \le \begin{cases}\,Ce^{\boldsymbol{\mu}}r^n \quad &\text{if} \quad r \ge 2\sqrt{\bar \tau n}; \\
\,Cr^n \quad &\text{if} \quad r < 2\sqrt{\bar \tau n}.
\end{cases}
\end{align*}
\end{lem}

\begin{proof}
We set $\rho=2\sqrt{F}$ and $D(r)=\{x\in M \mid \rho \le r\}$. Moreover we define $V(r)=\int_{D(r)}\,dV_t$ and $\chi(r)=\int_{D(r)}R(t)\,dV_t$. It follows from a similar computation as \cite[(3.5)]{CZ10}, by using \eqref{E109} and \eqref{E110}, that 
\begin{align}\label{eq:V1}
nV-rV'=2\bar\tau \chi-\frac{4{\bar\tau}^2}{r}\chi'.
\end{align}
If we set $r_0=\sqrt{2\bar\tau(n+2)}$, by integrating \eqref{eq:V1} we obtain, see \cite[(3.6)]{CZ10} for details, that
\begin{align*}
V(r) \le 2r^nr_0^{-n}V(r_0)
\end{align*}
for any $r \ge 2\sqrt{\bar \tau n}$.
Then it follows from Lemma \ref{L100} that for any $r \ge 2\sqrt{\bar \tau n}$,
\begin{align*}
|B_t(p,r)|_t \le V\lc r+\sqrt{2n\bar \tau} \rc\le V(2r) \le 2^{n+1}r^nr_0^{-n}V(r_0).
\end{align*}

By definition, we have
\begin{align*}
D(r_0)&=\left\{x\in M \left|  F \le \frac{\bar \tau(n+2)}{2} \right. \right\}\\
&=\left\{x\in M \left| f(x,t) \le \frac{n+2}{2} \right. \right\}=\left\{x\in M \left| f(\psi^t(x)) \le \frac{n+2}{2} \right. \right\}.
\end{align*}
Moreover, since $g(t)=\bar \tau (\psi^t)^*g$, 
\begin{align*}
V(r_0) \le \bar \tau^{\frac{n}{2}} \int_{f(x) \le \frac{n+2}{2}} \,dV \le \bar \tau^{\frac{n}{2}}| \{x \mid f(x) \le (n+2)/2\}|.
\end{align*}
For any $x$ such that $f(x) \le (n+2)/2$. it follows from Lemma \ref{L100} that $d(p,x) \le c_0(n)$. Therefore for any $r \ge 2\sqrt{\bar \tau n}$,
\begin{align*}
|B_t(p,r)|_t \le C(n)|B(p,c_0)| r^n \le C(n)e^{\boldsymbol{\mu}}r^n,
\end{align*}
where the last inequality follows from \cite[Lemma $2.3$]{LLW18}.

Finally, the case $r \le 2\sqrt{\bar \tau n}$ follows from the comparison theorem \cite[Theorem 1.2]{WW09} by using \eqref{E103}. Indeed, for any $x$ with $d_t(p,x) \le 2\sqrt{\bar \tau n}$, it follows from Lemma \ref{L100} that $f(x,t)=\bar \tau^{-1}F(x,t) \le C$. Therefore, from \eqref{E107} we obtain $|\na f|(x,t) \le C{\bar \tau}^{-1/2}$. Now it follows from \cite[(1.5) of Theorem 1.2]{WW09} that for any $s \le r$,
\begin{align*}
\int_{B_t(p,r)} e^{-f(x,t)} \,dV_t \le e^{Cr{\bar \tau}^{-1/2}}\frac{r^n}{s^n}\int_{B_t(p,s)} e^{-f(x,t)} \,dV_t \le C\frac{r^n}{s^n}\int_{B_t(p,s)} e^{-f(x,t)} \,dV_t.
\end{align*}
Then the conclusion follows if we let $s \to 0$.
\end{proof}

\section{Cutoff functions, maximum principle and heat kernel}

Now we construct a family of cutoff functions which is important when we perform integration by parts throughout the paper.

Fix a function $\eta \in C^{\infty}([0,\infty))$ such that $0\le \eta \le 1$, $\eta=1$ on $[0,1]$ and $\eta=0$ on $[2,\infty)$. Furthermore, $-C \le \eta'/\eta^{\frac{1}{2}}\le 0$ and $|\eta''|+|\eta '''| \le C$ for a universal constant $C>0$.
For each $r \ge 1$, we define 
\begin{align}
  \phi^r \coloneqq \eta\left(\frac{F}{r} \right). 
  \label{eqn:PK11_8}
\end{align}
Then $\phi^r$ is a smooth function on $M \times (-\infty,1)$. 
The following estimates of $\phi^{r}$ will be repeatedly used in this paper. 

\begin{lem} 
There exists a constant $C=C(n)$ such that 
\begin{align}
(\phi^r)^{-1} |\nabla \phi^r|^2 & \le Cr^{-1}, \label{E113}\\
|\phi^r_t|&\le C{\bar \tau}^{-1}, \label{E114}\\
|\Delta \phi^r|&\le C(\bar \tau^{-1}+r^{-1}), \label{E115} \\
|\square \phi^r|& \le Cr^{-1}, \label{E115a}\\
|\square^{*} \phi^{r}|&\leq C\left( r^{-1} +\bar{\tau}^{-1} + \bar{\tau}^{-2} r \right).   \label{eqn:PK12_5}
\end{align}
\label{lem:cutoff} 
\end{lem}

\begin{proof}
  Note that $F \leq 2r$ on the support of $\phi^r$, it follows from the assumption of $\eta$ and (\ref{eqn:PK11_9}) that  
\begin{align*}
  \frac{|\nabla \phi^r|^2}{\phi^r}&=r^{-2}\eta'^2\eta^{-1} |\nabla F|^2\le Cr^{-2}F \le Cr^{-1}.
\end{align*}
This finishes the proof of (\ref{E113}). Similarly, by using (\ref{E108}), (\ref{E110}), (\ref{E110X}) and (\ref{eqn:PK11_9}), we can prove  
\begin{align*}
|\phi^r_t|&=r^{-1}|\eta'F_t| \le Cr^{-1}\bar \tau R\le Cr^{-1}\bar \tau^{-1}F \le C{\bar\tau}^{-1}, \\
|\square \phi^r|&=|(\partial_t-\Delta)\phi^r| =|r^{-1}\eta' \square F-r^{-2}\eta''|\nabla F|^2|=|-nr^{-1}\eta'/2-r^{-2}\eta''|\nabla F|^2| \le Cr^{-1}.
\end{align*}
So (\ref{E114}) and (\ref{E115a}) are proved. Then we have
\begin{align*}
  |\Delta \phi^r|&=|-\square \phi^r +\partial_t \phi^{r}| \leq |\square \phi^{r}| + |\phi^{r}_{t}|  \le C(\bar \tau^{-1}+r^{-1}). 
\end{align*}
Hence we obtain (\ref{E115}). 
Finally, using (\ref{E110}) again, we have
\begin{align*}
  |\square^{*} \phi^{r}|&=|\left( -\partial_t-\Delta + R \right) \phi^{r}|=|\left( \square -2\partial_t +R \right) \phi^{r}|\\
  &\leq |\square \phi^{r}| + 2|\phi_{t}^{r}| + R \phi^{r} \leq |\square \phi^{r}| + 2|\phi_{t}^{r}| + \bar{\tau}^{-2} F \phi^{r}\\
  &\leq C\left( r^{-1} +\bar{\tau}^{-1} + \bar{\tau}^{-2} r \right),    
\end{align*}
which proves (\ref{eqn:PK12_5}). 
\end{proof}


Now we move on to show the maximum principle on general Ricci shrinkers. 
On a closed manifold, maximum principle holds automatically.  If the underlying manifold is noncompact, 
then some additional assumptions are needed in order the maximum principle to hold. 
For example, in~\cite[Theorem $15.2$]{Peter12},  a condition
\begin{align} 
\int_a^b \int_M u_+^2(x,t)e^{-cd^2(x)}\,dV\,dt < \infty \label{eqn:PK12_4}
\end{align}
is needed for the maximum principle of the static heat equation subsolution $u$. 
In our current setting of Ricci shrinker spacetime, the metrics are evolving under Ricci curvature. 
Then the distance distortion of  different time slices is not easy to estimate directly without Ricci curvature bound.  
Fortunately, we can replace $d^2$ by $f$ and obtain a maximum principle under a condition similar to (\ref{eqn:PK12_4}). 

\begin{thm}[\textbf{Maximum principle on Ricci shrinkers}]
\label{T101}
Let $(M^n,g,f)$ be a Ricci shrinker. Given any closed interval $[a,b] \subset (-\infty,1)$ and a function $u$ which satisfies $\square u \le 0$ on $M \times [a,b]$, suppose that 
\begin{align} 
\int_a^b \int_M u^2_{+}(x,t)e^{-2f(x,t)}\,dV_t(x)\,dt < \infty.
\label{E116}
\end{align}
If $u(\cdot,a) \le c$, then $u(\cdot, b) \le c$.
\end{thm} 

\begin{proof}
From Lemma \ref{L100} and Lemma \ref{L101}, it is easy to see
\begin{align*} 
\int_a^b \int_M e^{-2f(x,t)}\,dV_t(x)\,dt < \infty
\end{align*}
Therefore, we only need to prove the special case when $c=0$, by considering $u-c$.

Multiplying both sides of $\square u \le 0$ by $u_+(\phi^r)^2e^{-2f}$ and integrating on the spacetime $M \times [a,b]$, then we obtain
\begin{align} 
\int_a^b \int_M \left(\frac{u_+^2}{2}\right)_t(\phi^r)^2e^{-2f}\,dV_t\,dt &\le \int_a^b \int_M \Delta u u_+ (\phi^r)^2e^{-2f}\,dV_t\,dt.
\label{E117}
\end{align}
For the left side of \eqref{E117}, we have
\begin{align} 
&\int_a^b \int_M \left(\frac{u_+^2}{2}\right)_t(\phi^r)^2e^{-2f}\,dV_t\,dt \notag \\
=&\left(\int_M \frac{u_+^2}{2} (\phi^r)^2e^{-2f} \,dV_t     \right)(b) -\int_a^b \int_M u_+^2\phi^r\phi^r_te^{-2f}\,dV_t\,dt \notag \\
&+\int_a^b \int_M u_+^2(\phi^r)^2f_te^{-2f}\,dV_t\,dt +\int_a^b \int_M \frac{u_+^2}{2}(\phi^r)^2Re^{-2f}\,dV_t\,dt \notag \\
\ge &\left(\int_M \frac{u_+^2}{2} (\phi^r)^2e^{-2f} \,dV_t     \right)(b) -\int_a^b \int_M u_+^2\phi^r\phi^r_te^{-2f}\,dV_t\,dt \notag \\
&+\int_a^b \int_M u_+^2(\phi^r)^2|\nabla f|^2e^{-2f}\,dV_t\,dt,
\label{E118}
\end{align}
where we have used $R \ge 0$, $f_t=|\nabla f|^2$ and $u_+(\cdot,a)=0$.
For the right side of~\eqref{E117}, we have
\begin{align} 
& \int_a^b \int_M \Delta u u_+ (\phi^r)^2e^{-2f}\,dV_t\,dt \notag\\
=&\int_a^b \int_M -|\nabla(u_+\phi^r)|^2e^{-2f}\,dV_t\,dt+\int_a^b\int_M|\nabla \phi^r|^2u_+^2e^{-2f}\,dV_t\,dt  \notag\\
&+\int_a^b \int_M2\langle \nabla u_+,\nabla f \rangle u_+(\phi^r)^2e^{-2f}\,dV_t\,dt \notag\\
=&\int_a^b \int_M -|\nabla(u_+\phi^r)|^2e^{-2f}\,dV_t\,dt+\int_a^b\int_M|\nabla \phi^r|^2u_+^2e^{-2f}\,dV_t\,dt  \notag\\
&+\int_a^b \int_M2\langle \nabla (u_+\phi^r),\nabla f \rangle u_+\phi^re^{-2f}\,dV_t\,dt-\int_a^b \int_M2\langle \nabla \phi^r,\nabla f \rangle u_+^2\phi^re^{-2f}\,dV_t\,dt.
\label{E119}
\end{align}
Combining \eqref{E118} and \eqref{E119}, we obtain
\begin{align} 
\left(\int_M \frac{u_+^2}{2} (\phi^r)^2e^{-2f} \,dV_t     \right)(b) \le I+II,
\label{E119a}
\end{align}
where
\begin{align} 
I=& -\int_a^b \int_M u_+^2(\phi^r)^2|\nabla f|^2e^{-2f}\,dV_t\,dt-\int_a^b \int_M |\nabla(u_+\phi^r)|^2e^{-2f}\,dV_t\,dt \notag\\
&+\int_a^b \int_M2\langle \nabla (u_+\phi^r),\nabla f \rangle u_+\phi^re^{-2f}\,dV_t\,dt \le 0, \label{eqn:PK11_7} 
\end{align}
and
\begin{align} 
II=&\int_a^b \int_M u_+^2\phi^r\phi^r_te^{-2f}\,dV_t\,dt+\int_a^b\int_M|\nabla \phi^r|^2u_+^2e^{-2f}\,dV_t\,dt  \notag\\
&-\int_a^b \int_M2\langle \nabla \phi^r,\nabla f \rangle u_+^2\phi^re^{-2f}\,dV_t\,dt.   \label{eqn:PK11_6}
\end{align}
From our construction of $\phi^r$, it is easy to see that all functions involved in last three integrals are supported in the spacetime set
\begin{align}
  K_{r} \coloneqq \{r\le F(x,t)\le 2r, \, a\le t\le b\}. 
  \label{eqn:PK11_5}
\end{align}
Moreover, all the cutoff function terms can be estimated by \eqref{E113} and \eqref{E114}.  For example, we have
\begin{align*} 
  |\langle \nabla \phi^r,\nabla f \rangle| \le \bar \tau^{-1}|\nabla \phi^r||\nabla F| \le C\bar \tau r^{-1/2}\sqrt{F} \le C(1-b)^{-1}, \quad \textrm{on}\; K_{r}. 
\end{align*}
Plugging \eqref{E113},\eqref{E114} and the above inequality into (\ref{eqn:PK11_6}), we arrive at
\begin{align} 
  II\le C\left((1-b)^{-1}+r^{-1}\right)\iint_{K_r} u_+^2e^{-2f}\,dV_t\,dt. 
\label{E120}
\end{align}
It follows from \eqref{E119a},(\ref{eqn:PK11_7}) and \eqref{E120} that  
\begin{align*} 
\left(\int_M \frac{u_+^2}{2} (\phi^r)^2e^{-2f} \,dV_t     \right)(b) \le  C\left((1-b)^{-1}+r^{-1}\right)\iint_{K_r} u_+^2e^{-2f}\,dV_t\,dt. 
\end{align*}
Note that the left hand side of the above inequality is independent of $r$. 
Letting $r \to +\infty$, the finite integral assumption \eqref{E116} implies that
$$
\left(\int_M \frac{u_+^2}{2} e^{-2f} \,dV_t     \right)(b) \leq 0.
$$
Therefore, $u(\cdot,b) \le 0$ by the continuity of $u$ and positivity of $e^{-2f(\cdot, b)}$. 
\end{proof}

The condition (\ref{E116}) is satisfied in many cases.  For example, if $u$ is a bounded heat solution. 
The technique used in the proof of Theorem~\ref{T101} will be repeatedly used in this paper.

Now we control the spacetime integral of $|\text{Hess}\,F|^2$.

\begin{lem}\label{L103}
For any $\lambda>0$, $a<b<1$, there exists a constant $C=C(a,b,\lambda)$ such that
\begin{align*} 
\int_a^b\int |\text{Hess}\,F|^2 e^{-\lambda F} \,dV_t\,dt \le C.
\end{align*}
\end{lem}

\begin{proof}
From \eqref{E110X} and direct computations, 
\begin{align*} 
\square |\nabla F|^2=-2|\text{Hess}\, F|^2.
\end{align*}
Multiplying both sides of the above equation by $\phi^re^{-\lambda F}$ and integrating on the spacetime $M \times [a,b]$,  we obtain
\begin{align*} 
&2\int_a^b\int |\text{Hess}\,F|^2 \phi^re^{-\lambda F} \,dV_t\,dt \\
=&- \left.\left(\int |\nabla F|^2\phi^re^{-\lambda F} \,dV\right) \right |_a^b +\int_a^b\int \square^*\phi^r|\nabla F|^2 e^{-\lambda F} \,dV_t\,dt  \\
&+\lambda\int_a^b\int \left(\phi^r(\lambda|\nabla F|^2-F_t-\Delta F)-2\langle \nabla \phi^r,\nabla F\rangle \right)|\nabla F|^2\phi^r \,dV_t\,dt \notag \\
\le &- \left.\left(\int |\nabla F|^2\phi^re^{-\lambda F} \,dV\right) \right |_a^b +\int_a^b\int |\square^*\phi^r|Fe^{-\lambda F} \,dV_t\,dt \\
&+\lambda\int_a^b\int \left((\lambda+2{\bar \tau}^{-1})F+2r^{-\frac{1}{2}}F^{\frac{1}{2}}\right)Fe^{-\lambda F} \,dV_t\,dt. 
\end{align*}
Now we let $r \to \infty$ and the conclusion follows from Lemma \ref{L100} and Lemma \ref{L101}.
\end{proof}

\begin{theorem}
  On the Ricci flow spacetime $M \times (-\infty, 1)$ induced by a Ricci shrinker $(M, g, f)$, there exists a positive heat kernel
  function $H(x,t,y,s)$ for all $x, y \in M$ and $s, t \in (-\infty, 1)$ with $x \neq y$ and $s<t$. It satisfies  
  \begin{align}
    &\square_{x,t} H(x,t,y,s) \coloneqq \left( \partial_t - \Delta_{x} \right) H(x,t,y,s)=0,  \label{eqn:PK14_1}\\
    &\square_{y,s}^{*} H(x,t,y,s) \coloneqq \left( -\partial_{s} -\Delta_{y}+R(y,s) \right) H(x,t,y,s)=0,   \label{eqn:PK14_2}\\
    &\lim_{t \to s^{+}} H(x,t,y,s)=\delta_{y}, \label{eqn:PK14_3}\\
    &\lim_{s \to t^{-}} H(x,t,y,s)=\delta_{x}. \label{eqn:PK14_4}
  \end{align}
  Furthermore, the heat kernel $H$ satisfies the semigroup property
  \begin{align}
    H(x,t,y,s)=&\int_{M}H(x,t,z,\rho)H(z,\rho,y,s)\, dV_{\rho}(z),  \quad  \forall \; x, y \in M, \; \rho \in (s,t) \subset (-\infty, 1),  \label{eqn:PK14_5}  
  \end{align}
  and the following integral relationships
  \begin{align} 
  &\int_{M}H(x,t,y,s)\,dV_t(x) \le 1, \label{Eqsto}\\
  &\int_{M}H(x,t,y,s)\,dV_s(y) =1. \label{eqn:PK14_6}
 \end{align}
\label{thm:PK14_1}
\end{theorem}

\begin{proof}

We shall divide the proof of Theorem~\ref{thm:PK14_1} into four steps. 

  \textit{Step 1. Existence of a heat kernel function $H$ solving heat equation and conjugate heat equation.} 

Fix a compact interval $I=[a,b] \subset (-\infty,1)$ and a compact set $\Omega \subset M$ with smooth boundary, there exists a Dirichlet heat kernel.
The proof can be found in in~\cite[Chapter $24$, Section $5$]{CCGG}.
Regarding $(-\infty,1)$ as the union of $[-2^{k}, 1-2^{-k}]$, it is easy to see that the Dirichlet  heat kernel actually exists on $\Omega \times (-\infty, 1)$.
Now we let $\left\{ \Omega_i \right\}$ be an exhaustion of $M$ by relatively compact domains with smooth boundary such that
$\overline{\Omega_i}\subset \Omega_{i+1}$. Let $H_i(x,t,y,s)$ be the Dirichlet heat kernel of $(\overline{\Omega_i},g)$.
Then the following properties hold. 
\begin{align} 
\partial_tH_i(x,t,y,s)=&\Delta_{x,t} H_i(x,t,y,s), \label{E201} \\
\partial_sH_i(x,t,y,s)=&-\Delta_{y,s}H_i(x,t,y,s)+R(y,s)H_i(x,t,y,s);  \label{E203} \\
\lim_{ t \searrow s}H_i(x,t,y,s)=&\delta_y,  \label{E202} \\
\lim_{ s \nearrow t}H_i(x,t,y,s)=&\delta_x. \label{E204}
\end{align}
Let $\vec{n}$ be the outward normal vector of $\partial \Omega_i$, then the positivity of $H_i$ implies that $\frac{\partial H_i}{\partial \vec n} \le 0$.  
Since $R \geq 0$ on Ricci shrinkers,  direct computation shows that
\begin{align} 
\partial_t\int_{\Omega_i}H_i(x,t,y,s)\,dV_t(x)=&\int_{\Omega_i} \left( \Delta_{x,t} -R \right) H_i(x,t,y,s) \,dV_t(x)
  \le \int_{\partial \Omega_i} \frac{\partial H_i}{\partial \vec n}\,d\sigma_t(x) \le 0. 
\end{align}
Hence from \eqref{E202}, we have
\begin{align} 
\int_{\Omega_i}H_i(x,t,y,s)\,dV_t(x) \le 1.
\label{E205}
\end{align}
Similarly, we have 
\begin{align} 
\partial_s\int_{\Omega_i}H_i(x,t,y,s)\,dV_s(y)=&-\int_{\Omega_i}\Delta_{y,s} H_i(x,t,y,s) \,dV_s(y) 
=-\int_{\partial \Omega_i}\frac{\partial H_i}{\partial \vec n}\,d\sigma_s(y) \ge 0,
\end{align}
which implies that
\begin{align} 
\int_{\Omega_i}H_i(x,t,y,s)\,dV_s(y) \le 1.
\label{E206}
\end{align}
As $H_i>0$ on $\Omega_i \times (-\infty,1)$, it follows from the classical maximum principle that
\begin{align} 
0\le H_i \le H_{i+1}
\label{E207}
\end{align}
on $\Omega_i \times \Omega_i \times (-\infty, 1)$.
Now we define the heat kernel on $M \times (-\infty, 1)$ by
\begin{align} 
  H(x,t,y,s) \coloneqq \lim_{i \to \infty}H_i(x,t,y,s).  \label{eqn:PK13_1}
\end{align}
From the well-known mean value theorem (cf.~Theorem $25.2$ in~\cite{CCGG}) the interior regularity estimates for the heat equation and conjugate heat equation,
it follows from (\ref{E205}) and (\ref{E206}) that $H_i$ is uniformly bounded when $s,t$ are fixed.
Threfore, $H$ exists as a smooth function.  Its positivity is guranteed by (\ref{E207}). 
The regularity estimates also imply that the convergence from $H_i$ to $H$ is locally  smooth.
In particular, we can take limit of (\ref{E201}) and (\ref{E203}) to obtain that $H$ solves heat equation and conjugate heat equation 
on $M \times (-\infty, 1)$. In other words, (\ref{eqn:PK14_1}) and (\ref{eqn:PK14_2}) are satisfied. 

\textit{Step 2. The heat kernel is a fundamental solution of heat equation and conjugate heat equation.}

Let $\phi$ be a smooth function on $M$ with compact support $K$. For fixed $y$ and $s$, we have
\begin{align} 
&\left|\partial_t\int_{\Omega_i}H_i(x,t,y,s)\phi(x)\,dV_t(x)\right| \notag\\
=&\left|\int_{\Omega_i} (\Delta_{x,t}-R) H_i(x,t,y,s)\phi(x)\,dV_t(x)\right| \notag \\
\le& \left|\int_{\Omega_i} H_i(x,t,y,s)\Delta \phi(x)\,dV_t(x)\right|+\left|\int_{\Omega_i}RH_i(x,t,y,s)\phi(x) \,dV_t(x)\right| \notag \\
\le&C \left|\int_{\Omega_i}H_i(x,t,y,s)\,dV_t(x)\right| \le C,  \label{eqn:PK13_5}
\end{align}
where $C$ is independent of $H_i$. Notice that the last two inequalities hold since we just need to restrict the integral on $K$, and for a fixed $s$, when $t$ is close to $s$, the metrics are uniformly equivalent on $K \times [s,t]$. Combining~\eqref{E202} with (\ref{eqn:PK13_5}), we obtain 
\begin{align} 
  \left|\int_{\Omega_i}H_i(x,t,y,s)\phi(x)\,dV_t(x)-\phi(y)\right|\le C(t-s).  \label{eqn:PK13_6}
\end{align}
Since $\phi$ has compact support, it is clear that 
\begin{align*}
  \lim_{i \to \infty} \int_{\Omega_i} H_i(x,t,y,s)\phi(x)\,dV_t(x)=\int_{M} H(x,t,y,s) \phi(x) dV_{t}(x). 
\end{align*}
Plugging the above equation into (\ref{eqn:PK13_6}) yields that
\begin{align*} 
\left|\int_{M}H(x,t,y,s)\phi(x)\,dV_t(x)-\phi(y)\right|\le C(t-s),
\end{align*}
which means that
\begin{align*}
  \lim_{ t \to s^{+}} \int_{M}H(x,t,y,s)\phi(x)\,dV_t(x)=\phi(y).
\end{align*}
By the arbitrary choice of $\phi$, we obtain (\ref{eqn:PK14_3}). Therefore, $H$ is a fundamental solution of the heat equation.
Similary, we can use the limit argument to derive (\ref{eqn:PK14_4}) and claim that $H$ is a fundamental solution of the conjugate heat equation. 

 \textit{Step 3. The heat kernel satisfies the semigroup property.}

 From its construction,  $H_i$ satisfies the semigroup property:
\begin{align}
  H_i(x,t,y,s)=\int_{\Omega_i}H_i(x,t,z,\rho)H_i(z,\rho,y,s)\, dV_{\rho}(z),  \quad  \forall \; x, y \in \Omega_i, \; \rho \in (s,t) \subset (-\infty, 1).   \label{E204a}
\end{align}
 For each compact set $K \subset M$, it is clear that $K \subset \Omega_i$ for large $i$. 
 By the positivity of each $H_i$, we have
 \begin{align*}
   H(x,t,y,s) &=\lim_{i \to \infty} H_i(x,t,y,s)=\lim_{i \to \infty} \int_{\Omega_i}H_i(x,t,z,\rho)H_i(z,\rho,y,s)\, dV_{\rho}(z)\\
   &\geq \lim_{i \to \infty} \int_{K} H_i(x,t,z,\rho)H_i(z,\rho,y,s)\, dV_{\rho}(z)=\int_{K} H(x,t,z,\rho)H(z,\rho,y,s)\, dV_{\rho}(z). 
 \end{align*}
 By the arbitrary choice of $K \subset M$, the above inequality implies that
 \begin{align}
   H(x,t,y,s) \geq  \int_{M} H(x,t,z,\rho)H(z,\rho,y,s)\, dV_{\rho}(z).   \label{eqn:PK14_7}  
 \end{align}
 By (\ref{E207}), (\ref{eqn:PK13_1}) and the positivity of $H$,  we have
 \begin{align*}
   H_i(x,t,y,s)&=\int_{\Omega_i}H_i(x,t,z,\rho)H_i(z,\rho,y,s)\, dV_{\rho}(z) \leq \int_{\Omega_i} H(x,t,z,\rho)H(z,\rho,y,s) dV_{\rho}(z)\\
   &<\int_{M} H(x,t,z,\rho) H(z,\rho,y,s) dV_{\rho}(z),
 \end{align*}
 whose limit form is
 \begin{align}
   H(x,t,y,s) \leq \int_{M} H(x,t,z,\rho) H(z,\rho,y,s) dV_{\rho}(z).   \label{eqn:PK14_8}
 \end{align}
 Therefore, the semigroup property (\ref{eqn:PK14_5}) follows from the combination of (\ref{eqn:PK14_7}) and (\ref{eqn:PK14_8}).

 \textit{Step 4. The integral relationships (\ref{Eqsto}) and (\ref{eqn:PK14_6}) are satisfied.}

 On each compact set $K \subset M$, since $K \subset \Omega_i$ for large $i$ and each $H_i$ is positive on $\Omega_i$, we have
 \begin{align*}
   \int_{K} H(x,t,y,s) dV_{t}(x) =\lim_{i \to \infty} \int_{K} H_i(x,t,y,s) dV_{t}(x) \leq \lim_{i \to \infty} \int_{\Omega_i} H_i(x,t,y,s) dV_{t}(x) \leq 1,
 \end{align*}
 where (\ref{E205}) is applied in the last step. The arbitrary choice of $K$ then yields that
 \begin{align*}
    \int_{M} H(x,t,y,s) dV_{t}(x) \leq 1, 
 \end{align*}
 which is nothing but (\ref{Eqsto}). Similar reasoning can pass (\ref{E206}) to obtain
 \begin{align}
   \int_{K} H(x,t,y,s) dV_{s}(y) \leq 1,   \label{eqn:PK14_9}
 \end{align}
 where the inequality will be improved to equality (\ref{eqn:PK14_6}) in the following argument. 
 In fact, let $\phi^{r}$ be the cutoff function defined in (\ref{eqn:PK11_8}).
 For any fixed $x$ and $t$, it follows from the cutoff function estimate (\ref{E115a}) that
 \begin{align*} 
 &\left|\partial_s\int H(x,t,y,s)\phi^r(y,s)\,dV_s(y)\right| \\
 =&\left|\int H(x,t,y,s)\square_{y,s}\phi^r(y,s)\, dV_s(y)\right|
 \leq Cr^{-1}\int H(x,t,y,s) \, dV_s(y).
 \end{align*}
 Plugging (\ref{eqn:PK14_9}) into the above inequality, we obtain 
 \begin{align*} 
 \left|\partial_s\int H(x,t,y,s)\phi^r(y,s)\,dV_s(y)\right| \leq Cr^{-1}. 
 \end{align*}
 When $r$ is large, $x$ is covered by the support of $\phi^r$ at the time $t$.  Using (\ref{eqn:PK14_4}), the above inequality implies that
 \begin{align*} 
 \left|\int H(x,t,y,s)\phi^r(y,s)\,dV_s(y)-1\right| \le Cr^{-1}(t-s).
 \end{align*}
 Since $r$ could be arbitrarily large in the above inequality, we obtain (\ref{eqn:PK14_6}) by letting $r \to \infty$. 
\end{proof}

\begin{lemma}
  Suppose $[a,b] \subset (-\infty, 1)$ and $u_{a}$ is a bounded function on the time slice $(M, g(a))$. Then
  \begin{align}
    u(x,t) \coloneqq \int_{M}H(x,t,y,a)u_a(y) dV_{a}(y), \quad \forall \; t \in [a,b]  \label{eqn:PK15_3}
  \end{align}
  is the unique bounded heat solution with initial value $u_a$. 
  \label{lma:PK15_2}
\end{lemma}

\begin{proof}
  Clearly, $u$ is a well-defined heat solution with the initial value $u_a$. Suppose $\tilde{u}$ is another heat solution with initial value $u_a$.  Then
  $u-\tilde{u}$ is a bounded heat solution with initial value $0$.  Therefore, we can apply maximum principle Theorem~\ref{T101} on $\pm (u-\tilde{u})$ to obtain that 
  \begin{align*}
    u -\tilde{u} \equiv 0, \quad \textrm{on} \; M \times [a,b].
  \end{align*}
  In other words, $\tilde{u} \equiv u$ and the uniqueness is proved. 
\end{proof}

\begin{corollary}
  Suppose $u_a$ is a smooth, bounded, integrable function on $(M, g(a))$.
  Let $u$ be the unique bounded heat solution on $M \times [a,b]$ starting from $u_a$. Then we have
  \begin{align}
    \sup_{M} |\nabla u(\cdot, b)|  \leq \sup_{M} |\nabla u(\cdot, a)|.   \label{eqn:PK15_4}
  \end{align}
  \label{cly:PK15_1}
\end{corollary}

\begin{proof}
  Fix $r>>1$ and multiply both sides of $\square u=0$ by $u (\phi^r)^2$ and integrating on $M \times [a,b]$, we obtain  
  \begin{align} 
  \left. \frac{1}{2}\int_{M} u^2(\phi^r)^2 \,dV \right|_a^b-\int_a^b \int_{M} u^2 \phi^r\phi^r_t \,dV\,dt=\int_a^b \int_{M} \left\{-|\na(u \phi^r)|^2+|\na \phi^r|^2u^2 \right\} dV\,dt.
  \label{eqn:PK15_6}  
\end{align}
By Lemma~\ref{lma:PK15_2} we know
\begin{align*}
   u =\int_{M}H(x,t,y,a)u_a(y) dV_{a}.
\end{align*}
Then it follows from (\ref{Eqsto}) that $u$ is bounded and integrable. Consequently, $u^2$ is integrable. It follows from \eqref{E113} and \eqref{E114} that by letting $r \to \infty$, we obtain from \eqref{eqn:PK15_6} 
\begin{align*} 
  \int_a^b \int_{M} |\na u|^2\,dV\,dt  \leq -\left. \frac{1}{2} \int_{M} u^2 dV \right|_a^b+C\bar\tau^{-1}\int_a^b \int_{M} u^2\,dV\,dt <\infty.
\end{align*}
Therefore, the assumption of Theorem~\ref{T101} is satisfied.  Since $\square |\nabla u|^2=-2|\text{Hess}\,u|^2 \leq 0$, 
following the maximum principle, we arrive at (\ref{eqn:PK15_4}).  
\end{proof}

\begin{proposition}
Suppose $u$ is a bounded, integrable heat solution, $w$ is a bounded conjugate heat solution on $M \times [a,b]$ for some compact interval $[a,b] \subset (-\infty, 1)$.
Then we have
\begin{align}
   \left.  \int_{M} uw dV_{t}  \right|_{t=b}=\left. \int_{M} uw dV_{t} \right|_{t=a}.    \label{eqn:PK15_2} 
\end{align}
\label{prn:PK15_1}
\end{proposition}

\begin{proof}

Fix $r >>1$. We calculate 
\begin{align} 
  \partial_t\int_{M} wu\phi^r \,dV&=\int_{M} \left\{ w \square (u\phi^{r}) - (u \phi^{r}) \square^{*} w \right\} dV=\int_{M} w \square \left( u \phi^{r} \right)dV \notag\\
  &=\int_{M} w\left\{ u \square \phi^{r} + \phi^{r} \square u - 2\langle \nabla u, \nabla \phi^{r}\rangle \right\} \notag\\
  &=\int_{M} w\left\{ u \square \phi^{r} - 2\langle \nabla u, \nabla \phi^{r}\rangle \right\}.  \label{EL202a}
\end{align}
Note that $|\nabla u|\le C$ by Corollary~\ref{cly:PK15_1}.  Plugging the cutoff function estimates (\ref{E113}) and (\ref{E115a}) into the above inequality, we obtain
\begin{align*}
  \left| \left.   \int_{M} wu \phi^{r} dV \right|_{a}^{b} \right| \leq C(r^{-1}+r^{-\frac{1}{2}}).
\end{align*}
Taking $r \to \infty$, the right hand side of the above inequality tends to zero, the left hand side converges to
\begin{align*}
  \left. \int_{M} wu dV \right|_{t=b} - \left. \int_{M} wu dV \right|_{t=a},
\end{align*}
since $w$ is bounded and $u$ is integrable. Consequently, we arrive at (\ref{eqn:PK15_2}).  
\end{proof}

\begin{lemma}
  Suppose $[a,b] \subset (-\infty, 1)$ and $w_{b}$ is a bounded function on the time slice $(M, g(b))$. Then
  \begin{align}
    w(y,s) \coloneqq \int_{M}H(x,b,y,s)w_b(x) dV_{b}(x)  \label{eqn:PK15_7}
  \end{align}
  is the unique bounded conjugate heat solution with initial value $w_{b}$. 
  \label{lma:PK15_3}
\end{lemma}

\begin{proof}
  Fix a time $a_0 \in [a,b]$ and let $h$ be an arbitrary smooth function with compact support. 
  Then we solve the heat equation starting from $h$ to obtain
  a unique bounded integrable function $u$ as
  \begin{align*}
    u(x,t)=\int_{M} H(x,t,y,a)h(y) dV_{a}(y).
  \end{align*}
  Since $w$ is given by (\ref{eqn:PK15_7}), it follows from (\ref{eqn:PK14_6}) that $w$ is a bounded function on $M \times [a,b]$.
  Suppose $\tilde{w}$ is another bounded conjugate heat solution starting from $w_b$, then $\tilde{w}-w$ is a bounded conjugate heat solution
  starting from $0$.  Then we can apply Lemma~\ref{lma:PK15_2} to the couple of $u$ and $\tilde{w}-w$ to obtain that for any $t \in [a_0,b]$,
  \begin{align*}
    \int_{M} \left(\tilde{w}(x,t)-w(x,t) \right) u(x,t) dV_{t}(x)=0.
  \end{align*}
In particular,
  \begin{align*}
    \int_{M} \left(\tilde{w}(x,a_0)-w(x,a_0) \right) h(x) dV_{a_0}(x)=0.
  \end{align*}
  By the arbitrary choice of $h$, we obtain $\tilde{w}(\cdot, a_0)-w(\cdot, a_0) \equiv 0$. Then by the arbitrary choice of $a_0$, we see that 
  \begin{align*}
    \tilde{w}(\cdot, t) \equiv w(\cdot, t), \quad \; \forall \; t \in [a,b].
  \end{align*}
  Therefore, the uniqueness of the bounded conjugate heat solution is proved. 
\end{proof}

\begin{lemma}
 Suppose $w$ is a bounded function on $M \times [a,b]$ satisfying $\square^* w \leq 0$.
 Then we have
 \begin{align}
    \sup_{M} w(\cdot, a) \leq \sup_{M} w(\cdot, b).   \label{eqn:PK16_1}
 \end{align}
 
\label{lma:PK16_1} 
\end{lemma}

\begin{proof}
  Without loss of generality,  by adding a constant, we may assume that $\displaystyle \sup_{M} w(\cdot, b)=0$.  Then it suffices to show that
  \begin{align}
      \sup_{M} w(\cdot, a) \leq 0.   \label{eqn:PK16_2}
  \end{align}
  At the time slice $t=a$, we choose an arbitrary nonnegative smooth function $h$ with compact support.  Then we solve the forward heat solution
  starting from $h$ and denote the function by $u$. It is clear that $u \geq 0$. 
  Similar to the proof of Proposition \ref{prn:PK15_1}, we obtain that
  \begin{align*}
     \int_{M} w(x,a)h(x) dV_{a}(x) \leq \int_{M} w(x,b) u(x,b) dV_{b}(x) \leq 0,
  \end{align*}
  since at time $t=b$ we have $u \geq 0$ and $w \leq 0$.  
  Therefore, the inequality (\ref{eqn:PK16_2}) follows from the  arbitrary choice of $h$. 
\end{proof}

\begin{theorem}[\textbf{Bounded heat solution}]
Suppose $t_0 \in (-\infty, 1)$ and $h$ is a bounded function on the time-slice $(M, g(t_0))$.
On $M \times (t_0, 1)$,  starting from $h$,  there is a unique heat solution $u$ which  is bounded on each compact time-interval of $[t_0,1)$. The solution is
\begin{align}
    u(x,t) = \int_{M} H(x,t,y,t_0) h(y) dV_{t_0}(y), \quad \forall \; x \in M, \; t \in (t_0, 1).   \label{eqn:PK17_1}
\end{align}
Similarly, starting from $h$, there is a unique conjugate heat solution $w$ which is bounded on each compact time interval of $(-\infty, t_0]$.
The solution is 
\begin{align}
    w(x,t)=\int_{M} H(y,t_0,x,t) h(y) dV_{t_0}(y), \quad \forall \; x \in M, \; t \in (-\infty, t_0).   \label{eqn:PK17_2}
\end{align}

\label{thm:PK17_1}
\end{theorem}

\begin{theorem}[\textbf{Maximum principle of bounded functions}]
 Suppose $u$ is a bounded super-heat-solution, i.e.,  $\square u \leq 0$  on $M \times [a,b]$. Then 
 \begin{align}
    \sup_{M} u(\cdot, b) \leq \sup_{M} u(\cdot, a).   \label{eqn:PK17_3}
 \end{align}
 Similarly, if $w$ is a bounded super-conjugate-heat-solution, i.e., $\square^* w \leq 0$ on $M \times [a,b]$.   Then
 \begin{align}
     \sup_{M} w(\cdot, b) \geq \sup_{M} w(\cdot, a).     \label{eqn:PK17_4}
 \end{align}
\label{thm:PK17_2} 
\end{theorem}

From (\ref{eqn:PK15_1}) and (\ref{E110Y}) from previous section, on the space-time $M \times (-\infty, 1)$, there are standard heat solution and conjugate heat solutions
$F+\frac{n}{2}t$ and $\bar{v}=(4\pi(1-t))^{-\frac{n}{2}} e^{-f}$.  We can apply Theorem~\ref{T101} and Theorem~\ref{thm:PK17_2} to compare other supersolutions or subsolutions
with $F+\frac{n}{2}t$ and $\bar{v}=(4\pi(1-t))^{-\frac{n}{2}} e^{-f}$.   In particular, we have the following Lemma.

\begin{lem}\label{L202}
Given a smooth function $\phi$ with compact support on a Ricci shrinker $(M^n,g,f)$. For any $b<1$, let $w(x,t)=\int H(y,b,x,t)\phi(y) \,dV_b(y)$ be the bounded solution of conjugate heat equation with $w(\cdot,b)=\phi$. Then there exists a constant $C>0$ such that for $t \le b$
\begin{align} 
w(x,t) \le C\bar v(x,t)=C\frac{e^{-f(x,t)}}{(4\pi\bar\tau)^{n/2}}.
\end{align}
\end{lem}

Lemma \ref{L202} tells us that starting from a compact supported function, the solution of the conjugate heat equation is at least exponentially decaying.


\section{Monotonicity of Perelman's entropy}

 Recall that on any compact Riemannian manifold $(M^n,g)$, Perelman's $\mathcal W$ entropy \cite{Pe1} is defined as
\begin{equation}
\mathcal W(g,\phi,\tau)=\int \left ( \tau(|\nabla \phi|^2+R)+\phi-n \right ) \frac{e^{-\phi}}{(4 \pi \tau)^{n/2}}\,dV
\end{equation} 
for $\phi$ a smooth function and $\tau >0$. Let $u^2=\frac{e^{-\phi}}{(4\pi\tau)^{n/2}}$, we can rewrite above functional as
\begin{equation}
\overline{\mathcal W}(g,u,\tau)=\int \tau(4|\nabla u|^2+Ru^2)-u^2\log u^2 \,dV-\lc n+\frac{n}{2}\log(4\pi\tau)\rc\int u^2   \,dV.
\end{equation} 
For a general Ricci shrinker $(M^n,g,f)$, we define the $\boldsymbol{\mu}$-functional as
\begin{align}
\boldsymbol{\mu}(g,\tau)= \inf\left\{\overline{\mathcal W}(g,u,\tau) \left| u \in \mathcal \mathcal W_{*}^{1,2}(M) \right.  \right\},  \label{eqn:PL21_4}
\end{align}
where
\begin{align} 
W_{*}^{1,2}(M)=\left\{u \left| \int_M |\nabla u|^2\,dV <\infty, \right.  \int_M u^2 \,dV=1 \,\text{and}\, \int_M d^2(p,\cdot)u^2\,dV<\infty \right\}. \label{eqn:PL21_5}
\end{align}

The last integral condition $\int d^2(p,\cdot)u^2\,dV<\infty$ is imposed for two reasons. First, it follows from Lemma \ref{L100} and \eqref{E107} that
\begin{align} 
\int_M Ru^2\,dV<\infty.
\end{align}
Second, the term $\int u^2 \log u^2\,dV$ in the definition of $\overline{\mathcal W}(g,u,\tau)$ is well defined. Indeed, if we consider the rescaled measure $d\tilde V \coloneqq e^{-d^2(p,x)}V$, then it follows from the volume estimate Lemma \ref{L101} that $\tilde V(M)$ is finite. Given a $u \in W_{*}^{1,2}$, we set $A \coloneqq \{x\in M\mid u(x)<1\}$ and $\tilde u \coloneqq \chi_Au$, where $\chi_A$ is the characteristic function of the set $A$. Then it is clear that $\int d^2(p,x)\tilde u^2(x)\,dV<\infty$. By a direct calculation,
\begin{align}\label{eq:changebase}
\int \tilde u^2\log \tilde u^2\,dV=\int \hat u^2\log \hat u^2\,d\tilde V-\int d^2(p,\cdot)\tilde u^2\,dV,
\end{align}
where $\hat u^2=\tilde u^2 e^{d^2(p,\cdot)}$. By Jensen's inequality, we obtain
\begin{align*} 
\int \hat u^2\log \hat u^2\,d\tilde V \ge \lc \int \hat u^2\,d\tilde V \rc \log \lc\frac{1}{\tilde V(M)}\int \hat u^2\,d\tilde V \rc >-\infty
\end{align*}
since $\int \hat u^2\,d\tilde V=\int \tilde u^2\,dV \in [0,1]$. Therefore it follows from \eqref{eq:changebase} that 
\begin{align*} 
\int \tilde u^2\log \tilde u^2\,dV>-\infty.
\end{align*}

In other words, it implies that for any $u \in W_{*}^{1,2}(M)$, the negative part of $u^2\log u^2$ is integrable and $\overline{\mathcal W}(g,u,\tau) \in [-\infty,+\infty)$. In fact, it will be proved later, see Proposition \ref{lma:PK10_3} that $\overline{\mathcal W}(g,u,\tau)$ cannot be $-\infty$.

\begin{rem}\label{rem:mudef}
The space $W_{*}^{1,2}(M)$ can be regarded as a collection of probability measure $v$ such that
\begin{enumerate}[label=(\roman*)]
\item $v=\rho V$, that is, $v$ is absolutely continuous with respect to the volume form $V$.
\item $v$ has finite moment of second order ($v \in P_2(M)$), that is, for any point $q \in M$,
\begin{align*} 
\int d^2(q,\cdot)\,dv < \infty.
\end{align*}
\item The Fisher information
\begin{align*} 
F(\rho) \coloneqq 4\int_M |\na \sqrt{\rho}|^2\,dV<+\infty.
\end{align*}
\end{enumerate}
\end{rem}

Now we show that for any Ricci shrinker, we can always restrict the infimum on all smooth functions with compact support.

\begin{prop}
For any Ricci shrinker $(M^n,g,f)$,
\begin{align}
\boldsymbol{\mu}(g,\tau)= \inf\left\{\overline{\mathcal W}(g,u,\tau) \left|\, u \in C_0^{\infty}(M) \, \, \text{and} \, \int u^2 \,dV=1 \right. \right\}.
\label{eqn:PK08_5}
\end{align}
\end{prop}

\begin{proof}
For any function $u \in W_{*}^{1,2}(M)$ such that $\overline{\mathcal W}(g,u,\tau)$ is finite, we define a positive constant
\begin{align*} 
c_r^2=\int_M u^2(\phi^r)^2 \,dV.
\end{align*}

It is clear from the definition that $c_r \le 1$ and $\lim_{r \to \infty}c_r=1$. From direct computations,
\begin{align*} 
& \overline{\mathcal W}(g,c_r^{-1}u\phi^r,\tau) \notag\\
=&\int c_r^{-2}\tau(4|\nabla (u\phi^r)|^2+R(u\phi^r)^2)-(c_r^{-1}u\phi^r)^2\log (c_r^{-1}u\phi^r)^2  \,dV-n-\frac{n}{2}\log(4\pi\tau) \notag \\
=&\int 4\tau c_r^{-2}\left((\phi^r)^2|\nabla u|^2+|\nabla \phi^r|^2u^2+2u\phi^r\langle \nabla u, \nabla\phi^r \rangle\right)+c_r^{-2}\tau R(u\phi^r)^2 \, dV \notag \\
&-\int (c_r^{-1}\phi^r)^2u^2\log u^2+(c_r^{-1}\phi^r)^2 \log{(\phi^r)^2}u^2\, dV+\log c_r^2-n-\frac{n}{2}\log(4\pi\tau).
\end{align*}

Now by the definition of $W_{*}^{1,2}$ and the dominated convergence theorem,
\begin{align*} 
&\lim_{r \to \infty}\overline{\mathcal W}(g,c_r^{-1}u\phi^r,\tau)-\overline{\mathcal W}(g,u,\tau) \notag \\
=&-\lim_{r \to \infty}\int \left(1-(c_r^{-1}\phi^r)^2\right) u^2\log u^2 \,dV
\end{align*}

Since $u^2\log u^2$ is absolutely integrable, by the dominated convergence theorem,
\begin{align*} 
\lim_{r \to \infty}\int \left(1-(c_r^{-1}\phi^r)^2\right) u^2\log u^2 \,dV=0
\end{align*}
and hence
\begin{align*} 
&\lim_{r \to \infty}\overline{\mathcal W}(g,c_r^{-1}u\phi^r,\tau)=\overline{\mathcal W}(g,u,\tau).
\end{align*}

Similarly, if $\overline{\mathcal W}(g,u,\tau)=-\infty$, then
\begin{align*} 
\lim_{r \to \infty}\overline{\mathcal W}(g,c_r^{-1}u\phi^r,\tau)=-\infty.
\end{align*}

For a fixed $r$, it is not hard to choose a sequence of smooth functions $u_s$ with compact support by the usual smoothing process such that
\begin{align*} 
&\lim_{s \to \infty}\overline{\mathcal W}(g,u_s,\tau)=\overline{\mathcal W}(g,c_r^{-1}u\phi^r,\tau).
\end{align*}
\end{proof}

Now we prove the celebrated monotonicity theorem of Perelman on Ricci shrinkers.

\begin{thm} \label{T201}
For any Ricci shrinker $(M^n,g,f)$ and $\tau>0$,
\begin{align} 
\boldsymbol{\mu}(g(t),\tau-t)
\label{E212}
\end{align}
is increasing for $t < \min\{1,\tau\}$.
\end{thm}

\begin{proof}
We fix a time $t_1 < \min\{1,\tau\}$ and an nonnegative smooth function $\sqrt w$ with compact support such that $\int w=1$. By defining
\begin{align} 
w(x,t)=\int H(y,t_1,x,t)w(y,t_1)\,dV_{t_1}(y),
\label{E213}
\end{align}
it is straightforward to check that 
\begin{align*} 
\int w(x,t) \, dV_t(x)=\iint H(y,t_1,x,t)w(y,t_1)\,dV_t(x)\,dV_{t_1}(y)=\int w(y,t_1)dV_{t_1}(y)=1,
\end{align*}
where we have used stochastic completeness (\ref{eqn:PK14_6}) for the last equality. 

\begin{lem}\label{L201a}
For any time $t_0 < t_1$, 
\begin{align} 
4\int_{t_0}^{t_1}\int |\na \sqrt w|^2\,dV_t\,dt =\int_{t_0}^{t_1}\int\frac{|\nabla w|^2}{w}\,dV_t\,dt <\infty.
\label{E213a}
\end{align}
\end{lem}

\noindent\emph{Proof of Lemma \ref{L201a}}:
By direct computations,
\begin{align} 
&\int_{t_0}^{t_1}\int\frac{|\nabla w|^2}{w}\phi^r\,dV_t\,dt \notag \\
=&\int_{t_0}^{t_1}\int\langle \nabla (\log w), \nabla w \rangle\phi^r\,dV_t\,dt \notag\\
=&-\int_{t_0}^{t_1}\int (\log w) \Delta w \phi^r\,dV_t\,dt-\int_{t_0}^{t_1}\int\log w \langle\nabla  w, \nabla \phi^r \rangle\,dV_t\,dt \notag \\
=&I+II.
\label{E214}
\end{align}

We estimate $I$ first.
\begin{align} 
I \coloneqq &-\int_{t_0}^{t_1}\int (\log w) \Delta w \phi^r\,dV_t\,dt \notag \\
=&\int_{t_0}^{t_1}\int (\log w) w_t \phi^r\,dV_t\,dt-\int_{t_0}^{t_1}\int (\log w) R \phi^r\,dV_t\,dt \notag \\
=&\left.\left(\int (\log w) w \phi^r\,dV\right)\right |_{t_0}^{t_1}-\int_{t_0}^{t_1}\int (\log w)_t w \phi^r\,dV_t\,dt-\int_{t_0}^{t_1}\int (\log w) w \phi^r_t\,dV_t\,dt \notag\\
&+\int_{t_0}^{t_1}\int (\log w) R \phi^r\,dV_t\,dt-\int_{t_0}^{t_1}\int (\log w) R \phi^r\,dV_t\,dt \notag\\
=&\left.\left(\int (\log w) w \phi^r\,dV\right)\right |_{t_0}^{t_1}-\int_{t_0}^{t_1}\int w_t \phi^r\,dV_t\,dt-\int_{t_0}^{t_1}\int (\log w) w \phi^r_t\,dV_t\,dt \notag\\
=&\left.\left(\int (\log w) w \phi^r\,dV\right)\right |_{t_0}^{t_1}-\left.\left(\int w \phi^r\,dV\right)\right |_{t_0}^{t_1}+\int_{t_0}^{t_1}\int w \phi^r_t\,dV_t\,dt \notag \\
&-\int_{t_0}^{t_1}\int w R\phi^r\,dV_t\,dt-\int_{t_0}^{t_1}\int (\log w) w \phi^r_t\,dV_t\,dt.
\label{E215}
\end{align}

Now it is easy to show that all integrals in \eqref{E215} are bounded. Indeed, from Lemma \ref{L202}, there exists a constant $C$ such that 
\begin{align*} 
w(x,t) \le Ce^{-f(x,t)}
\end{align*}
on $M \times [t_0,t_1]$, where $C$ depends only on $t_1$, $t_2$ and the upper bound of $w(\cdot, t_1)$.

Therefore for $t \in [t_0,t_1]$ 
\begin{align*} 
\int |w(\log w)| \, dV_t \le C\int w^{1/2}+w^2\, dV_t \le C\int e^{-f/2}+e^{-f}\, dV_t \le C.
\end{align*}

Moreover, by using \eqref{E107},
\begin{align*} 
\int_{t_0}^{t_1}\int w R\,dV_t\,dt \le C\int_{t_0}^{t_1}\int fe^{-f}\,dV_t\,dt \le C.
\end{align*}

Now we estimate $II$ in \eqref{E214}.
\begin{align} 
|II|\le& \left|\int_{t_0}^{t_1}\int\log w \langle\nabla  w, \nabla \phi^r \rangle\,dV_t\,dt \right| \notag \\
\le &\int_{t_0}^{t_1}\int |\log w||\nabla  w||\nabla \phi^r|\,dV_t\,dt \notag\\
=&\int_{t_0}^{t_1}\int |\log w|\frac{|\nabla  w|}{\sqrt{w}}\frac{|\nabla \phi^r|}{\sqrt{\phi^r}}\sqrt{w}\sqrt{\phi^r}\,dV_t\,dt \notag\\
\le & \frac{1}{2}\int_{t_0}^{t_1}\int\frac{|\nabla w|^2}{w}\phi^r\,dV_t\,dt+\frac{1}{2}\int_{t_0}^{t_1}\int w|\log w|^2\frac{|\nabla \phi^r|^2}{\phi^r}\,dV_t\,dt. 
\label{E216}
\end{align}

By our construction of $\phi^r$, $\frac{|\nabla \phi^r|^2}{\phi^r}$ is uniformly bounded. Reasoning as before,
\begin{align*} 
\int_{t_0}^{t_1}\int w|\log w|^2\frac{|\nabla \phi^r|^2}{\phi^r}\,dV_t\,dt \le C\iint w^{1/2}+w^2\, dV_t\,dt \le C\iint e^{-f/2}+e^{-f}\, dV_t\,dt \le C.
\end{align*}

Now it is easy to see from \eqref{E215} and \eqref{E216}, that 
\begin{align*} 
\int_{t_0}^{t_1}\int\frac{|\nabla w|^2}{w}\phi^r\,dV_t\,dt \le C,
\end{align*}
where $C$ depends only on $t_0$, $t_1$ and the upper bound of $w(\cdot,t_1)$. By taking $r \to \infty$, we have proved Lemma \ref{L201a}.

Now we define the function $\phi$ as
\begin{align*} 
w(x,t)=\frac{e^{-\phi}}{(4\pi(\tau-t))^{n/2}}.
\end{align*}

By direct computations, see Theorem $9.1$ of \cite{Pe1}, that if we set 
\begin{align*} 
v=\left((\tau-t)(2\Delta \phi-|\nabla \phi|^2+R)+\phi-n\right)w,
\end{align*}
then for $t<\tau$,
\begin{align} 
\square^* v=-2(\tau-t)\left|Rc+\text{Hess}\,\phi-\frac{g}{2(\tau-t)}\right|^2w \le 0,
\label{E217}
\end{align}
that is, $v$ is a subsolution of the conjugate heat equation.

We set $\tau_1=\tau_1(t)=\tau-t$ for simplicity. By the definition,
\begin{align} 
v=\tau_1\left(-2\Delta w+\frac{|\nabla w|^2}{w}+Rw\right)-w\log w-(n+\frac{n}{2}\log(4\pi \tau_1))w.
\label{E217a}
\end{align}

Now we multiply both sides of \eqref{E217} by $\phi^r$ so that
\begin{align} 
\int_{t_0}^{t_1}\int v_t \phi^r\,dV_t\,dt \ge& -\int_{t_0}^{t_1}\int \Delta v \phi^r\,dV_t\,dt+\int_{t_0}^{t_1}\int Rv\phi^r\,dV_t\,dt.
\label{E218}
\end{align}

The left side of \eqref{E218} is 
\begin{align} 
\int_{t_0}^{t_1}\int v_t \phi^r\,dV_t\,dt =& -\int_{t_0}^{t_1}\int v \phi_t^r\,dV_t\,dt +\int_{t_0}^{t_1}\int Rv\phi^r\,dV_t\,dt+\left.\left( \int v\phi^r\,dV\right)\right |_{t_0}^{t_1}.
\label{E219}
\end{align}

The right side of \eqref{E218} is
\begin{align} 
-\int_{t_0}^{t_1}\int \Delta v \phi^r\,dV_t\,dt+\int_{t_0}^{t_1}\int Rv\phi^r\,dV_t\,dt=&-\int_{t_0}^{t_1}\int  v\Delta \phi^r\,dV_t\,dt+\int_{t_0}^{t_1}\int R\phi^r\,dV_t\,dt
\label{E220}
\end{align}

Therefore, we have
\begin{align} 
\left.\left( \int v\phi^r\,dV\right)\right |_{t_0}^{t_1} \ge \int_{t_0}^{t_1}\int v \square\phi^r\,dV_t\,dt.
 \label{E221}
\end{align}

Now it is important to use the exact expression of $\square \phi^r$, that is,
\begin{align} 
\square \phi^r=-nr^{-1}{\eta}'/2-r^{-2}{\eta}''|\nabla F|^2.
 \label{E222}
\end{align}

We consider the first term of $v$ and prove the following lemma.

\begin{lem}\label{L201b}
\begin{align} 
\lim_{r \to \infty}\int_{t_0}^{t_1}\int  \Delta w \square \phi^r\,dV_t\,dt=0.
\end{align}
\end{lem}

\noindent\emph{Proof of Lemma \ref{L201b}}:
From \eqref{E222}, we have
\begin{align} 
\int_{t_0}^{t_1}\int  \Delta w \square \phi^r\,dV_t\,dt=&-\frac{n}{2r}\int_{t_0}^{t_1}\int   \Delta w\eta'\,dV_t\,dt-r^{-2}\int_{t_0}^{t_1}\int  \Delta w {\eta}'' |\nabla F|^2 \,dV_t\,dt  \notag \\
=& I+II.
\label{E223}
\end{align}

Now 
\begin{align} 
|I|=& \left|-\frac{n}{2r}\int_{t_0}^{t_1}\int   \Delta w\eta'\,dV_t\,dt \right|=\left |\frac{n}{2r}\int_{t_0}^{t_1}\int   \langle\nabla w,\nabla {\eta}' \rangle\,dV_t\,dt \right| \notag\\
=& \left|\frac{n}{2r^2}\int_{t_0}^{t_1}\int   \langle\nabla w,\nabla F\rangle {\eta}''\,dV_t\,dt \right| \le \frac{n}{2r^2}\int_{t_0}^{t_1}\int   |\nabla w||\nabla F||{\eta}''|\,dV_t\,dt  \notag \\
=&\frac{n}{2r^2}\int_{t_0}^{t_1}\int   \frac{|\nabla w|}{\sqrt{w}}|\nabla F||{\eta}''|\sqrt{w}\,dV_t\,dt  \notag \\
\le &\frac{n}{2r^2}\left(\int_{t_0}^{t_1}\int  \frac{|\nabla w|^2}{w}\,dV_t\,dt \right)^{1/2}\left(\int_{t_0}^{t_1}\int  |\nabla F|^2|{\eta}''|^2w\,dV_t\,dt \right)^{1/2}.
\label{E224}
\end{align}
Now the first integral of \eqref{E223} is bounded by \eqref{E213a} while the second
\begin{align} 
\int_{t_0}^{t_1}\int  |\nabla F|^2|{\eta}''|^2w\,dV_t\,dt \le C\int_{t_0}^{t_1}\int Fw\,dV_t\,dt \le C\int_{t_0}^{t_1}\int Fe^{-f}\,dV_t\,dt v \le C
\end{align}
where the last constant $C$ depends only on $t_0$, $t_1$ and the upper bound of $w(\cdot,t_1)$.

It is immediate that from \eqref{E224} by taking $r \to \infty$ that $\lim_{r \to \infty} I=0$.

We continue to estimate $II$ in \eqref{E223}.
\begin{align} 
|II|=& \left|-r^{-2}\int_{t_0}^{t_1}\int  \Delta w {\eta}'' |\nabla F|^2 \,dV_t\,dt  \right| \notag\\
\le&\left |r^{-2}\int_{t_0}^{t_1}\int   \langle\nabla w,\nabla {\eta}'' \rangle|\nabla F|^2\,dV_t\,dt \right|
+\left|r^{-2}\int_{t_0}^{t_1}\int   \langle\nabla w,\nabla |\nabla F|^2\rangle {\eta}''\,dV_t\,dt \right| \notag \\
=&III+IV.
\label{E225}
\end{align}

Now we have
\begin{align} 
III=&\left |r^{-2}\int_{t_0}^{t_1}\int   \langle\nabla w,\nabla {\eta}'' \rangle|\nabla F|^2\,dV_t\,dt \right|\le r^{-3}\int_{t_0}^{t_1}\int   |\nabla w||\nabla F|^3 |{\eta}'''|\,dV_t\,dt \notag \\
\le & Cr^{-3}\int_{t_0}^{t_1}\int  \frac{|\nabla w|}{\sqrt w}|\nabla F|^3\sqrt{w}\,dV_t\,dt  \notag \\
\le &Cr^{-3}\left(\int_{t_0}^{t_1}\int\frac{|\nabla w|^2}{w}\,dV_t\,dt \right)^{1/2}\left(\int_{t_0}^{t_1}\int |\nabla F|^6w\,dV_t\,dt \right)^{1/2} \le C \notag
\end{align}
since 
\begin{align*} 
\int_{t_0}^{t_1}\int |\nabla F|^6w\,dV_t\,dt \le\int_{t_0}^{t_1}\int F^3e^{-f}\,dV_t\,dt \le C.
\end{align*}

Therefore $\lim_{r \to \infty} III=0$.

Similarly,
\begin{align*} 
IV=&\left|r^{-2}\int_{t_0}^{t_1}\int   \langle\nabla w,\nabla |\nabla F|^2\rangle {\eta}''\,dV_t\,dt \right| \le Cr^{-2}\int_{t_0}^{t_1}\int |\nabla w||\nabla F||\text{Hess} F||{\eta}''|\,dV_t\,dt \notag \\
\le& Cr^{-3/2}\int_{t_0}^{t_1}\int |\nabla w||\text{Hess} F|\,dV_t\,dt=Cr^{-3/2}\int_{t_0}^{t_1}\int \frac{|\nabla w|}{\sqrt w}|\text{Hess} F|\sqrt{w}\,dV_t\,dt \notag \\
\le& Cr^{-3/2}\left(\int_{t_0}^{t_1}\int\frac{|\nabla w|^2}{w}\,dV_t\,dt \right)^{1/2}\left(\int_{t_0}^{t_1}\int |\text{Hess}F|^2w\,dV_t\,dt \right)^{1/2}.
\end{align*}

Now from Lemma \ref{L103} the last integral is bounded since $w \le Ce^{-f}$, so $\lim_{r \to \infty} IV=0$. Therefore, Lemma \ref{L201b} is proved.

We can estimate the integral of $v \square\phi^r$.

From the expression of $v$ in \eqref{E217a}, we have
\begin{align} 
&\int_{t_0}^{t_1}\int v \square\phi^r\,dV_t\,dt \notag \\
=&\int_{t_0}^{t_1}\int \left(\tau_1 (-2\Delta w+\frac{|\nabla w|^2}{w}+Rw)-w\log w-(n+\frac{n}{2}\log(4\pi \tau_1))w\right) \square\phi^r\,dV_t\,dt.
\label{E226}
\end{align}

Since we have $|\square \phi^r|\le Cr^{-1}$ from \eqref{E115a} and all terms except the first in \eqref{E226} have bounded integral on spacetime, it is easy to show, by taking into account of the claim, that 
\begin{align} 
\lim_{r \to \infty} \int_{t_0}^{t_1}\int v \square\phi^r\,dV_t\,dt=0.
\label{E227}
\end{align}

Now from \eqref{E221},
\begin{align} 
\lim_{r \to \infty}\left( \int v\phi^r\,dV\right)(t_1) \ge \lim_{r \to \infty}\left( \int v\phi^r\,dV\right)(t_0).
\label{E228}
\end{align}

Since we choose $\sqrt{w}(\cdot,t_1)$ to be a smooth function with compact support, it is immediate that 
\begin{align} 
\lim_{r \to \infty}\left( \int v\phi^r\,dV\right)(t_1)=\overline{\mathcal W}(g(t_1),\sqrt{w}(\cdot,t_1),\tau-t_1).
\label{E229}
\end{align}

\begin{lem}\label{L201c}
\begin{align*} 
\sqrt{w}(\cdot,t_0)\in W_{*}^{1,2}
\end{align*}
and
\begin{align} \label{eq:lap1} 
\lim_{r \to \infty}\int \Delta w\phi^r \,dV_{t_0} =0.
\end{align}
\end{lem}

\noindent\emph{Proof of Lemma \ref{L201c}}:
From the definition of $v$,
\begin{align*} 
&\lim_{r \to \infty}\left( \int v\phi^r\,dV\right)(t_0)\\
=&\lim_{r \to \infty}\int \left((\tau_1 (-2\Delta w+\frac{|\nabla w|^2}{w}+Rw)-w\log w-(n+\frac{n}{2}\log(4\pi \tau_1))w\right)\phi^r \,dV_{t_0}.
\end{align*}

All terms except for the first two in the above integral are absolutely integrable, due to $w \le Ce^{-f}$ and $R \le \tau^{-2}F$.

Combining with \eqref{E228}, we conclude that
\begin{align*} 
\lim_{r \to \infty}\int( -2\Delta w+\frac{|\nabla w|^2}{w})\phi^r \,dV_{t_0}
\end{align*}
is bounded above.

Then we have
\begin{align} 
&\lim_{r \to \infty}\int( -2\Delta w+\frac{|\nabla w|^2}{w})\phi^r \,dV_{t_0} \notag \\
=&\lim_{r \to \infty}\int2\langle \nabla w,\nabla \phi^r \rangle+\frac{|\nabla w|^2}{w}\phi^r \,dV_{t_0} \notag \\
\ge &\lim_{r \to \infty}\int - \frac{|\nabla w|^2}{2w}\phi^r-2\frac{|\nabla \phi^r|^2}{\phi^r}w+\frac{|\nabla w|^2}{w}\phi^r \,dV_{t_0} \notag\\
=& \frac{1}{2}\lim_{r \to \infty}\int\frac{|\nabla w|^2}{w}\phi^r \,dV_{t_0},
\end{align}
where we have used
\begin{align*}
\lim_{r \to \infty}\int \frac{|\nabla \phi^r|^2}{\phi^r}w \,dV_{t_0}=0.
\end{align*}

To prove \eqref{eq:lap1}, for any $\ep>0$,
\begin{align*} 
\lim_{r \to \infty}\left|\int \Delta w\phi^r \,dV_{t_0}\right|=&\lim_{r \to \infty}\left|\int \langle \nabla w,\nabla \phi^r \rangle \,dV_{t_0}\right| \\
\le& \lim_{r \to \infty} \int \ep \frac{|\nabla w|^2}{w}\phi^r+\ep^{-1}\frac{|\nabla \phi^r|^2}{4\phi^r}w \,dV_{t_0} \\
=&\ep\int \frac{|\nabla w|^2}{w}\,dV_{t_0} 
\end{align*}

By taking $\ep \to 0$, we conclude that \eqref{eq:lap1} holds. Therefore, the proof of Lemma \ref{L201c} is complete.

Therefore, 
\begin{align*} 
\lim_{r \to \infty}\left( \int v\phi^r\,dV\right)(t_0)=\overline{\mathcal W}(g(t_0),\sqrt{w}(\cdot,t_0),\tau-t_0).
\end{align*}

In summary, we have shown from \eqref{E228} that
\begin{align*} 
\overline{\mathcal W}(g(t_1),\sqrt{w}(\cdot,t_1),\tau-t_1) \ge \overline{\mathcal W}(g(t_0),\sqrt{w}(\cdot,t_0),\tau-t_0) \ge \boldsymbol{\mu}(g(t_0),\tau-t_0).
\end{align*}

Since $\tau$, $t_0$, $t_1$ and $\sqrt{w}(\cdot,t_1)$ are arbitrary, the proof of Theorem \ref{T201} is complete.
\end{proof}

\begin{cor}
\label{c:mono}
On a Ricci shrinker $(M^n,g,f)$, the functional $\boldsymbol{\mu}(g,\tau)$ is decreasing for $ 0<\tau <1$ and increasing for $\tau >1$.
\end{cor}

\begin{proof}
The same argument appeared in Step 1, Proposition 9.5 of~\cite{LLW18}.  We repeat the argument here for the convenience of the readers. 

For a fixed constant $\tau_0>1$, from Theorem \ref{T201},
\begin{align*} 
\boldsymbol{\mu}(g(t),\tau_0-t)=\boldsymbol{\mu}((1-t)(\psi^t)^*g,\tau_0-t)=\boldsymbol{\mu}\left(g,\frac{\tau_0-t}{1-t} \right)
\end{align*}
is increasing for $t <1$. Now as $t$ goes from $0$ to $1$, $\frac{\tau_0-t}{1-t}$ goes from $\tau_0$ to $\infty$. As $\tau_0>1$ is arbitrary, we have proved that
$\boldsymbol{\mu}(g,\tau)$ is increasing for all $\tau >1$. Similarly, for any $\tau_0<1$, as $t$ goes from $0$ to $\tau_0$, $\frac{\tau_0-t}{1-t}$ goes from $\tau_0$ to $0$. Therefore, $\boldsymbol{\mu}(g,\tau)$ is decreasing for all $\tau <1$. 
\end{proof}

\section{Optimal logarithmic Sobolev constant---Part I}

For any Ricci shrinker $(M^n,g,f)$ with the normalization \eqref{E101}, we define
\begin{align} \label{eq:mu}
\boldsymbol{\mu}=\boldsymbol{\mu}(g)\coloneqq \log \int\frac{e^{-f}}{(4\pi)^{n/2}}\, dV.
\end{align}

It follows from a direct calculation that $e^{\boldsymbol{\mu}}$ is comparable to the volume of the unit ball $B(p,1)$. 

\begin{lem}[cf. Lemma 2.5 of ~\cite{LLW18}]\label{lem:unitvol}
For any Ricci shrinker $(M^n,g,f)$, there exists a constant $C=C(n)>1$ such that
\begin{align*} 
C^{-1}e^{\boldsymbol{\mu}} \le |B(p,1)| \le Ce^{\boldsymbol{\mu}}.
\end{align*}
\end{lem} 

Next we recall from \cite{Bak94} some standard definitions and properties of the space which satisfies the curvature-dimension estimate.

\begin{defn}
A Riemannian manifold $(M,g,v)$, equipped with a reference measure $v=e^{-W}V$ where $W \in C^2$ and $V$ is the standard volume form, satisfies the $CD(K,\infty)$ condition if the generalized Ricci tensor 
\begin{align*} 
Ric_W \coloneqq Ric+\text{Hess}\,W \ge Kg.
\end{align*}
\end{defn}
In particular, on a Ricci shrinker $(M^n,g,f)$, if we define 
\begin{align}
  f_0=&f+\boldsymbol{\mu}+\frac{n}{2}\log (4\pi), 
  \label{eqn:PK11_1}\\
v_0=&e^{-f_0}V \label{eq:standardmeasure}
\end{align}
then $v_0$ is a probability measure and $(M,g,v_0) \in CD(\frac{1}{2},\infty)$.
Then the following celebrated theorem of Bakry-\'Emery can be applied on Ricci shrinkers. 

\begin{thm}[Bakry-\'Emery theorem \cite{BE85}]
For any Riemannian manifold $(M,g,v)$ satisfying the $CD(K,\infty)$ condition for some $K>0$, the following logarithmic Sobolev inequality holds
\begin{align} 
  \int \rho \log \rho \,dv \le \frac{1}{2K}\int \frac{|\nabla \rho|^2}{\rho} \,dv,   \label{eqn:PK11_2} 
\end{align}
where $v$ and $\rho\, v$  are probability measures which have finite moments of second order and $\rho$ is locally Lipschitz.
\end{thm}

The original proof by Bakry and \'Emery is complete for compact manifolds. A proof using the optimal transport by Lott and Villani for
the general case can be found in \cite[Corollary $6.12$]{LV09}, see also \cite[Theorem $21.2$]{Vil08}. For the self-containedness, we give a proof of the Bakry-\'Emery theorem for Ricci shrinkers.

\begin{thm}\label{T:Bakeme}
For any Ricci shrinker $(M^n,g,f)$ and any nonnegative function $\rho$ such that $\sqrt \rho \in W^{1,2}(M,v_0)$ and $\int d^2(p,\cdot)\rho \,dv_0 <\infty$,
\begin{align*} 
\int \rho\log \rho \,dv_0-\left(\int \rho \,dv_0\right)\log{\left(\int \rho\, dv_0\right)} \le  \int \frac{|\nabla \rho|^2}{\rho}\,dv_0.
\end{align*}
If the equality holds, then either $\rho$ is a constant or $(M^n,g)$ splits off a $\R$ factor.
\end{thm}

Before we prove Theorem \ref{T:Bakeme}, we prove the following two lemmas.

\begin{lem}\label{L:Bakeme}
For any smooth function $u(t,x)$ on $M \times [0,T]$ such that
\begin{align*} 
\square_f u\coloneqq (\partial_t-\Delta_f)u \le 0,
\end{align*}
and for some constant $a>0$,
\begin{align*} 
\int_0^T\int u^2(t,x)e^{-ad^2(p,x)}\,dv_0dt<\infty,
\end{align*}
if $u(\cdot,0) \le c$, then $u \le c$ on $M \times [0,T]$.
\end{lem}

\begin{proof}
The proof follows from \cite[Theorem $15.2$]{Peter12} verbatim by using $\Delta_f$ and the measure $v_0$ instead of $\Delta$ and the volume form $V$.
\end{proof}

We define a new familiy of cutoff functions by setting
\begin{align*}
  \overline \phi^r \coloneqq \eta\left(\frac{f}{r} \right),
\end{align*}
where $\eta$ is the same function as in \eqref{eqn:PK11_8} and $f$ is the potential function at time $0$. A direct calculation shows that
\begin{align*}
\Delta_f   \overline \phi^r=r^{-2}\eta''|\na f|^2+r^{-1}\eta' \Delta_f f=r^{-2}\eta''|\na f|^2+r^{-1}\eta'(\frac{n}{2}-f).
\end{align*}
Then it is clear that $\Delta_f   \overline \phi^r$ is supported on $\{f \ge r\}$ and there exists a constant $C=C(n)$ such that
\begin{align}\label{eq:Ba1}
|\Delta_f  \overline \phi^r | \le C.
\end{align}

\begin{lem}\label{L:Bakeme2}
For any smooth bounded function $u$ on $M$,
\begin{align*} 
\lim_{r \to \infty}\int (\Delta_f u) \overline \phi^r\,dv_0=0.
\end{align*}
\end{lem}
\begin{proof}
From the integration by parts,
\begin{align*}
\lim_{r \to \infty} \int (\Delta_f u)\overline \phi^r \,dv_0=\lim_{r \to \infty} \int  u(\Delta_f \overline \phi^r) \,dv_0=0,
\end{align*}
where the last equality holds since $u$ is bounded and $v_0$ is a probability measure.
\end{proof}

\noindent \emph{Proof of Theorem \ref{T:Bakeme}}: We only prove the inequality for $\rho_0$ such that $\sqrt \rho_0$ is a compactly supported smooth function and the general case follows from approximations as in Proposition \ref{eqn:PK08_5}. In addition, we assume that $\int \rho_0\,dv_0=1$.

Given such $\rho_0$, we consider the heat flow with respect to the measure $v_0$, that is,
\begin{align*}
\begin{cases}
&\partial_t \rho=\Delta_f \rho,\\
&\rho(0,\cdot)=\rho_0.
\end{cases}
\end{align*}

It is clear that there exists a constant $C$ such that $\rho \le C$ on $M \times [0,\infty)$. Now we set 
\begin{align*}
E(t) \coloneqq \lc \int \rho \log \rho \,dv_0 \rc (t).
\end{align*}

By direct computations
\begin{align*} 
 \partial_t \int \rho (\log \rho)  \overline \phi^r \,dv_0 =&\int \rho_t(\log \rho+1) \overline \phi^r \,dv_0\\
=&\int \Delta_f \rho(\log \rho+1) \overline \phi^r \,dv_0 \\
=&\int -\frac{|\na \rho|^2}{\rho} \overline \phi^r+\Delta_f(\rho\log \rho)\overline \phi^r\,dv_0. 
\end{align*}

Therefore, for any $T>0$,
\begin{align*} 
&\lc \int \rho (\log \rho)  \overline \phi^r \,dv_0 \rc(T)-\lc\int \rho (\log \rho)  \overline \phi^r \,dv_0\rc(0) \\
=& \int_0^T \int -\frac{|\na \rho|^2}{\rho} \overline \phi^r+\Delta_f(\rho\log \rho)\overline \phi^r\,dv_0dt.
\end{align*}

It follows from Lemma \ref{L:Bakeme2} that
\begin{align}  \label{eq:BE1}
 \int_0^T \int \frac{|\na \rho|^2}{\rho}\,dv_0dt <\infty
\end{align}
and
\begin{align} \label{eq:BE2}
E(T)-E(0)=- \int_0^T \int \frac{|\na \rho|^2}{\rho}\,dv_0dt.
\end{align}

We compute
\begin{align}\label{eq:BE2a}
 \partial_t \int \frac{|\na \rho|^2}{\rho} \overline \phi^r\,dv_0 =\int \square_f \left(\frac{|\na \rho|^2}{\rho}\right) \overline \phi^r+ \Delta_f \left(\frac{|\na \rho|^2}{\rho}\right) \overline \phi^r\,dv_0.
\end{align}

From Bochner's formula,
\begin{align*}
\partial_t |\na \rho|^2=& 2\la \na \Delta_f \rho,\na \rho \ra \\
=& \Delta_f |\na \rho|^2-2|\text{Hess}\,\rho|^2-2(Rc+\text{Hess}\,f)(\na \rho,\na \rho) \\
=&\Delta_f |\na \rho|^2-2|\text{Hess}\,\rho|^2-|\na \rho|^2,
\end{align*}
where we have used the Ricci shrinker equation for the last equality.

Therefore,
\begin{align}\label{eq:BE4}
\square_f |\na \rho|^2=-2|\text{Hess}\,\rho|^2-|\na \rho|^2.
\end{align}

A direct calculation shows that
\begin{align}\label{eq:BE5}
\square_f \frac{|\nabla \rho|^2}{\rho}=-\frac{2}{\rho}\left|\text{Hess}\,\rho-\frac{d\rho \otimes d\rho}{\rho}\right|^2 -\frac{|\nabla \rho|^2}{\rho}.
\end{align}

It follows from \eqref{eq:BE4} and Lemma \ref{L:Bakeme} that there exists a constant $C>0$ such that
\begin{align}\label{eq:BE6}
\frac{|\nabla \rho|^2}{\rho} \le C.
\end{align}

Therefore, by \eqref{eq:BE2a} and Lemma \ref{L:Bakeme2}, for any $T>S>0$,
\begin{align} 
&\lc \int \frac{|\nabla \rho|^2}{\rho} \,dv_0 \rc(T)-\lc\int \frac{|\nabla \rho|^2}{\rho} \,dv_0\rc(S)  \notag\\
=& \int_S^T \int -\frac{2}{\rho}\left|\text{Hess}\,\rho-\frac{d\rho \otimes d\rho}{\rho}\right|^2 -\frac{|\nabla \rho|^2}{\rho}\,dv_0dt.\label{eq:BE7}
\end{align}

It follows from \eqref{eq:BE2} that for any $t \ge 0$,
\begin{align}\label{eq:BE8}
E'(t)=-\lc \int  \frac{|\nabla \rho|^2}{\rho} \,dv_0 \rc (t) \le 0
\end{align}

Moreover, for any $t >s \ge 0$, it follows from \eqref{eq:BE7} that
\begin{align}\label{eq:BE9}
-E'(t)+E'(s) \le \int_s^t E'(z)\,dz \le 0.
\end{align}

Then it is easy to see from \eqref{eq:BE9} that 
\begin{align}\label{eq:BE10}
E'(t) \ge E'(0)e^{-t}.
\end{align}

Now we claim that $E(t) \to 0$ if $t \to \infty$. Since $E(t)$ is decreasing by \eqref{eq:BE8}, we only need to prove the claim by considering a sequence $t_i \to \infty$. We define $u_i=\sqrt{\rho(t_i,\cdot)}$, then 
\begin{align}\label{eq:BE11}
\int u_i^2 \,dv_0=1
\end{align}
and by \eqref{eq:BE10},
\begin{align}\label{eq:BE12}
\int |\na u_i|^2 \,dv_0 \to 0.
\end{align}

Then by taking a subsequence, we claim that $u_i$ converges to $u_{\infty}$ weakly in $W^{1,2}(M,v_0)$. It is clear from \eqref{eq:BE11} and \eqref{eq:BE12} that $u_{\infty} \equiv 1$. Since we can assume that $u_i$ converges to $1$ almost everywhere,
\begin{align}\label{eq:BE13}
\lim_{i \to \infty} \int u_i^2 \log u_i^2\,dv_0 =0
\end{align}
by the dominated convergence theorem. Therefore, $E(t) \to 0$ if $t \to \infty$.

It follows from \eqref{eq:BE2} and \eqref{eq:BE10} that
\begin{align*}
\int \rho_0\log \rho_0\,dv_0=E(0)=-\int_0^{\infty}E'(t)\,dt \le -E'(0) \int_0^{\infty}e^{-t}\,dt=\int  \frac{|\nabla \rho_0|^2}{\rho_0} \,dv_0.
\end{align*}

If the equality holds and $\rho_0$ is not a constant, it follows from \eqref{eq:BE7} that 
\begin{align*}
\text{Hess}(\log \rho)=\frac{1}{\rho}\lc\text{Hess}\,\rho-\frac{d\rho \otimes d\rho}{\rho} \rc=0.
\end{align*}
Therefore, $(M^n,g)$ splits off a $\R$ factor.

In summary, the proof of Theorem \ref{T:Bakeme} is complete.

Using the Bakry-\'Emery theorem, Carrillo and Ni have proved in \cite{CN09} the following result.

\begin{proposition}[Carrillo-Ni~\cite{CN09}]
For any Ricci shrinker $(M^n,g,f)$, we have 
\begin{align}
  \overline{\mathcal{W}} (g,e^{-\frac{f_0}{2}},1)=\boldsymbol{\boldsymbol{\mu}}(g,1)=\boldsymbol{\boldsymbol{\mu}},
\label{eq:cani}  
\end{align}
where $f_0$ is the normalization of $f$ defined in (\ref{eqn:PK11_1}). 
\label{pro:logsob1}
\end{proposition}

\begin{proof}
We shall follow the argument of Carrillo and Ni.   The proof is given for the self-containedness. 

For any Ricci shrinker $(M^n,g,f)$ and any smooth function $u$ on $M$ with compact support such that $\int u^2 \,dV=1$, we define $w=u^2e^{f_0}$. Then it is clear that both $v_0$ and $wv_0$ belong to $P_2(M)$ from the estimates of $f$ and $dV$.

It follows from Theorem \ref{T:Bakeme} that
\begin{align} \label{eq:bak1} 
\int w \log w \,dv_0 \le \int \frac{|\na w|^2}{w} \,dv_0.
\end{align}
By rewriting \eqref{eq:bak1} in terms of $u$, we have
\begin{align}  \label{eq:bak2} 
\int u^2 \log u^2\,dV+\int f_0u^2\,dV \le \int 4|\na u|^2+|\na f_0|^2u^2+4\la \na u,\na f_0\ra u\,dV.
\end{align}
It follows from the integration by parts for the last term that \eqref{eq:bak2} becomes
\begin{align}  \label{eq:bak3} 
\int u^2 \log u^2\,dV+\boldsymbol{\mu}+\frac{n}{2}\log(4\pi) \le \int 4|\na u|^2+u^2(|\na f|^2-2\Delta f-f)\,dV.
\end{align}
It follows from the $|\na f|^2+R=f$ and $\Delta f+R=\frac{n}{2}$ that $|\na f|^2-2\Delta f-f=R-f$. Therefore, by \eqref{eq:bak3} that
\begin{align*}  
  \overline{\mathcal{W}} (g,u,1)=\int \left\{ 4|\na u|^2+Ru^2-u^2\log u^2 \right\} dV-n-\frac{n}{2}\log(4\pi)\ge \boldsymbol{\mu}.
\end{align*}
By the arbitrary choice of $u$, the above inequality means that 
\begin{align}  
  \boldsymbol{\mu}(g,1) \ge \boldsymbol{\mu}.  \label{eqn:PK11_3}
\end{align}
On the other hand, if we set $u_1=e^{-\frac{f_0}{2}}$, it follows from direct calculation that
\begin{align*}  
\overline{\mathcal{W}} (g,u_1,1)
  =\int \lc |\na f|^2+R+f-n \rc\,e^{-f_0}dV+\boldsymbol{\mu}. 
\end{align*}
Recall that $R+|\nabla f|^2=f$ and $R+\Delta f=\frac{n}{2}$ on a Ricci shrinker. So the above equation can be simplified as
\begin{align*}
\overline{\mathcal{W}} (g,u_1,1)-\boldsymbol{\mu}=\int \lc 2f-n \rc\,e^{-f_0}dV=-2\int (\Delta_{f} f) e^{-f_0} dV
  =-2\int (\Delta_{f_0} f) e^{-f_0} dV=0. 
\end{align*}
Then it follows from definition that
\begin{align}
  \boldsymbol{\mu}(g,1) \leq  \overline{\mathcal{W}} (g,u_1,1)=\boldsymbol{\mu}.   \label{eqn:PK11_4}
\end{align}
Therefore, (\ref{eq:cani}) follows from the combination of (\ref{eqn:PK11_3}) and (\ref{eqn:PK11_4}).
\end{proof}

\begin{cor}
For any Ricci shrinker $(M^n,g,f)$, if there exist more than one minimizer $u \in W^{1,2}_{*}$ for $\overline{\mathcal{W}} (g,u,1)$, then $(M,g)$ must split off a $\R$ factor.
\end{cor}

\begin{proof}
If $u$ is a minimizer other than $e^{-\frac{f_0}{2}}$, then the same proof as Proposition \ref{pro:logsob1} shows that
\begin{align*}
\int w \log w \,dv_0 = \int \frac{|\na w|^2}{w} \,dv_0,
\end{align*}
where $w=u^2e^{f_0}$.
Then the conclusion follows from the equality case of Theorem \ref{T:Bakeme}.
\end{proof}

Proposition \ref{pro:logsob1} indicates that $\boldsymbol{\mu}$ is the optimal log-Sobolev constant for $(M^n,g,f)$ on scale $1$.
We shall improve \eqref{eq:cani} by showing that $\boldsymbol{\mu}$ is in fact the optimal log-Sobolev constant for all scales.
Note that the same result has already been proved for compact Ricci shrinkers in Proposition 9.5 of~\cite{LLW18}.

\begin{proposition}\label{prn:PK10_2}
For any Ricci shrinker $(M^n,g,f)$, we have 
\begin{align}
\boldsymbol{\nu}(g) \coloneqq \inf_{\tau >0}\boldsymbol{\mu}(g,\tau)=\boldsymbol{\mu}.
\label{eqn:PK10_16}
\end{align}
\end{proposition}

We first show two important intermediate steps before we prove Proposition~\ref{prn:PK10_2}. 

\begin{lemma}  \label{lma:PK10_3}
  For each $\tau \in (0,1)$, we have
  \begin{align}
    \boldsymbol{\mu}(g, \tau) \geq \boldsymbol{\mu}=\boldsymbol{\mu}(g,1).   \label{eqn:PK08_4} 
  \end{align}
\end{lemma}

\begin{proof}

  Fix $\eta_0 \in (0,1)$.  Let $w$ be a nonnegative, compactly supported smooth function satisfying the normalization condition $\int w dV=1$. 
  We now regard $w$ as a smooth function on the time slice $t=0$ and solve the conjugate heat equation $\square^* w=0$.
  Then $w$ is a smooth function on the space-time $M \times (-\infty, 1)$. It follows from Lemma \ref{L202} that there exists a constant $C>0$ such that
  \begin{align}
  w(x,t) \leq C (4\pi(1-t))^{-\frac{n}{2}} e^{-f(x,t)}, \quad \forall \; x \in M, \; t \in (-\infty, 0].
  \label{eqn:PK08_1}
  \end{align}
  By the diffeomorphism invariance of the $\overline{\mathcal{W}}$-functional, it is easy to see that
  \begin{align}
    \overline{\mathcal{W}}\left(g(t), \sqrt{w(\cdot, t)}, \eta_0-t \right)
    =\overline{\mathcal{W}}\left( (1-t) (\psi^{t})^{*} g, \sqrt{w(\cdot, t)}, \eta_0-t \right)
    =\overline{\mathcal{W}}\left(g, u(\cdot, t), \theta(t) \right)
    \label{eqn:PK08_8}  
  \end{align}
  where we have used the notation
\begin{align}
  &u(\cdot, t) \coloneqq (1-t)^{\frac{n}{4}}\sqrt{((\psi^{t})^{-1})^{*} w(\cdot, t)},   \label{eqn:PK08_6}\\
  &\theta(t) \coloneqq \frac{\eta_0-t}{1-t}.   \label{eqn:PK08_7}
\end{align}
Notice that $\int u^2 dV \equiv 1$ according to our construction.  It follows from definition and direct calculations that
\begin{align}
  &\quad \overline{\mathcal{W}}\left(g, u(\cdot, t), \theta(t) \right) \notag\\
  &=\int \left\{ \theta\left( 4|\nabla u|^2 + Ru^2 \right) -u^2 \log u^2 \right\} dV -n-\frac{n}{2} \log (4\pi \theta) \notag \\
  &= \theta \left\{ \int \left\{ \left( 4|\nabla u|^2 + Ru^2 \right) -u^2 \log u^2 \right\} dV -n-\frac{n}{2} \log (4\pi) \right\} \notag\\
  &\quad +(\theta-1) \left\{ \int u^2 \log u^2 dV + n+\frac{n}{2} \log (4\pi) \right\} -\frac{n}{2} \log \theta \notag\\
  &\geq \theta \boldsymbol{\mu}(g,1) + (\theta-1) \left\{ \int u^2 \log u^2 dV +n +\frac{n}{2} \log (4\pi) \right\}
  -\frac{n}{2} \log \theta.
  \label{eqn:PK08_3}  
\end{align}
By (\ref{eqn:PK08_6}), the inequality (\ref{eqn:PK08_1}) can be understood as
\begin{align*}
  u^2(x,t) \leq C e^{-f(x,0)}
\end{align*}
for some constant $C$ indepenent of $t$.  Consequently, as $f \geq 0$, we obtain
\begin{align*}
  \int u^2 \log u^2 dV \leq \int \left\{ -f+\log C \right\} u^2 dV \leq \log C - \int f \cdot u^2 dV \leq \log C.
\end{align*}
Note that $\theta(t)<1$ when $t<0$.  Plugging the above inequality into (\ref{eqn:PK08_3}), and noting that
\begin{align*}
    \overline{\mathcal{W}}\left(g(0), \sqrt{w(\cdot, 0)}, \eta_0 \right)  \geq \overline{\mathcal{W}}\left(g(t), \sqrt{w(\cdot, t)}, \eta_0-t \right),  \quad \forall \; t \in (-\infty, 0), 
\end{align*}
we can use (\ref{eqn:PK08_8}) to obtain 
\begin{align*}
 \overline{\mathcal{W}}\left(g(0), \sqrt{w(\cdot, 0)}, \eta_0 \right) 
 \geq  \theta \boldsymbol{\mu}(g,1) + (\theta-1)\left\{ \log C +n+\frac{n}{2} \log (4\pi) \right\}
  -\frac{n}{2} \log \theta. 
\end{align*}
From (\ref{eqn:PK08_7}), it is clear that $\displaystyle \lim_{t \to -\infty} \theta(t)=1$. 
On the right hand side of the above inequlaity, letting $t \to -\infty$, we arrive at
\begin{align*}
  \overline{\mathcal{W}}\left(g(0), \sqrt{w(\cdot, 0)}, \eta_0 \right)  \geq \boldsymbol{\mu}(g, 1).
\end{align*}
Since $w(\cdot, 0)$ could be arbitrary smooth nonnegative function satisfying the normalization condition, and $g=g(0)$, in light of (\ref{eqn:PK08_5}), it is clear that 
(\ref{eqn:PK08_4}) follows from the above inequality. 
\end{proof}

\begin{lemma}  \label{lma:PK10_4}
  For each $\tau \in (1, \infty)$, we have
  \begin{align}
    \boldsymbol{\mu}(g, \tau) \geq \boldsymbol{\mu}=\boldsymbol{\mu}(g,1).   \label{eqn:PK10_1} 
  \end{align}
\end{lemma}

\begin{proof}
For any $u \in W_{*}^{1,2}$ and $\tau >1$,
\begin{align*} 
\overline{\mathcal W}(g,u,\tau)=&\int \left\{\tau(4|\nabla u|^2+Ru^2)-u^2\log u^2 \right\} dV-n-\frac{n}{2}\log(4\pi\tau)  \\
\ge&\int \left\{ (4|\nabla u|^2+Ru^2)-u^2\log u^2 \right\} dV-n-\frac{n}{2}\log(4\pi\tau)  \\
\ge&\,\boldsymbol{\mu}(g, 1)-\frac{n}{2}\log\tau=\boldsymbol{\mu} -\frac{n}{2}\log\tau. 
\end{align*}
By the arbitrary choice of $u \in W_{*}^{1,2}$, it follows that
\begin{align*} 
\boldsymbol{\mu}(g, \tau)\ge \boldsymbol{\mu}-\frac{n}{2}\log\tau.
\end{align*}
Let $\tau \to 1^{+}$, we obtain that
\begin{align*}
  \liminf_{\tau \to 1^{+}} \boldsymbol{\mu}(g, \tau) \geq \boldsymbol{\mu}.   
\end{align*}
By Corollary~\ref{c:mono}, we know that $\boldsymbol{\mu}(g, \tau)$ is an  increasing function of $\tau$ for $\tau \in (1, \infty)$. 
Then it is clear that (\ref{eqn:PK10_1}) follows directly from the above inequality. 
\end{proof}

\begin{proof}[Proof of Proposition~\ref{prn:PK10_2}:]
  It follows from the combination of Lemma~\ref{lma:PK10_3} and Lemma~\ref{lma:PK10_4}.   
\end{proof}

\begin{lemma}\label{lma:PK10_1}
  Suppose $(M, g)$ is a complete Riemannian manifold with Sobolev constant $C_{RS}$. Namely, for each smooth function $u$ with compact support, we have
  \begin{align}
    \left( \int u^{\frac{2n}{n-2}} dV \right)^{\frac{n-2}{n}} \leq C_{RS} \int \left\{ 4|\nabla u|^2 + Ru^2 \right\} dV.
    \label{eqn:PK10_4}  
  \end{align}
 Then for each positive $\tau$,  the following estimates hold for any $u \in W_{*}^{1,2}$,
\begin{align}
  e^{-\frac{2E}{n}} \leq \tau \int \left\{ 4|\nabla u|^2 + Ru^2 \right\} dV \leq \max \left\{n^2,  2E \right\},     \label{eqn:PK10_5}     
\end{align}
where
\begin{align}
  E =  \overline{\mathcal W}(g,u,\tau) +  \frac{n}{2} \log (4\pi e^2  C_{RS}).   \label{eqn:PK10_8}
\end{align}
\end{lemma}

\begin{proof}
  By Jensen's inequality, we know that
  \begin{align*}
    \int u^2\log u^2 \,dV =\frac{n-2}{2}\int u^2\log u^{\frac{4}{n-2}} \,dV \le \frac{n-2}{2}\log\left(\int u^{\frac{2n}{n-2}} \,dV \right).
  \end{align*}
  Plugging the Sobolev inequality (\ref{eqn:PK10_4}) into the above inequality yields that 
  \begin{align}
    \int u^2\log u^2 \,dV  \leq \frac{n}{2}\log C_{RS} + \frac{n}{2} \log \int \left\{ 4|\nabla u|^2 + Ru^2 \right\} dV.
    \label{eqn:PK10_3}  
  \end{align}
  It follows that
  \begin{align*}
    \overline{\mathcal W}(g,u,\tau) & \geq \int \left\{ \tau(4|\nabla u|^2+Ru^2)-u^2\log u^2 \right\} dV-n-\frac{n}{2}\log(4\pi\tau)\\
    &\geq \int \tau \left\{ 4|\nabla u|^2+Ru^2 \right\} dV -\frac{n}{2} \log \int \tau \left\{ 4|\nabla u|^2 + Ru^2 \right\} dV
    -n-\frac{n}{2} \log (4\pi C_{RS}).
  \end{align*}
  Let $x=\int \left( 4|\nabla u|^2+Ru^2 \right) dV$. The above inequality can be rewritten as 
  \begin{align}
    \tau x -\frac{n}{2} \log (\tau x) \leq  \overline{\mathcal W}(g,u,\tau) +n+\frac{n}{2} \log (4\pi C_{RS})=E.  \label{eqn:PK10_6} 
  \end{align}
  Since $\tau x>0$, it follows from (\ref{eqn:PK10_6}) that
  \begin{align}
    \tau x \geq e^{-\frac{2}{n}E}.   \label{eqn:PK10_9}
  \end{align}
  On the other hand, it is clear that
  \begin{align}
    s-\frac{n}{2} \log s \geq \frac{s}{2}, \quad \textrm{on} \; [n^2, \infty).  \label{eqn:PK10_7} 
  \end{align}
  Suppose $\tau x \geq n^2$, then the combination of (\ref{eqn:PK10_6}) and (\ref{eqn:PK10_7}) implies that $\tau x \leq 2E$.
  Consequently, we always have
  \begin{align}
    \tau x \leq \max \left\{ n^2, 2E \right\}.  \label{eqn:PK10_10}
  \end{align}
  Clearly, (\ref{eqn:PK10_5}) follows from the combination of (\ref{eqn:PK10_9}) and (\ref{eqn:PK10_10}). 
\end{proof}

\begin{cor}[Sobolev inequality]
\label{cor234}
Let $\Big\{(M^n,g(t)),\, t\in (-\infty,1)\Big\}$ be the Ricci flow solution of a Ricci shrinker $(M^n,p,g,f)$, there exists a constant $C=C(n)$ such that at any time $t<1$,
\begin{align}
  \left(\int u^{\frac{2n}{n-2}}\,dV\right)^{\frac{n-2}{n}} \le Ce^{-\frac{2 \boldsymbol{\mu}}{n}} \int \left\{ 4|\nabla u|^2+Ru^2  \right\} dV
\label{E234}
\end{align}
for any smooth function $u$ with compact support.
\end{cor}

\begin{proof}
We consider the Schr\"odinger operator $H=-2\Delta +\frac{R}{2}$ and the quadratic forms $Q(u)\coloneqq \int (Hu)u\,dV$ with its corresponding Markov semigroup $\{e^{-Hs},\,s\ge 0\}$. Since $\boldsymbol{\mu}(g(t),\tau) =\boldsymbol{\mu}(g,\frac{\tau}{1-t})\ge \boldsymbol{\mu}$, we have
\begin{align*} 
\int u^2 \log u\,dV \le \tau Q(u)+\beta(\tau)
\end{align*}
for any $\int u^2\,dV=1$, where $\beta(\tau)=-\frac{n}{2}-\frac{n}{4}\log(4\pi \tau)-\boldsymbol{\mu}$. Then it follows from \cite[Corollary $2.2.8$]{Dav89} that  for any $s>0$,
\begin{align} \label{eq:sob1}
\|e^{-Hs}\|_{\infty,2} \le e^{M(s)}\le Cs^{-\frac{n}{4}}e^{-\frac{\boldsymbol{\mu}}{2}},
\end{align}
where $M(s)\coloneqq \frac{1}{s}\int_0^s \beta(\tau)\,d\tau$. Now we use the same argument as in \cite[Theorem $2.4.2$]{Dav89} to derive the Sobolev inequaltiy. It follows from \eqref{eq:sob1} that for any $u \in L^2$,
\begin{align} \label{eq:sob2}
\|e^{-Hs}u\|_{\infty} \le Cs^{-\frac{n}{4}}e^{-\frac{\boldsymbol{\mu}}{2}}\|u\|_2.
\end{align}
Since $e^{-Hs}$ is self-adjoint, by taking the conjugation of \eqref{eq:sob2} we obtain
\begin{align} \label{eq:sob3}
\|e^{-Hs}u\|_2 \le Cs^{-\frac{n}{4}}e^{-\frac{\boldsymbol{\mu}}{2}}\|u\|_1.
\end{align}
Therefore, for any $s>0$,
\begin{align} \label{eq:sob4}
\|e^{-Hs}u\|_{\infty} \le Cs^{-\frac{n}{4}}e^{-\frac{\boldsymbol{\mu}}{2}}\|e^{-\frac{Hs}{2}}u\|_2 \le Cs^{-\frac{n}{2}}e^{-\boldsymbol{\mu}}\|u\|_1.
\end{align}

Combining \eqref{eq:sob4} with the fact that $e^{-Hs}$ is a contraction on $L^{\infty}$, it follows from the Riesz-Thorin interpolation that for any $q \in [1,\infty]$.
\begin{align} \label{eq:sob5}
\|e^{-Hs}u\|_{\infty} \le Cs^{-\frac{n}{2q}}e^{-\frac{\boldsymbol{\mu}}{q}}\|u\|_q.
\end{align}

We now write
\begin{align*} 
H^{-\frac{1}{2}}u=a+b
\end{align*}
where 
\begin{align*} 
a=&\Gamma^{-1}(1/2) \int_0^T s^{-\frac{1}{2}}e^{-Hs}u\,ds,\\
b=&\Gamma^{-1}(1/2) \int_T^{\infty} s^{-\frac{1}{2}}e^{-Hs}u\,ds.
\end{align*}
It follows from \eqref{eq:sob5} that
\begin{align*} 
\|b\|_{\infty} \le C\Gamma^{-1}(1/2) \int_T^{\infty} s^{-\frac{1}{2}-\frac{n}{2q}}e^{-\frac{\boldsymbol{\mu}}{q}}\|u\|_q\,ds=ce^{-\frac{\boldsymbol{\mu}}{q}}\|u\|_qT^{\frac{1}{2}-\frac{n}{2q}}
\end{align*}
for some constant $c=c(n)$. Given $\lambda>0$, we define $T>0$ by $\frac{\lambda}{2}=ce^{-\frac{\boldsymbol{\mu}}{q}}\|u\|_qT^{\frac{1}{2}-\frac{n}{2q}}$. It is clear that
\begin{align*} 
|\{x:\,|H^{-\frac{1}{2}}u(x)|\ge \lambda\}|\le |\{x:\,|a(x)|\ge \lambda/2\}| \le 2^q\lambda^{-q}\|a\|_{q}^q \le C\lambda^{-q}T^{\frac{q}{2}}\|u\|_q^q,
\end{align*}
since $e^{-Hs}$ is a contraction on $L^q$. For any $1<q<n$, we set $\frac{1}{r}=\frac{1}{q}-\frac{1}{n}$, then it follows from our choice of $\lambda$ that
\begin{align*} 
|\{x:\,|H^{-\frac{1}{2}}u(x)|\ge \lambda\}| \le Ce^{-\frac{\boldsymbol{\mu} q}{n-q}}\lambda^{-r} \|u\|_q^r.
\end{align*}
In other words,
\begin{align} \label{eq:sob6}
\|H^{-\frac{1}{2}}u\|_{r,w} \le Ce^{-\frac{\boldsymbol{\mu} q}{r(n-q)}}\|u\|_q
\end{align}
where $\|\cdot\|_{r,w}$ denotes the weak $L^r$ space. Therefore, it follows from the Marcinkiewicz interpolation theorem that
\begin{align} \label{eq:sob7}
\|H^{-\frac{1}{2}}u\|_{p} \le Ce^{-\frac{2\boldsymbol{\mu} }{p(n-2)}}\|u\|_2=Ce^{-\frac{\boldsymbol{\mu} }{n}}\|u\|_2,
\end{align}
where $\frac{1}{p}=\frac{1}{2}-\frac{1}{n}$. Therefore, \eqref{E234} is a direct consequence.
\end{proof}

\begin{rem}
It follows from the above corollary that the Yamabe invariant of $(M^n,g,f) $
\begin{align}
Y([g]) \coloneqq \inf_{u \in C_0^{\infty}(M)}\frac{\int \frac{4(n-1)}{n-2}|\nabla u|^2+Ru^2 \,dV}{\left(\int u^{\frac{2n}{n-2}}\,dV\right)^{\frac{n-2}{n}}} >0.
\end{align}
Here $Y$ depends only on the conformal class of $g$. Hence it implies some connections between a Ricci shrinker and its conformal class. Note that it is shown in \cite{ZhangZL} that each Ricci shrinker has a conformal metric such that its Ricci curvature has local bound depending only on the dimension. This fact plays a key role in \cite{LLW18}.
\end{rem}

\begin{proposition}\label{prn:PK10_3}
On a Ricci shrinker $(M^n,g,f)$, the functional  $\boldsymbol{\mu}(g,\tau)$ is a continuous function of $\tau \in (0, \infty)$. 
\end{proposition}

\begin{proof}
Fix $\tau_0 \in (0, \infty)$.  We need to show both the upper semi-continuity and the lower semi-continuity as $\tau_0$.

The upper-semicontinuity is more or less standard. 
Fix $u \in W_{*}^{1,2}$, we have 
\begin{align*} 
\limsup_{\tau \to \tau_0}\boldsymbol{\mu}(g,\tau)\le&\limsup_{\tau \to \tau_0}\overline{\mathcal W}(g,u,\tau) \\
=&\limsup_{\tau \to \tau_0}\int \tau(4|\nabla u|^2+Ru^2)-u^2\log u^2 \,dV-n-\frac{n}{2}\log(4\pi\tau)\\
=&\int \tau_0(4|\nabla u|^2+Ru^2)-u^2\log u^2 \,dV-n-\frac{n}{2}\log(4\pi\tau_0)\\
=&\overline{\mathcal W}(g,u,\tau_0).
\end{align*}
By taking the infimum among all qualified $u$'s, we have
\begin{align}
  \limsup_{\tau \to \tau_0}\boldsymbol{\mu}(g,\tau)\le \boldsymbol{\mu}(g,\tau_0). 
  \label{eqn:PK10_11}
\end{align}
Hence $\boldsymbol{\mu}(g,\tau)$ is upper semicontinuous. 

The lower semicontinuity relies on the estimate (\ref{eqn:PK10_5}) in Lemma~\ref{lma:PK10_1}.  
Actually, for arbitrary $u \in W_{*}^{1,2}$ satisfying the normalization condition,  direct calculation shows that
\begin{align}
  \overline{\mathcal W}(g,u,\tau)&=\overline{\mathcal W}(g,u,\tau_0) + (\tau-\tau_0)\int \left\{ 4|\nabla u|^2+Ru^2 \right\} dV-\frac{n}{2}\log \lc\frac{\tau}{\tau_0} \rc \notag\\
  &\geq \boldsymbol{\mu}(g, \tau_0) -|\tau-\tau_0|\int \left\{ 4|\nabla u|^2+Ru^2 \right\} dV-\frac{n}{2}\log \lc\frac{\tau}{\tau_0} \rc. \label{eqn:PK10_12}
\end{align}
For any $\tau_i \to \tau_0$, we choose $u_i \in W_{*}^{1,2}$ such that 
\begin{align}
  \overline{\mathcal W}(g,u_i,\tau_i)-\boldsymbol{\mu}(g,\tau_i)<i^{-1}.
  \label{eqn:PK10_14}
\end{align}
Together with (\ref{eqn:PK10_11}), this implies that
\begin{align}
  \limsup_{i \to \infty} \overline{\mathcal W}(g,u_i,\tau_i) =\limsup_{i \to \infty} \boldsymbol{\mu}(g,\tau_i) \leq \boldsymbol{\mu}(g, \tau_0). \label{eqn:PK10_13} 
\end{align}
By Corollary~\ref{cor234}, the Sobolev constant on each Ricci shrinker is finite. 
It follows from (\ref{eqn:PK10_5}) and (\ref{eqn:PK10_8}) that $\int (4|\nabla u_i|^2+Ru_i^2)\,dV$ is uniformly bounded.
In (\ref{eqn:PK10_12}), replacing $u$ by $u_i$ and letting $i \to \infty$, we obtain
\begin{align*}
  \liminf_{i \to \infty} \overline{\mathcal W}(g,u_i,\tau_i) \geq \boldsymbol{\mu}(g, \tau_0).
\end{align*}
Combining the above inequality with (\ref{eqn:PK10_14}), we obtain that 
\begin{align*}
  \liminf_{i \to \infty} \boldsymbol{\mu}(g,\tau_i) \geq \boldsymbol{\mu}(g, \tau_0),
\end{align*}
which is the lower semi-continuity at $\tau_0$.   The continuity of $\boldsymbol{\mu}(g,\tau)$ with respect to $\tau$ at $\tau_0$ follows from the combination of the 
above inequality and (\ref{eqn:PK10_11}). 
\end{proof}

\section{Optimal logarithmic Sobolev constant---Part II}

We first prove the log-Sobolev inequality for the conjugate heat kernel following \cite{HN13}. The proof in \cite{HN13} is for spacetime with bounded geometry. Since we do not impose any curvature restriction here, more should be done due to the integration by parts.

\begin{thm}
For any Ricci shrinker $(M^n,g,f)$ with its heat kernel $H(x,t,y,s)$,
\begin{align*} 
\int \rho\log \rho \,dv_s-\left(\int \rho \,dv_s\right)\log{\left(\int \rho\, dv_s\right)} \le  (t-s) \int \frac{|\nabla \rho|^2}{\rho}\,dv_s.
\end{align*}
\label{T305}
Here $dv_s(y)=H(x,t,y,s)dV_s(y)$ for any $x\in M$ and $s<t<1$ and $\rho$ is any nonnegative function such that $\sqrt \rho \in W^{1,2}(M,v_s)$ and $\int d^2(p,\cdot)\rho \,dv_s <\infty$. If the equality holds, then either $\rho$ is a constant or $(M^n,g)$ splits off a $\R$ factor.
\end{thm}

\begin{proof}
By a similar approximation process as in Section $4$, we only need to prove the inequality for any $\rho$ such that $\sqrt \rho$ is a compactly supported smooth function. Without loss of generality, we assume $s=0$ and fix $T>0$ and $q \in M$. Moreover, we set $w(x,t)=H(q,T,x,t)$, $dv=w(y,0)\,dV_0(y)$ and $\rho(x,t)$ is the bounded solution of the heat equation starting from $\rho(x)$. In the proof, we denote $\rho(x,t)$ by $\rho$ with the time $t$ implicitly understood. We also assume that $\rho$ is uniformly bounded by $1$ on $M \times [0,T]$.

It is clear from the definition of $w$ that
\begin{align*} 
\lim_{t \nearrow T} \int \rho (\log \rho) w\phi^r \,dV_t=\rho(q,T) \log \rho(q,T)=\left(\int  \rho\,dv\right)\log{\left(\int \rho\, dv\right)}
\end{align*}
and
\begin{align*} 
\int \rho (\log \rho) w\phi^r \,dV_0=\int \rho\log \rho \phi^r \,dv.
\end{align*}
Therefore, we have
\begin{align*} 
\int \rho\log \rho \phi^r \,dv-\left(\int \rho \,dv\right)\log{\left(\int \rho\, dv\right)} = \int_0^T -\partial_t \int \rho (\log \rho) w\phi^r \,dV_t \,dt.
\end{align*}

By direct computations
\begin{align} 
& -\partial_t \int \rho (\log \rho) w\phi^r \,dV_t \notag \\
&=-\int \rho_t(\log \rho+1)w\phi^r+\rho(\log \rho)w_t \phi^r+\rho(\log \rho)w\phi^r_t-R\rho(\log \rho)w\phi^r \,dV_t \notag \\
&=-\int \Delta \rho(\log \rho+1)w\phi^r-\rho(\log \rho)\Delta w \phi^r+\rho(\log \rho)w\phi^r_t \,dV_t  \notag \\
&=\int \frac{|\na \rho|^2}{\rho}w\phi^r+2w\la\na(\rho\log \rho),\na \phi^r \ra-\rho(\log \rho)w \square \phi^r\,dV_t. \label{X601}
\end{align}

Similarly,
\begin{align*}
 \partial_t \int \frac{|\na \rho|^2}{\rho}w\phi^r\,dV_t =\int \square \left(\frac{|\na \rho|^2}{\rho}\right)w\phi^r-2w\la\na\left(\frac{|\na \rho|^2}{\rho}\right),\na \phi^r \ra+\frac{|\na \rho|^2}{\rho}w \square \phi^r\,dV_t.
\end{align*}

Since 
\begin{align}\label{X603}
\square \frac{|\nabla \rho|^2}{\rho}=-\frac{2}{\rho}\left|\text{Hess}\,\rho-\frac{d\rho \otimes d\rho}{\rho}\right|^2 
\end{align}
we have for any $s \in [0,T]$,
\begin{align}
&\int \frac{|\na \rho|^2}{\rho}w\phi^r\,dV_s \notag\\
 =& \int \frac{|\na \rho|^2}{\rho} \phi^r \,dv-\int_0^s\int 2w\la\na\left(\frac{|\na \rho|^2}{\rho}\right),\na \phi^r \ra\,dV_t \,dt \notag \\
&+\int_0^s\int \frac{|\na \rho|^2}{\rho}w \square \phi^r\,dV_t \,dt- \int_0^s\int \frac{2}{\rho}\left|\text{Hess}\,\rho-\frac{d\rho \otimes d\rho}{\rho}\right|^2w\phi^r \,dV_t \,dt.\label{X602} 
\end{align}

With \eqref{X601} and \eqref{X602}, we have proved so far that if $r$ is sufficiently large,
\begin{align*}
&\int \rho\log \rho \,dv-\left(\int \rho \,dv\right)\log{\left(\int \rho\, dv\right)}- T\int \frac{|\na \rho|^2}{\rho} \,dv \\
 =&\int_0^T\int2w\la\na(\rho\log \rho),\na \phi^r \ra \,dV_t\,dt-\int_0^T\int \rho(\log \rho)w \square \phi^r\,dV_t\,dt \\
&+\int_0^T\int_0^s\int\frac{|\na \rho|^2}{\rho}w \square \phi^r\,dV_t dtds-\int_0^T\int_0^s\int 2w\la\na\left(\frac{|\na \rho|^2}{\rho}\right),\na \phi^r \ra\, dV_t dtds \\
&- \int_0^T\int_0^s\int \frac{2}{\rho}\left|\text{Hess}\,\rho-\frac{d\rho \otimes d\rho}{\rho}\right|^2w\phi^r \,dV_t \,dt\,ds\\
=&I+II+III+IV+V,
\end{align*}
where
\begin{align*}
I=& \int_0^T\int2w\la\na(\rho\log \rho),\na \phi^r \ra \,dV_t\,dt,\\
II=& \int_0^T\int-\rho(\log \rho)w \square \phi^r\,dV_t\,dt, \\
III=& \int_0^T\int_0^s\int\frac{|\na \rho|^2}{\rho}w \square \phi^r\,dV_t dtds, \\
IV=&\int_0^T\int_0^s\int -2w\la\na\left(\frac{|\na \rho|^2}{\rho}\right),\na \phi^r \ra \,dV_t dtds, \\
V=&\int_0^T\int_0^s\int -\frac{2}{\rho}\left|\text{Hess}\,\rho-\frac{d\rho \otimes d\rho}{\rho}\right|^2w\phi^r \,dV_t \,dtds.
\end{align*}

It remains to show that when $r \to \infty$ the sum is less or equal to 0.

We first notice that as $\rho$ is smooth with compact support, by using \eqref{X603} and the maximum principle,
\begin{align*}
\frac{|\nabla \rho|^2}{\rho} \le C.
\end{align*}
Here the assumption in Theorem \ref{T101} can be checked as \eqref{E213a}.

Now we have for the first term $I$
\begin{align*}
\lim_{r \to \infty}|I| \le \lim_{r \to \infty} 2\int_0^T\int w|\nabla \rho|(1+|\log \rho|)|\na \phi^r| \,dV_t\,dt \le  \lim_{r \to \infty} Cr^{-1/2}=0.
\end{align*}

For the second term $II$,
\begin{align*}
\lim_{r \to \infty}|II| \le \lim_{r \to \infty} \int_0^T\int w\rho|\log \rho| |\square \phi^r| \,dV_t\,dt \le  \lim_{r \to \infty} Cr^{-1}=0.
\end{align*}

Similarly for the third term $III$,
\begin{align*}
\lim_{r \to \infty}|III| \le \lim_{r \to \infty} \int_0^T\int_0^s\int\frac{|\na \rho|^2}{\rho}w |\square \phi^r|\,dV_t dtds \le  \lim_{r \to \infty} Cr^{-1}=0.
\end{align*}

The fourth term $IV$ is more involved, by computation we have
\begin{align}\label{X604}
\na\frac{|\na \rho|^2}{\rho}=2 \frac{\la \text{Hess}\,\rho,\na \rho \ra}{\rho} -\frac{|\na \rho|^2}{\rho^2}\na \rho=2 \frac{\la \text{Hess}\,\rho-\rho^{-1}d\rho\otimes d\rho,\na \rho \ra}{\rho}+\frac{|\na \rho|^2}{\rho^2}\na \rho. 
\end{align}

From \eqref{X604}, we have
\begin{align*}
|IV| \le& \int_0^T\int_0^s\int 2w|\na \phi^r| \left(2\frac{|h||\na \rho|}{\rho}+\frac{|\na \rho|^3}{\rho^2} \right)  \,dV_t dtds \\
\le& \int_0^T\int_0^s\int 2\ep\frac{|h|^2}{\rho}w\phi^r+2\ep^{-1}\frac{|\na \rho|^2}{\rho}\frac{|\na \phi^r|^2}{\phi^r}w+2w|\na \phi^r| \frac{|\na \rho|^3}{\rho^2} \,dV_t dtds \\
\le&-\ep V+C\ep^{-1}r^{-1}+2r^{-1/2}\int_0^T\int_0^s\int w\frac{|\na \rho|^3}{\rho^2} \,dV_t dtds
\end{align*}
where we denote $\text{Hess}\,\rho-\rho^{-1}d\rho\otimes d\rho$ by $h$ and $\ep \in (0,1)$.

To deal with the last integral, we notice from Lemma \ref{L301} that
\begin{align*}
\frac{|\na \rho|^3}{\rho^2} =\frac{|\nabla \rho|^{3/2}}{\rho^{3/2}} \frac{|\nabla \rho|^{3/2}}{\rho^{3/4}} \rho^{1/4} \le \frac{C}{t^{3/4}}\left( \rho^{1/6} \log{\frac{M}{\rho}} \right)^{3/2} \le \frac{C}{t^{3/4}}
\end{align*}
and hence
\begin{align*}
\int_0^s\int w\frac{|\na \rho|^3}{\rho^2} \,dV_t dt \le C\int_0^s t^{-3/4} \,dt \le C.
\end{align*}

Therefore, $\lim_{r \to \infty}V$ is finite and $\lim_{r \to \infty}|IV| \le -\ep\lc\lim_{r \to  \infty}V\rc$. By taking $\ep \to 0$, we obtain that $\lim_{r \to \infty}|IV|=0$ and hence
\begin{align}\label{eq:logsobeq}
&\int \rho\log \rho \,dv-\left(\int \rho \,dv\right)\log{\left(\int \rho\, dv\right)}- T\int \frac{|\na \rho|^2}{\rho} \,dv \\
 =&-\int_0^T\int_0^s\int \frac{2}{\rho}\left|\text{Hess}\,\rho-\frac{d\rho \otimes d\rho}{\rho}\right|^2w \,dV_t \,dt\,ds \le 0.
\end{align}

If the equality holds and $\rho$ is not a constant, it follows from \eqref{eq:logsobeq} that 
\begin{align*}
\text{Hess}(\log \rho)=\frac{1}{\rho}\lc\text{Hess}\,\rho-\frac{d\rho \otimes d\rho}{\rho} \rc=0.
\end{align*}
Therefore, $(M^n,g)$ splits off a $\R$ factor.
\end{proof}

For fixed $x,t$ and $s$, Theorem \ref{T305} implies that the probability measure $dv_s(y)=H(x,t,y,s)dV_s(y)$ satisfies the log-Sobolev inequality with the constant $\frac{1}{2(t-s)}$. It is a standard fact that log-Sobolev condition implies the Talagrand's inequality and equivalently, the Gaussian concentration, see \cite[Theorem $22.17$, Theorem $22.10$]{Vil08}. In particular we have the following theorem, see also \cite[Theorem $1.13$]{HN13}.

\begin{thm}[Gaussian concentration]
For any Ricci shrinker $(M^n,g,f)$ with its heat kernel $H(x,t,y,s)$ and reference measure $dv_s(y)=H(x,t,y,s)dV_s(y)$ and any $\sigma>0$
\begin{align*} 
v_s(A)v_s^{\frac{1}{\sigma}}(B)\le \exp{\lc-\frac{r^2}{4(1+\sigma)(t-s)}\rc}
\end{align*}
\label{T306}
where $A$ and $B$ are two sets on $M$ such that $d_s(A,B) \ge r>0$.
\end{thm}

\begin{proof}
From Theorem \eqref{T305}, we have for any probability measure $\rho dv_s$,
\begin{align}\label{X605} 
\int \rho \log \rho \,dv_s \le (t-s)\int \frac{|\na \rho|^2}{\rho} \,dv_s.
\end{align}
By a further approximation, we can assume \eqref{X605} holds for any locally Lipschitz $\rho$. Now it follows from \cite[Theorem $22.17$]{Vil08} that $dv_s$ satisfies the $T_2$ Talagrand inequality, that is,
\begin{align} \label{X606}  
W_2(\eta, v_s) \le \sqrt{4(t-s)}\left(\int  \rho \log  \rho \,dv_s\right)^{1/2}
\end{align}
for any measure $\eta \in P_2(M)$, where $W_2$ is the Wasserstein distance of second order. For any two sets $A$ and $B$ on $M$ such that $d_s(A,B) \ge r>0$. We set $\eta=\frac{1_A}{v_s(A)}v_s$ and $v=\frac{1_B}{v_s(B)}v_s$. Then on the one hand,
\begin{align*} 
W_2(\eta,v) \le& W_2(\eta,v_s)+W_2(v,v_s) \\
\le& \sqrt{4(t-s)}\left(\left(\int  \frac{1_A}{v_s(A)} \log \frac{1_A}{v_s(A)} \,dv_s\right)^{1/2}+\left(\int  \frac{1_B}{v_s(B)} \log \frac{1_B}{v_s(B)}\,dv_s\right)^{1/2}\right) \\
= &\sqrt{4(t-s)}\left( \lc\log{\frac{1}{v_s(A)}}\rc^{1/2}+\lc\log{\frac{1}{v_s(B)}}\rc^{1/2}\right)
\end{align*}
and hence
\begin{align*} 
W^2_2(\eta,v) \le&  4(t-s)\left( \lc\log{\frac{1}{v_s(A)}}\rc^{1/2}+\lc\log{\frac{1}{v_s(B)}}\rc^{1/2}\right)^2 \\
\le&  4(t-s)\left((1+\sigma)\log{\frac{1}{v_s(A)}}+(1+\sigma^{-1})\log{\frac{1}{v_s(B)}}\right).
\end{align*}
On the other hand, it follows from the definition of $W_2$ that
\begin{align*} 
W^2_2(\eta,v) =\int d^2_s(x,y)\,d\pi(x,y) \ge r^2
\end{align*}
where $\pi$ is the optimal transport between $\eta$ and $v$.

Therefore by computation
\begin{align*}
v_s(A)v_s^{\frac{1}{\sigma}}(B)\le \exp{\lc-\frac{r^2}{4(1+\sigma)(t-s)}\rc}
\end{align*}
\end{proof}

In fact, with the Gaussian concentration, we can prove that $v_s$ has finite square-exponential moment.

\begin{cor}
For any Ricci shrinker $(M^n,g,f)$ with its heat kernel $H(x,t,y,s)$ and reference measure $dv_s(y)=H(x,t,y,s)dV_s(y)$, if $a< \frac{1}{4(t-s)}$, then
\begin{align*} 
\int e^{ad_s^2(p,x)}\,dv_s < \infty.
\end{align*}
\end{cor}

\begin{proof}
We choose a constant $\sigma>0$ such that $a < \frac{1}{4(1+\sigma)(t-s)} $.
It follows from Theorem \ref{T306} that for any integer $k \ge 2$, 
\begin{align*}
v_s(M \backslash B_s(p,k))\le \exp{\lc-\frac{(k-1)^2}{4(1+\sigma)(t-s)}\rc}
\end{align*}
Hence
\begin{align*}
\int e^{ad_s^2(p,x)}\,dv_s \le& C\lc1+\sum_{k=2}^{\infty}\int_{B_s(p,k+1)\backslash B_s(p,k)}e^{ad_s^2(p,x)}\,dv_s\rc \\
& \le C +C\sum_{k=2}^{\infty}(k+1)^n\exp{\lc a(k+1)^2-\frac{(k-1)^2}{4(1+\sigma)(t-s)}\rc}
\end{align*}
where we have used Lemma \ref{L101}. Since $a< \frac{1}{4(1+\sigma)(t-s)}$, it is easy to show that the last sum is finite.
\end{proof}


\section{Heat kernel estimates}

We first prove a pointwise upper bound for the heat kernel $H$. The idea of the proof is from \cite[Chapter $2$]{Dav89}, see also \cite{Zhangqi2}.

\begin{thm}[Ultracontractivity]
For any Ricci shrinker $(M^n,g,f)$,
\begin{align*} 
H(x,t,y,s) \le \frac{e^{-\boldsymbol{\mu}}}{(4\pi(t-s))^{\frac{n}{2}}}.
\end{align*}
\label{T301}
\end{thm}

\begin{proof}
We fix $x \in M$ and two constants $s<T<1$. For notational simplicity, we assume that $\tau=T-t$ and $\partial_{\tau}=-\partial_t$. We also fix a function $p(\tau)=\frac{T-s}{T-s-\tau}$ for $\tau \in [0,T-s)$. For any nonnegative smooth function $h$ such compact support we  define
\begin{align} 
w(y,\tau)=\int H(x,T,y,T-\tau)h(x)\,dV_T(x),
\end{align}
then $\square^* w=0$.

 Now we compute,
\begin{align} 
&\partial_{\tau}||w\phi^r||_{p({\tau})}=\partial_{\tau}\left(\int (w\phi^r)^{p({\tau})}\,dV_{\tau}\right)^{\frac{1}{p({\tau})}} \notag \\
=&-\frac{p'({\tau})}{p({\tau})}||w\phi^r||_{p({\tau})} \log\left(\int  (w\phi^r)^{p({\tau})} \,dV_{\tau}\right) \notag\\
&+\frac{1}{p({\tau})}\left(\int  (w\phi^r)^{p({\tau})}\,dV_{\tau}\right)^{\frac{1}{p({\tau})}-1}\left(\int  (w\phi^r)^{p({\tau})}(\log w\phi^r) p'({\tau}) \,dV \right) \notag \\
&+\frac{1}{p({\tau})}\left(\int  (w\phi^r)^{p({\tau})}\,dV_{\tau}\right)^{\frac{1}{p({\tau})}-1}\left(\int p({\tau})(w\phi^r)^{p({\tau})-1}(w\phi^r)_{\tau}+R(w\phi^r)^{p(\tau)} \,dV \right).
\label{E301}
\end{align}

Since 
\begin{align} 
(w\phi^r)_{\tau}=\Delta w \phi^r-Rw\phi^r+w\phi^r_{\tau}=\Delta(w\phi^r)-Rw\phi^r-(\square_{\tau}\phi^r)w-2\langle \nabla w,\nabla \phi^r\rangle,
\label{E302}
\end{align}
if we multiply both sides of \eqref{E301} by $p^2({\tau})||w\phi^r||^{p({\tau})}_{p({\tau})}$, we have
\begin{align} 
&p^2({\tau})||w\phi^r||^{p({\tau})}_{p({\tau})}\partial_{\tau}||w\phi^r||_{p({\tau})} \notag\\
=&-p‘({\tau})||w\phi^r||^{p({\tau})+1}_{p({\tau})} \log\left(\int  (w\phi^r)^{p({\tau})} \,dV_{\tau}\right) \notag\\
&+p({\tau})p'({\tau})||w\phi^r||_{p({\tau})}\int (w\phi^r)^{p({\tau})}\log (w\phi^r) \,dV_{\tau} \notag \\
&-p^2({\tau})(p({\tau})-1)||w\phi^r||_{p({\tau})}\int (w\phi^r)^{p({\tau})-2}|\nabla(w\phi^r)|^2 \,dV_{\tau} \notag \\
&-(p({\tau})-1) ||w\phi^r||_{p({\tau})}\int R(w\phi^r)^{p({\tau})} \,dV_{\tau}+X.
\label{E303}
\end{align}
where 
\begin{align*} 
X=p^2({\tau})||w\phi^r||_{p({\tau})}\int (w\phi^r)^{p({\tau})-1}\left(-(\square_{\tau}\phi^r)w-2\langle \nabla w,\nabla \phi^r\rangle\right)\,dV_{\tau}. 
\end{align*}

Now we divide both sides of \eqref{E303} by $||w\phi^r||_{p({\tau})}$, then
\begin{align} 
&p^2({\tau})||w\phi^r||^{p({\tau})}_{p({\tau})}\partial_{\tau}\log||w\phi^r||_{p({\tau})}\notag \\
=&-p‘({\tau})||w\phi^r||^{p({\tau})}_{p({\tau})} \log\left(\int  (w\phi^r)^{p({\tau})} \,dV_{\tau}\right) \notag\\
&+p({\tau})p'({\tau})\int (w\phi^r)^{p({\tau})}\log (w\phi^r) \,dV_{\tau} \notag \\
&-4(p({\tau})-1)\int |\nabla(w\phi^r)^{\frac{p({\tau})}{2}}|^2 \,dV_{\tau} \notag \\
&-(p({\tau})-1) \int R(w\phi^r)^{p({\tau})} \,dV_{\tau}+Y.
\label{E304}
\end{align}
where 
\begin{align*} 
Y=p^2({\tau})\int (w\phi^r)^{p({\tau})-1}\left(-(\square_{\tau}\phi^r)w-2\langle \nabla w,\nabla \phi^r\rangle\right)\,dV_{\tau}. 
\end{align*}

We denote $v=(w\phi^r)^{\frac{p({\tau})}{2}}/||(w\phi^r)^{\frac{p({\tau})}{2}}||_2$ so that $||v||_2=1$. Now by direct computations,
\begin{align*} 
v^2\log v^2=p({\tau})v^2\log (w\phi^r)-2v^2\log ||(w\phi^r)^{\frac{p({\tau})}{2}}||_2.
\end{align*}

So \eqref{E304} becomes
\begin{align*} 
p^2({\tau})\partial_{\tau}\log||w\phi^r||_{p({\tau})}=& p'({\tau})\int v^2\log v^2 \,dV_{\tau}-4(p({\tau})-1)\int |\nabla v|^2 \,dV_{\tau}-(p({\tau})-1)\int Rv^2\,dV_{\tau}+Z
\end{align*}
where
\begin{align*} 
Z=\frac{p^2({\tau})}{||w\phi^r||^{p({\tau})}_{p({\tau})}}\int (w\phi^r)^{p({\tau})-1}\left(-(\square_{\tau}\phi^r)w-2\langle \nabla w,\nabla \phi^r\rangle\right)\,dV_{\tau}. 
\end{align*}

Now we obtain
\begin{align} 
p^2({\tau})\partial_{\tau}\log||w\phi^r||_{p({\tau})}=& p'({\tau})\left(\int v^2\log v^2 \,dV_{\tau}-\frac{p({\tau})-1}{p'({\tau})}\int 4|\nabla v|^2+ Rv^2\,dV_{\tau}\right)+Z.
\label{E305}
\end{align}

Since $\frac{p({\tau})-1}{p'({\tau})}=\frac{\tau(T-s-\tau)}{T-s}>0$, we have from \eqref{E305} 
\begin{align} 
p^2({\tau})\partial_{\tau}\log||w\phi^r||_{p({\tau})}\le& p'({\tau})\left(-\boldsymbol{\mu}-n-\frac{n}{2}\log(4\pi(p({\tau})-1)/p'({\tau}))\right)+Z.
\label{E306}
\end{align}

Now we divide both sides by $p^2(\tau)$, we have
\begin{align} 
\partial_{\tau}\log||w\phi^r||_{p({\tau})}\le \frac{p'({\tau})}{p^2(\tau)}\left(-\boldsymbol{\mu}-n-\frac{n}{2}\log(4\pi)-\frac{n}{2}\log\left(\frac{p({\tau})-1)}{p'({\tau})}\right)\right)+U(\tau),
\label{E307}
\end{align}
where
\begin{align*} 
U(\tau)=\frac{1}{||w\phi^r||^{p({\tau})}_{p({\tau})}}\int (w\phi^r)^{p({\tau})-1}\left(-(\square_{\tau}\phi^r)w-2\langle \nabla w,\nabla \phi^r\rangle\right)\,dV_{\tau}. 
\end{align*}

Now we integrate both sides of \eqref{E307} and estimate the two terms of right side separately.

For a number $L < T-s$, we integrate \eqref{E307} from $0$ to $L$ so that
\begin{align} 
&\log ||w\phi^r||_{p(L)}-\log||w\phi^r||_1 \notag \\
\le& \int_0^L \frac{p'({\tau})}{p^2(\tau)}\left(-\boldsymbol{\mu}-n-\frac{n}{2}\log(4\pi)-\frac{n}{2}\log\left(\frac{p({\tau})-1)}{p'({\tau})}\right)\right)\, d\tau+\int _0^LU(\tau)\,d\tau\notag\\
=&I(L)+II(L).
\label{E308}
\end{align}

By direct computations,
\begin{align} 
I(T-s)=& \int_0^{T-s} \frac{p'({\tau})}{p^2(\tau)}\left(-\boldsymbol{\mu}-n-\frac{n}{2}\log(4\pi)-\frac{n}{2}\log\left(\frac{p({\tau})-1)}{p'({\tau})}\right)\right)\, d\tau \notag\\
=&-\frac{n}{2}\log(T-s)-\boldsymbol{\mu}-\frac{n}{2}\log(4\pi)
\label{E309}
\end{align}

Now we consider the term $U(\tau)$.
\begin{align} 
|U(\tau)| \le&\frac{1}{||w\phi^r||^{p({\tau})}_{p({\tau})}}\int w^{p({\tau})}|\square_{\tau}\phi^r|+2|\nabla w||\nabla \phi^r|\,dV_{\tau}. 
\label{E310}
\end{align}

Since we construct $w$ through a smooth function with compact support, 
\begin{align*} 
w \le Ce^{-f}
\end{align*}
for a constant $C$ uniformly on $M \times [T-s-L,T-s]$. On the other hand, by Lemma \ref{L201c} $\sqrt w \in W_{*}^{1,2}$ for any $\tau>0$, in particular
\begin{align*} 
\int \frac{|\nabla w|^2}{w} \,dV < \infty.
\end{align*}

Now the second term in \eqref{E310} can be estimated as
\begin{align*} 
\int |\nabla w||\nabla \phi^r| \, dV=&\int \frac{|\nabla w|}{\sqrt w}|\nabla \phi^r|\sqrt w \, dV 
\le  \left(\int \frac{|\nabla w|^2}{ w} \, dV \right)^{\frac{1}{2}} \left(\int |\nabla \phi^r|^2 w \, dV \right)^{\frac{1}{2}}
\end{align*}

For any fixed $L$, it is easy to say $U(\tau)$ is uniformly bounded for any $\tau \in [T-s-L,T-s]$ and $r \ge 1$. By taking $r \to \infty$ in \eqref{E308}, from the dominated convergence theorem, 
\begin{align*} 
\log ||w||_{p(L)}-\log||w||_1 \le I(L).
\end{align*}

Now by taking $L \to T-s$ we have
\begin{align*} 
\log ||w||_{\infty}-\log||w||_1 \le -\frac{n}{2}\log(T-s)-\boldsymbol{\mu}-\frac{n}{2}\log(4\pi).
\end{align*}

Therefore,
\begin{align*} 
\int H(x,T,y,s)h(x)\,dV_T(x) &\le \frac{e^{-\boldsymbol{\mu}}}{(4\pi(T-s))^{n/2}} \iint H(x,T,y,s)h(x)\,dV_T(x)\,dV_s(y) \\
&=\frac{e^{-\boldsymbol{\mu}}}{(4\pi(T-s))^{n/2}} \int h(x)\,dV_T(x).
\end{align*}

Since $h(x)$ can be any smooth function with compact support, we derive that
\begin{align*} 
H(x,T,y,s) \le \frac{e^{-\boldsymbol{\mu}}}{(4\pi(T-s))^{n/2}}.
\end{align*}
\end{proof}


Now we derive the lower bound of $H$. Recall that the reduced distance between $(x,t)$ and $(y,s)$ are defined as 
\begin{align} 
l_{(x,t)}(y,s)=\frac{1}{2\sqrt{t-s}} \inf \left\{\mathcal L(\gamma): \, \gamma:[s,t] \to M \, \text{between $(x,t)$ and $(y,s)$}\right\},
\end{align}
where 
\begin{align} 
\mathcal L(\gamma)=\int_s^t \sqrt{t-z} \left(|\gamma'(z)|_{z}^2+R(\gamma(z),z)\right)\,dz.
\end{align}

Now we have the following important estimate, see Corollary $9.5$ of \cite{Pe1}. The proof is motivated by \cite[Proposition $1$]{CLY84}.

\begin{thm}
For any Ricci shrinker $(M^n,g,f)$,
\begin{align*} 
H(x,t,y,s) \ge \frac{e^{-l_{(x,t)}(y,s)}}{(4\pi(t-s))^{\frac{n}{2}}}.
\end{align*}
\label{T302}
\end{thm}

\begin{proof}
We set 
\begin{align} 
L(x,t,y,s)= \frac{e^{-l_{(x,t)}(y,s)}}{(4\pi(t-s))^{\frac{n}{2}}}.
\end{align}

It follows from the definition of $l_{(x,t)}(y,s)$, see \cite{Pe1} and \cite{Ye04}, that 
\begin{align} 
-\partial_s L(x,t,y,s) \le \Delta_{y,s} L(x,t,y,s)-R(y,s)L(x,t,y,s)
\label{E310aa}
\end{align}
and
\begin{align} 
\lim_{s \nearrow t}L(x,t,y,s)= \delta_x.
\end{align}

For any $x,y \in M$, $s<T$ and small $\epsilon >0$ we have
\begin{align} 
&\int L(x,T,z,T-\epsilon)H(z,T-\epsilon,y,s)\phi^r(z,T-\epsilon)\,dV_{T-\epsilon}(z)\notag\\
&-\int L(x,T,z,s+\epsilon)H(z,s+\epsilon,y,s)\phi^r(z,s+\epsilon)\,dV_{s+\epsilon}(z) \notag \\
=&\int_{s+\epsilon}^{T-\epsilon} \partial_t\left(\int L(x,T,z,t)H(z,t,y,s)\phi^r(z,t)\, dV_t\right) \,dt \notag \\
=&\int_{s+\epsilon}^{T-\epsilon} \int L_tH\phi^r \, dV \,dt+\int_{s+\epsilon}^{T-\epsilon} \int LH_t\phi^r \,dV \,dt \notag \\
&+\int_{s+\epsilon}^{T-\epsilon} \int LH\phi^r_t\, dV\,dt-\int_{s+\epsilon}^{T-\epsilon} \int LH\phi^rR\, dV \,dt                         \notag \\
\ge &-\int_{s+\epsilon}^{T-\epsilon} \int \Delta LH\phi^r \, dV \,dt+\int_{s+\epsilon}^{T-\epsilon} \int L\Delta H\phi^r \,dV \,dt +\int_{s+\epsilon}^{T-\epsilon} \int LH\phi_t^r\, dV\,dt.
\label{E310a}
\end{align}

Here and after we omit all $z,t$ for notational simplicity.

By the integration by parts, we have
\begin{align} 
-\int_{s+\epsilon}^{T-\epsilon} \int \Delta LH\phi^r\, dV\,dt=&-\int_{s+\epsilon}^{T-\epsilon} \int L\left( \Delta H\phi^r+H\Delta \phi^r+2\langle \nabla H,\nabla \phi^r \rangle \right)\, dV\,dt
\label{E310b}
\end{align}

Therefore,
\begin{align} 
&\int L(x,T,z,T-\epsilon)H(z,T-\epsilon,y,s)\phi^r(z,T-\epsilon)\,dV_{T-\epsilon}(z)\notag\\
&-\int L(x,T,z,s+\epsilon)H(z,s+\epsilon,y,s)\phi^r(z,s+\epsilon)\,dV_{s+\epsilon}(z) \notag \\
\ge&\int_{s+\epsilon}^{T-\epsilon} \int LH \square \phi^r-2L\langle \nabla H,\nabla \phi^r \rangle\, dV\,dt.
\label{E310c}
\end{align}

Now we multiply both sides of $\square H=0$ by $(\phi^r)^2H$ and do the integration.
\begin{align} 
\int_{s+\epsilon}^{T-\epsilon}\int|\nabla (\phi^rH)|^2 \,dV\,dt \le& \int_{s+\epsilon}^{T-\epsilon}\int|\nabla \phi^r|^2H^2 \,dV\,dt+ 
\int_{s+\epsilon}^{T-\epsilon}\int \frac{H^2}{2}(\phi^r)^2_t \,dV\,dt \notag \\
&-\left.\left(\int (\phi^r)^2H^2 \,dV\right)\right|_{s+\epsilon}^{T-\epsilon}.
\label{E310d}
\end{align}
It is immediate by taking $r \to \infty$ that
\begin{align} 
\int_{s+\epsilon}^{T-\epsilon}\int|\nabla H|^2 \,dV\,dt <\infty.
\label{E310e}
\end{align}

For fixed $\epsilon$, we have
\begin{align} 
&\left|\int_{s+\epsilon}^{T-\epsilon} \int LH \square \phi^r-2L\langle \nabla H,\nabla \phi^r \rangle\, dV\,dt \right| \notag \\
 \le&\int_{s+\epsilon}^{T-\epsilon} \int LH |\square \phi^r|+2L|\nabla H||\nabla \phi^r|\, dV\,dt=I+II.
\label{E310f}
\end{align}

For the first term,
\begin{align*} 
\lim_{r\to \infty}I=\lim_{r\to \infty}\int_{s+\epsilon}^{T-\epsilon} \int LH |\square \phi^r|\, dV\,dt=0
\end{align*}
since $L$ is uniformly bounded on $M \times [s+\epsilon,T-\epsilon]$ and $H$ is integrable.

For the second term,
\begin{align*} 
II=\int_{s+\epsilon}^{T-\epsilon} \int L|\nabla H||\nabla \phi^r|\, dV\,dt \le 2\left( \int_{s+\epsilon}^{T-\epsilon} \int L^2|\nabla \phi^r|^2\, dV\,dt\right)^{\frac{1}{2}}\left(\int_{s+\epsilon}^{T-\epsilon} \int |\nabla H|^2\, dV\,dt\right)^{\frac{1}{2}}.
\end{align*}

It is easy to see that
\begin{align*} 
\int L\, dV_t \le 1
\end{align*}
for any $t <T$. The proof is almost identical with our proof of stochastic completeness (\ref{eqn:PK14_6}) by using \eqref{E310aa}.

Therefore, it is immediate from \eqref{E310e} that
\begin{align} 
\lim_{r \to \infty}II \le \lim_{r\to \infty}2\left( \int_{s+\epsilon}^{T-\epsilon} \int L^2|\nabla \phi^r|^2\, dV\,dt\right)^{\frac{1}{2}}\left(\int_{s+\epsilon}^{T-\epsilon} \int |\nabla H|^2\, dV\,dt\right)^{\frac{1}{2}}=0.
\label{E310h}
\end{align}

Now it follows from \eqref{E310c} that by taking $r \to \infty$,
\begin{align} 
\int L(x,T,z,T-\epsilon)H(z,T-\epsilon,y,s)\,dV_{T-\epsilon}(z)\ge \int L(x,T,z,s+\epsilon)H(z,s+\epsilon,y,s)\,dV_{s+\epsilon}(z).
\label{E310i}
\end{align}

As $\epsilon \to 0$, both $H(z,T-\epsilon,y,s)$ and $L(x,T,z,s+\epsilon)$ are uniformly bounded (in terms of $z$). We conclude from the definition of $\delta$ function that by taking $\epsilon \to 0$
\begin{align*} 
H(x,T,y,s)\ge L(x,T,y,s).
\end{align*}
\end{proof}

We also need the following gradient estimate from \cite{Zhangqi1}.

\begin{lem}
For any Ricci shrinker $(M^n,g,f)$, suppose $u$ is a positive bounded solution of the heat equation $\square u=0$ on $M \times [0,T]$, then
\begin{align*} 
\frac{|\nabla u|}{u} \le \sqrt{\frac{1}{t}}\sqrt{\log \frac{\Lambda}{u}}
\end{align*}
where $\Lambda=\max_{M \times [0,T]}u$.
\label{L301}
\end{lem}

\begin{proof}
From a direction computation
\begin{align*} 
\square \left(t\frac{|\nabla u|^2}{u}-u\log\frac{\Lambda}{u}\right)=-\frac{2}{u}\left|\text{Hess}\,u-\frac{du \otimes du}{u}\right|^2 \le 0.
\end{align*}
Now the theorem follows from Theorem \ref{T101} if
\begin{align*} 
\int_0^T\int\frac{|\nabla u|^2}{u}e^{-2f}\,dV_t\,dt <\infty.
\end{align*}
Notice that this follows the same proof as Lemma \ref{L201a}.
\end{proof}

Now we have the following corollary of Lemma \ref{L301}, see \cite[Equation $(3.44)$]{Zhangqi1}.
\begin{cor}\label{cor:harn}
With the same conditions as Lemma \ref{L301}, for any $\sigma>0$,
\begin{align} 
u(y,t) \le \Lambda^{\frac{\sigma}{1+\sigma}}u(x,t)^{\frac{1}{1+\sigma}}\exp\lc \frac{d_t^2(x,y)}{4\sigma t} \rc.
\label{E311}
\end{align}
\end{cor}
\begin{proof}
We rewrite Lemma \ref{L301} as
\begin{align*} 
\left|\na \sqrt{\log {\frac{\Lambda}{u}}}\right| \le \frac{1}{2\sqrt{t}},
\end{align*}
and hence
\begin{align*} 
\sqrt{\log {\frac{\Lambda}{u(x,t)}}} \le& \sqrt{\log {\frac{\Lambda}{u(y,t)}}}+\frac{d_t(x,y)}{2\sqrt{t}}.
\end{align*}

By squaring both sides above, we have
\begin{align*} 
\log{\frac{\Lambda}{u(x,t)}} \le& \lc\sqrt{\log{\frac{\Lambda}{u(y,t)}}}+\frac{d_t(x,y)}{2\sqrt{t}} \rc^2 \\
\le& (1+\sigma) \log{\frac{\Lambda}{u(y,t)}}+\frac{1+\sigma}{\sigma}\frac{d_t^2(x,y)}{4t}.
\end{align*}
Then the conclusion follows immediately.
\end{proof}

We now prove the pointwise lower bound of the heat kernel $H$.

\begin{thm}\label{XT301}
For any Ricci shrinker $(M^n,g,f)$, $0<\delta<1$, $D>1$ and $0<\epsilon <4$, there exists a constant $C=C(n,\delta,D)>0$ such that
\begin{align*} 
H(x,t,y,s) \ge \frac{C^{\frac{4}{\ep}}e^{\boldsymbol{\mu}(\frac{4}{\ep}-1)}}{(4\pi(t-s))^{n/2}} \exp{\lc-\frac{d_t^2(x,y)}{(4-\epsilon)(t-s)}\rc}
\end{align*}
for any $t \in [-\delta^{-1},1-\delta]$ and $d_t(p,y)+\sqrt{t-s}\le D$.
\end{thm}

\begin{proof}
From Theorem \ref{T302},
\begin{align} 
H(y,t,y,s) \ge \frac{e^{-l_{(y,t)}(y,s)}}{(4\pi(t-s))^{\frac{n}{2}}}.
\label{E312}
\end{align}

By the definition of $l$ and $\partial_z f(y,z)=|\nabla f|^2 \ge 0$,
\begin{align} 
l_{(y,t)}(y,s) \le& \frac{1}{2\sqrt{t-s}}\int_s^t \sqrt{t-z} R(y,z)\,dz \notag\\
\le& \frac{1}{2\sqrt{t-s}}\int_s^t \frac{\sqrt{t-z}}{1-z} f(y,z)\,dz \notag \\
\le&\frac{f(y,t)}{2\sqrt{t-s}}\int_s^t \frac{\sqrt{t-z}}{1-z}\,dz \le\frac{(t-s)}{3(1-t)^2}F(y,t)
\label{E313}
\end{align}
and hence
\begin{align*} 
H(y,t,y,s) \ge \frac{C}{(4\pi(t-s))^{n/2}}
\end{align*}
 for some constant $C=C(n,\delta,D)>0$.

By using \eqref{E311} for the heat kernel on $M \times [\frac{t+s}{2},t]$, we obtain
\begin{align*} 
H(y,t,y,s) \le e^{-\boldsymbol{\mu} \frac{\sigma}{1+\sigma}} (4\pi(t-s))^{-\frac{n}{2}\frac{\sigma}{1+\sigma}} H^{\frac{1}{1+\sigma}}(x,t,y,s)\exp{\lc\frac{d_t^2(x,y)}{4\sigma(t-s)}\rc}
\end{align*}
where we have used the result in Theorem \ref{T301} for the upper bound.

Therefore,
\begin{align*} 
H(x,t,y,s) \ge \frac{C^{1+\sigma}e^{\boldsymbol{\mu} \sigma}}{(4\pi(t-s))^{n/2}} \exp{\lc-\frac{1+\sigma}{\sigma}\frac{d_t^2(x,y)}{4(t-s)}\rc}.
\end{align*}

The conclusion follows by choosing $\sigma=4/\epsilon-1$.
\end{proof}

\begin{rem}
From the proof a more precise bound is, for any $0<\ep<4$,
\begin{align}\label{XT301a} 
H(x,t,y,s) \ge \frac{e^{\boldsymbol{\mu}(\frac{4}{\ep}-1)}}{(4\pi(t-s))^{n/2}} \exp{\lc-\frac{d_t^2(x,y)}{(4-\epsilon)(t-s)}-\frac{4(t-s)}{3(1-t)^2\epsilon}F(y,t)\rc}.
\end{align}
\end{rem}


In order to further estimate the upper bound of $H$, it is crucial to compare distance functions from different time slices. We first prove the second order estimate of the heat equation soluton on Ricci shrinkers, see \cite[Lemma $3.1$]{BZ17}.

\begin{lem}\label{L502}
Let $(M^n,g(t)),\, t\in[0,1)$ be the Ricci flow solution of a Ricci shrinker and let $u$ be a postive solution to the heat equation $\square u=0$ and $u \le \Lambda$ on $M \times [0,T]$. Then there exists a constant $C =C(n)$ such that
\begin{align} 
|\Delta u|+\frac{|\nabla u|^2}{u}-\Lambda R \le \frac{C\Lambda}{t}.
\label{E501}
\end{align}
\end{lem}
\begin{proof}
By rescaling, we assume that $\Lambda=1$. Let $L_1=-\Delta u+\frac{|\nabla u|^2}{u}-R$, then it follows from \cite[Equations $(3.3),(3.4)$]{BZ17} that
\begin{align} 
\square L_1 \le -\frac{1}{n}L_1^2+\frac{1}{e^2 t^2}.
\label{E502}
\end{align}

From \eqref{E502} we have
\begin{align} 
\square (L_1\phi^r) =&\phi^r\square L_1+L_1 \square \phi^r-2\langle \nabla \phi^r,\nabla L_1 \rangle  \notag \\
\le& \phi^r( -\frac{1}{n}L_1^2+\frac{1}{e^2 t^2})+L_1 \square \phi^r-2\langle \nabla \phi^r,\nabla L_1 \rangle \notag \\
=&\phi^r( -\frac{1}{n}L_1^2+\frac{1}{e^2 t^2})+L_1 \square \phi^r-2\frac{\langle \nabla (L_1\phi^r),\nabla \phi^r \rangle}{\phi^r}+2\frac{L_1|\nabla \phi^r|^2}{\phi^r}.
\label{E503}
\end{align}

Now at the maximum point of $L_1\phi^r$, we have
\begin{align} 
-\frac{1}{n}(L_1\phi^r)^2+(\phi^re^{-1}t^{-1})^2+(L_1\phi^r)\left( \square \phi^r+2\frac{|\nabla \phi^r|^2}{\phi^r}\right)\ge 0,
\label{E504}
\end{align}
so we obtain
\begin{align} 
L_1\phi^r \le n\left( \square \phi^r+2\frac{|\nabla \phi^r|^2}{\phi^r}\right)+\sqrt n\phi^re^{-1}t^{-1} \le C(n)(r^{-1}+t^{-1}).
\label{E505}
\end{align}

By taking $r \to \infty$, we have $L_1 \le C(n)t^{-1}$. Now if we set $L_2=\Delta u+\frac{|\nabla u|^2}{u}-R$, then similarly
\begin{align} 
\square L_2 \le -\frac{1}{2n}L_2^2+\frac{1+\frac{4}{n}}{e^2 t^2}.
\label{E506}
\end{align}

Therefore by the same method, we prove that $L_2 \le C(n)t^{-1}$.

Now the proof is complete.
\end{proof}

By applying the above lemma to the heat kernel, we immediately have from Theorem \ref{T301} that

\begin{lem}\label{L503}
For any Ricci shrinker $(M^n,g,f)$, there exists a constant $C =C(n)$ such that
\begin{align} 
|\partial_tH(x,t,y,s)|=|\Delta_xH(x,t,y,s)| \le C\frac{e^{-\boldsymbol{\mu}}}{(t-s)^{\frac{n}{2}}}\left (R(x,t)+\frac{1}{t-s}\right)
\label{E507}
\end{align}
for any $s<t<1$.
\end{lem}

Now we can prove the local distance distorsion on Ricci shrinkers. Notice that a similar estimate has been obtained on compact manifolds, see \cite[Theorem $1.1$]{{BZ17}}.

\begin{thm}[local distance distorsion estimate]
\label{T501a}
For any Ricci shrinker $(M^n,p,g,f) \in \mathcal M_n(A)$, $0<\delta<1$ and $D>1$, there exists a constant $Y=Y(n,A,\delta,D)>1$ such that for any two points $q$ and $z$ in $M$ with $d_t(p,q) \le D$ and $d_t(q,z)=r \le D$,
\begin{align*} 
Y^{-1}d_s(q,z) \le d_t(q,z) \le Y d_s(q,z)
\end{align*}
for any $t \in [-\delta^{-1},1-\delta-r^2]$ and $s \in [t-Y^{-1} r^2,t+Y^{-1}r^2]$.
\end{thm}

\begin{proof}
In the proof, all constants $C_i$ and $c_i$ depend on $n,A,\delta$ and $D$. Fix a time $T \in [-\delta^{-1},1-\delta-r^2]$, a point $q$ with $d_T(p,q) \le D$ and $r \le D$, we set $u(x,t)=H(x,t,q,T-r^2)$. It follows from Theorem \ref{XT301} that $w(y,T) \ge C_1 r^{-n}$ for any $y$ with $d_T(q,y) \le r$. For any $y \in B_T(q,r)$, we have from Lemma \ref{L503} that 
\begin{align} \label{e501a1}
|\partial_t w(y,t)| \le C_2r^{-n}(R(y,t)+r^{-2})
\end{align}
for $t \in [T-r^2/2,T+r^2]$. Since $d_T(p,y) \le d_T(p,q)+d_T(q,y) \le 2D$, it is clear from Lemma \ref{L100} that $F(y,T) \le c_1$. Moreover, it follows from \eqref{E108} and \eqref{E110} that
\begin{align*} 
|\partial_t F(y,t)|=|(1-t)R(y,t)| \le \frac{F(y,t)}{1-t} \le c_2 F(y,t).
\end{align*}
Therefore, it is clear that for any $t \in [T-r^2/2,T+r^2]$, $F(y,t) \le c_3$ and hence $R(y,t) \le c_4$ from \eqref{E110}. Since $r \le D$, we have from \eqref{e501a1}
\begin{align} \label{e501a2}
|\partial_t w(y,t)| \le C_3r^{-n-2}.
\end{align}
Now we set $c_5=C_1(2C_3)^{-1}$, it follows from $w(q,T) \ge C_1r^{-n}$ and \eqref{e501a2} that $w(y,t) \ge \frac{C_1}{2}r^{-n}$ on $B_T(q,r) \times [T-c_5r^2,T+c_5r^2]$. On the one hand, it follows from Corollary \ref{cor:harn} that $w \ge C_4r^{-n}$ on $B_t(y,r) \times \{t\}$. On the other hand, by Lemma \ref{L100}, $F$ and hence $R$ is bounded on $B_t(y,r) \times \{t\}$, we conclude from Theorem \ref{thm:PL06_2} that
\begin{align*} 
|B_t(y,r)|_t \ge C_5r^n.
\end{align*}
For any point $z$ with $d_T(q,z)=r$, we consider a geodesic $\gamma$ connecting $q$ and $z$. We claim that for any $t \in [T-c_5r^2,T+c_5r^2]$, $d_t(q,z) \le C_6 r$, where $C_6=8(C_4C_5)^{-1}$. Otherwise, we take a maximal set $\{y_i\}_{i=1}^N \subset \gamma$ such that $B_t(y_i,r)$ are mutually disjoint. In particular, it implies that $\{B_t(y_i,2r)\}$ covers $\gamma$. Then it is easy to see $C_6 r \le 4Nr$ and hence $N \ge \frac{C_6}{4}$. However, it follows from \eqref{Eqsto} that
\begin{align*} 
1 \ge \int w \,dV_t \ge \sum_{i=1}^N \int_{B_t(y_i,r)}w\,dV_t \ge \sum_{i=1}^N C_4r^{-n}|B_t(y_i,r)|_t \ge NC_4C_5 \ge 2,
\end{align*}
which is a contradiction. Now we set $c_6=c_5(2C_6)^{-2}$ and claim that $d_t(y,z) \ge (2C_6)^{-1}r$ for any $t \in [T-c_6r^2,T+c_6r^2]$. Otherwise, we can find a time $t \in [T-c_6r^2,T+c_6r^2]$ such that $d_t(y,z)=(2C_6)^{-1}r$. Since $c_6r^2 = c_5(2C_6)^{-2}r^2$, the argument before shows that $r=d_T(q,z) \le C_6 d_t(q,z)=r/2$ and this is impossible.

Therefore, by choosing $Y=\max\{c_6^{-1},2C_6\}$, the conclusion follows.
\end{proof}

Now we prove that $H$ has the exponential decay in the integral sense.

\begin{thm}
\label{XT306}
For any Ricci shrinker $(M^n,p,g,f)\in \mathcal M_n(A)$, $0<\delta<1$, $D>1$ and $\epsilon>0$, there exists a constant $C=C(n,A,\delta,D,\epsilon)>1$ such that
\begin{align*} 
\int_{M \backslash B_s(x,r\sqrt{t-s})} H(x,t,y,s)\,dV_s(y) \le C\exp{\lc-\frac{(r-1)^2}{4(1+\epsilon)}\rc}
\end{align*}
for any point $x \in M$, $t \in [-\delta^{-1},1-\delta]$, $d_t(p,x)+\sqrt{t-s}\le D$ and $r \ge 1$.
\end{thm}

\begin{proof}
It follows from Theorem \eqref{T306} with $\sigma=\epsilon$ that
\begin{align} 
\left(\int_{B_s(x,\sqrt{t-s})} H(x,t,y,s)\,dV_s(y)\right)^{\frac{1}{\epsilon}} \left(\int_{M \backslash B_s(x,r\sqrt{t-s})} H(x,t,y,s)\,dV_s(y)\right) \le \exp{\lc-\frac{(r-1)^2}{4(1+\epsilon)}\rc}
\label{E316a}
\end{align}
for any $r \ge 1$. So we only need to prove the first integral to be bounded below.

Theorem \ref{T501a} shows that there exists a constant $Y=Y(n,A,\delta,D)>1$ such that for any $y$ with $d_s(x,y) \le \sqrt{t-s}$, we have $d_t(x,y) \le Y\sqrt{t-s}$. Therefore, it follows from Theorem \ref{XT301} that
\begin{align*} 
H(x,t,y,s) \ge C(t-s)^{-n/2}
\end{align*}
for any $y$ with $d_s(x,y) \le \sqrt{t-s}$.

It implies that
\begin{align*}
\int_{B_s(x,\sqrt{t-s})} H(x,t,y,s)\,dV_s(y) \ge C(t-s)^{-n/2}|B_s(x,\sqrt{t-s})|_s \ge C
\end{align*}
where we have used the fact that $R$ is locally bounded.
\end{proof}

As we have proved that all distance functions to the base point $p$ are comparable, we prove the following weaker upper bound.

\begin{thm}
For any Ricci shrinker $(M^n,p,g,f)\in \mathcal M_n(A)$, $x \in M$ and $s<t<1$, there exist constants $C=C(n,A,x,t,s)>1$ and $c=c(n,A,x,t,s)>0$ such that
\begin{align*} 
H(x,t,y,s) \le Ce^{-cd_0^2(p,y)}.
\end{align*}
\label{T304}
\end{thm}
\begin{proof}
Fix $s<t<1$ and $x$ and we require that all constants in the proof depend on $n,x,s,t$ and $A$. Notice that since $s$ and $t$ are fixed, $f$ is comparable to $d_0^2(p,\cdot)$ by Lemma \ref{L101}.

For an $\epsilon>0$ to be chosen later, we have from the semigroup property
\begin{align*} 
H(x,t,y,s)=&\int H(x,t,z,l)H(z,l,y,s) \,dV_{l}(z) \notag \\
=&\int_{d_0(p,z) \ge \epsilon d_0(p,y) } H(x,t,z,l)H(z,l,y,s) \,dV_{l}(z) \\
&+\int_{d_0(p,z) \le \epsilon d_0(p,y)} H(x,t,z,l)H(z,l,y,s) \,dV_{l}(z)=I+II
\end{align*}
where $l=\frac{s+t}{2}$.

Now from Theorem \ref{XT306}
\begin{align} 
I =&\int_{d_0(p,z) \ge \epsilon d_0(p,y) } H(x,t,z,l)H(z,l,y,s) \,dV_{l}(z) \notag\\
\le& C_1 \int_{d_0(p,z) \ge \epsilon d_0(p,y) } H(x,t,z,l) \,dV_{l}(z) \le C_2e^{-c_1 \epsilon^2d_0^2(p,y)}.
\label{E318}
\end{align}
Note that here we can always assume that $\epsilon d_0(p,y)$ is large.

We choose $\phi$ which is identical $1$ on $B_l(p,c_2\epsilon d_0(p,y))$ and supported on $B_l(p,2c_2\epsilon d_0(p,y))$ where we choose $c_2$ that $B_0(p,\epsilon d_0(p,y)) \subset B_l(p,c_2\epsilon d_0(p,y))$.

If we set $w=\frac{e^{-f}}{(4\pi\tau)^{n/2}}$, there are $c_3$ and $c_4$ that
\begin{align*} 
c_3e^{c_4\epsilon^2d_0^2(p,y)}w(\cdot,l) \ge \phi(\cdot,l).
\end{align*}

Now from Lemma \ref{L202} that 
\begin{align*} 
II=& \int_{B_0(p,\epsilon d_0(p,y))} H(x,t,z,l)H(z,l,y,s) \,dV_{l}(z) \le c_5\int_{B_l(p,c_2\epsilon d_0(p,y))} H(z,l,y,s)\,dV_l(z) \notag \\
\le & c_5\int H(z,l,y,s)\phi(z,l)\,dV_l(z) \le c_6e^{c_4\epsilon^2d_0^2(p,y)}w(y,s).
\end{align*}

By the definition of $w$,
\begin{align*} 
w(y,s) \le c_7e^{-c_8d_0^2(p,y)}.
\end{align*}

Hence,
\begin{align} 
II \le c_9e^{-(c_8-c_4 \epsilon^2)d_0^2(p,y)}.
\label{E319}
\end{align}

If we choose $\epsilon=\sqrt{\frac{c_8}{2c_4}}$, it follows from \eqref{E318} and \eqref{E319} that
\begin{align*} 
H(x,t,y,s) \le Ce^{-cd_0^2(p,y)}.
\end{align*}
\end{proof}


\section{Differential Harnack inequality on Ricci shrinkers}

In this subsection, we prove that Perelman's differential Harnack inequality holds on Ricci shrinkers.

For any Ricci shrinker $(M^n,g,f)$, we fix a point $q \in M$ and a time $T <1$. Moreover, we set 
\begin{align} 
w(x,t)=H(q,T,x,t)=\frac{e^{-b(x,t)}}{(4\pi(T-t))^{n/2}}
\end{align}
and $\tau=T-t$.

We first prove

\begin{lem}\label{L401}
For any $r$ such that $\phi^r=1$ on an open neighborhood of $(q,T)$,
\begin{align} 
\lim_{t \nearrow T} \int bw\phi^r \,dV = \frac{n}{2}.
\label{E400}
\end{align}
\end{lem}

\begin{proof}
We set $K_r=\text{supp}\, \phi^r \bigcap M \times [T-1,T]$. Since we only care about the integral on the compact set $K_r$ when $t$ is sufficiently close to $T$, we can assume that the distances on different time slices from $t$ to $T$ are uniformly comparable. Now all constants $C$'s in the rest of the proof depend on $q,T,\boldsymbol{\mu}$ and the geometry on $K_r$. In particular, they are independent of $\tau$. 

Now we have from Theorem \ref{XT306} that
\begin{align}\label{E402}
\int_{d_t(q,x)\ge 2A\sqrt{\tau}} w(x,t)\,dV_t \le Ce^{-A^2/2}.
\end{align}

Moreover, from Theorem \ref{XT301},

\begin{align} 
b(x,t) \le C\left (1+\frac{d_t^2(q,x)}{\tau}\right)
\label{E403a}
\end{align}
for $(x,t) \in K_r$.

Now we set $d_t=d_t(q,x)$, then for any $A \ge 1$, we have

\begin{align*} 
\int_{K_r \cap \{d_t\ge 2A\sqrt{\tau}\}} bw\,dV_t &\le C\int_{K_r \cap \{d_t \ge 2A\sqrt{\tau}\}}w+\tau^{-1}d^2_tw\,dV_t \le Ce^{-A^2/2}+C\tau^{-1}\int_{K_r \cap \{d_t \ge 2A\sqrt{\tau}\}}d^2_tw\,dV_t.
\end{align*}

Now we have 
\begin{align*} 
\int_{K_r \cap \{d_t \ge 2A\sqrt{\tau}\}}d^2_tw\,dV_t & = \sum_{k=1}^{\infty}\int_{K_r \cap \{2^kA \sqrt{\tau} \le d_t \le 2^{k+1}A\sqrt{\tau}\}}d^2_tw\,dV_t \notag \\
&\le \sum_{k=1}^{\infty}2^{2k+2}A^2\tau\int_{K_r \cap \{2^kA \sqrt{\tau} \le d_t \le 2^{k+1}A\sqrt{\tau}\}}w\,dV_t \notag\\
& \le  \sum_{k=1}^{\infty}2^{2k+2}A^2e^{-2^{2k-3}A^2}\tau. \notag
\end{align*}

Therefore, we conclude that
\begin{align} \label{E404a} 
 \int_{K_r \cap \{d_t\ge 2A\sqrt{\tau}\}} bw\phi^r \,dV \le \int_{K_r \cap \{d_t\ge 2A\sqrt{\tau}\}} bw\,dV_t \le \eta(A)
\end{align}
where $\eta(A) \to 0$ if $A \to +\infty$.

In addition, it follows from Theorem \ref{T301} that $ b(x,t) \ge \boldsymbol{\mu}$ and hence
\begin{align} \label{E404b} 
 \int_{K_r \cap \{d_t\ge 2A\sqrt{\tau}\}} bw\phi^r \,dV \ge \boldsymbol{\mu}\int_{K_r \cap \{d_t\ge 2A\sqrt{\tau}\}} w \,dV \ge -Ce^{-A^2/2}
\end{align}
where the last inequality is from \eqref{E402}.

The inequalities \eqref{E404a} and \eqref{E404b} indicates that the integral $\int bw\phi^r$ is concentrated on the scale $\sqrt{\tau}$.

We take a sequence $\tau_i \to 0$ and set $g_i(t)=\tau_i^{-1}g(T-\tau_it)$ and $w_i(\cdot,t)=\tau_i^{n/2}w(\cdot,T-\tau_it)$. Then we have
\begin{align*} 
\partial_t w_i=\Delta_i w_i-R_iw_i,
\end{align*}
where $\Delta_i$ and $R_i$ are with respect to $g_i$.

Since $g_i$ is a blow-up sequence for the metric $g$ and $K_r$ has bounded geometry, it is easy to show that $(M,g_i,q)$ subconverges to $(\R^n,g_E,0)$ and $w_i$ converges a positive smooth function $w_{\infty}$ on $\R^n \times (0,\infty)$ such that
\begin{align*} 
\partial_t w_{\infty}=\Delta_{g_E} w_{\infty}.
\end{align*}

Now we can show as (\ref{eqn:PK14_4}) that $w_{\infty}$ is in fact a fundamental solution of the heat equation on the Euclidean space. Moreover it is easy to see by Fatou's  inequality that
\begin{align*} 
\int w_{\infty} \,dx \le 1
\end{align*}
for any time $t>1$. Now it follows from \cite[Corollary $9.6$]{Gr09} that $w_{\infty}$ is the heat kernel based at $0$, that is,
\begin{align*} 
w_{\infty}(x,t)=\frac{e^{-\frac{|x|^2}{4t}}}{(4\pi t)^{n/2}}. 
\end{align*}

From the smooth convergence,
\begin{align} \label{E405a} 
\lim_{i \to \infty}\left.\int_{K_r \cap \{d_t\le 2A\sqrt{\tau_i}\}} bw\,dV_t \right|_{t=T-\tau_i}=\int_{|x|\le 2A} \frac{|x|^2}{4}\frac{e^{-\frac{|x|^2}{4}}}{(4\pi)^{n/2}}\,dx.
\end{align}

By direct computations, 
\begin{align*} 
\int \frac{|x|^2}{4}\frac{e^{-\frac{|x|^2}{4}}}{(4\pi)^{n/2}}\,dx=\frac{n}{2}.
\end{align*}

Therefore, it is straightforward from \eqref{E404a}, \eqref{E404b} and the fact that $\phi^r$ is equal to $1$ on a neighborhood of $(q,T)$ that 
\begin{align*} 
\lim_{t \nearrow T} \int bw\phi^r \,dV = \frac{n}{2}.
\end{align*}
\end{proof}

\begin{rem}\label{rem:R401}
The same proof of Lemma \ref{L401} shows that if $u$ is a bounded smooth function on $M \times [T-1,T]$, then
\begin{align} 
\lim_{t \nearrow T} \int bwu\phi^r \,dV = \frac{n}{2}u(q,T).
\label{E400x}
\end{align}
\end{rem}

Now we set $d=d_T(q,\cdot)$, it follows from \eqref{XT301a} that
\begin{align} 
H(q,T,x,t) \ge C \frac{e^{-c_1\frac{d^2}{\tau}-c_2\tau F}}{\tau^{\frac{n}{2}}}.
\label{E402a}
\end{align}
In terms of $b$, we have
\begin{align} 
b(x,t) \le c_1\frac{d^2}{\tau}+c_2\tau F(x,T)+c_3
\label{E402b}
\end{align}

We denote $K^r_t=\{r \le F(\cdot, t) \le 2r\}$, then we have
\begin{lem}\label{L401a}
There exist constants $C_0$ and $C_1$ which depend only on $\boldsymbol{\mu},q$ and $T$ such that
\begin{align} 
\int_{T-1}^T \int_{K_t^r} |b|w \,dV\,dt \le C_0
\end{align}
for any $r \ge C_1$.
\end{lem}

\begin{proof}
From Lemma \ref{L100}, there exists $C_1>0$ such that for any $x \in K_t^r$ where $t \in [T-1,T]$,
$$
\frac{1}{5}d_t^2(p,x) \le F(x,t) \le d_t^2(p,x)
$$ 
if $r \ge C_1$

It follows from \eqref{E402b} that $|b|\le -\boldsymbol{\mu}+c_1\frac{d^2}{\tau}+c_2$. So we only need to estimate
\begin{align} 
\int_{T-1}^T \int_{K_t^r} d^2w \,dV\,dt.
\end{align}
Now it follows from the definition of $\phi^r$ that $K_t^r \subset \{c_4r \le d^2 \le c_5r\}$ if $C_1$ is sufficiently large, therefore
\begin{align} 
\int_{T-1}^T \int_{K_t^r}  d^2w \,dV\,dt \le&  C\int_{T-1}^Tr \int_{K_t^r} w \,dV\,dt \le  C\int_{T-1}^Tr e^{-\frac{c_3r}{\tau}} \,dt \le C_0.
\end{align}
Note that here we have used \eqref{E402}.
\end{proof}

Now we have the following spacetime integral estimate.

\begin{lem}\label{L402}
\begin{align} 
\int_{T-1}^{T-\epsilon} \int (|\nabla b|^2+R)w\,dV\,dt  \le C\log{\epsilon^{-1}},
\end{align}
\end{lem}
where $C$ depends only on $\boldsymbol{\mu},n,q$ and $T$.

\begin{proof}
From the evolution equation
\begin{align} 
\partial_t w=-\Delta w+Rw,
\end{align}
we immediately have
\begin{align} 
\partial_t b=-\Delta b+|\nabla b|^2-R+\frac{n}{2\tau}.
\end{align}

From an elementary computation,
\begin{align} 
&\partial_t \int wb\phi^r \,dV \notag\\
=& \int \left(b_tw\phi^r+bw_t\phi^r+bw\phi^r_t-bw\phi^rR \right)\,dV \notag \\
=&\int \left(-\Delta b+|\nabla b|^2-R+\frac{n}{2\tau}\right)w\phi^r-b\Delta w \phi^r+bw\phi^r_t\,dV \notag \\
=&\int \langle \nabla b,\nabla (w\phi^r) \rangle+\langle \nabla w,\nabla (b\phi^r) \rangle+|\nabla b|^2w\phi^r-Rw\phi^r+bw\phi^r_t+\frac{n}{2\tau}w\phi^r \,dV \notag \\
=&\int \langle \nabla b,\nabla \phi^r\rangle w+\langle \nabla w,\nabla \phi^r \rangle b-(|\nabla b|^2+R)w\phi^r+bw\phi^r_t+\frac{n}{2\tau}w\phi^r \,dV,
\label{E403}
\end{align}
where we have used $\nabla w=-w\nabla b$.

On the one hand we have,
\begin{align} 
\int \langle \nabla b,\nabla \phi^r\rangle w \,dV \le& \int |\nabla b||\nabla \phi^r| w \,dV \notag \\
\le& \frac{1}{4}\int |\nabla b|^2w\phi^r \,dV+\int \frac{|\nabla \phi^r|^2}{\phi^r} w \,dV.
\label{E404}
\end{align}

On the other hand
\begin{align} 
\int \langle \nabla w,\nabla \phi^r\rangle b \,dV \le& \int |\nabla w||\nabla \phi^r| b \,dV =\int |\nabla b||\nabla \phi^r| wb \,dV \notag \\
\le& \frac{1}{4}\int |\nabla b|^2w\phi^r \,dV+\int \frac{|\nabla \phi^r|^2}{\phi^r} b^2w \,dV.
\label{E405}
\end{align}

Now \eqref{E403} becomes
\begin{align} 
\partial_t \int wb\phi^r \,dV \le -\frac{1}{2}\int (|\nabla b|^2+R)w\phi^r \,dV+X+\frac{n}{2\tau}
\label{E406}
\end{align}
where
\begin{align*} 
X=\int bw\phi^r _t-\frac{|\nabla \phi^r|^2}{\phi^r} w-\frac{|\nabla \phi^r|^2}{\phi^r} b^2w \,dV.
\end{align*}

Integrate \eqref{E405} from $T-1$ to $T-\epsilon$, we have
\begin{align} 
\frac{1}{2}\int_{T-1}^{T-\epsilon} \int (|\nabla b|^2+R)w\phi^{r}\,dV\,dt  \le \left.\left(\int wb\phi^r\,dV\right)\right|_{T-\epsilon}^{T-1}+\frac{n}{2}\log{\epsilon^{-1}}+Y
\label{E407}
\end{align}
where
\begin{align*} 
Y=\int_{T-1}^{T-\epsilon}\int bw\phi^r _t-\frac{|\nabla \phi^r|^2}{\phi^r} w-\frac{|\nabla \phi^r|^2}{\phi^r} b^2w \,dV\,dt.
\end{align*}

At the time $T-1$, since $b=-\log w-\frac{n}{2}\log{4\pi}$, we have 
\begin{align} 
\left(\int wb\phi^r\,dV\right)(T-1)=\int w(-\log w-\frac{n}{2}\log{4\pi})\phi^r \,dV_{T-1} \le C
\label{E408}
\end{align}
where the last inequality can be seen from Theorem \ref{T304}.

Moverover,
\begin{align} 
\left(\int wb\phi^r\,dV\right)(T-\epsilon) \ge \boldsymbol{\mu} \int w\phi^r \,dV_{T-1} \ge \boldsymbol{\mu}.
\label{E409}
\end{align}

Now it follows from Theorem \ref{T304} and Lemma \ref{L101} that 
\begin{align} 
\lim_{r \to \infty}|Y|=0.
\end{align}

So if we let $r \to \infty$ in \eqref{E407}, the proof is complete.
\end{proof}

From Lemma \ref{L402}, we have

\begin{lem}\label{L403}
There exist a sequence $\tau_i \to 0$ and a constant $C>0$ such that 
\begin{align} 
\left( \int \tau(|\nabla b|^2+R)w\,dV  \right)(\tau_i) \le C.
\end{align}
\end{lem}

\begin{proof}
If the conclusion does not hold, we can find a function $C(\tau)$ such that $\lim_{\tau \to 0}C(\tau)=+\infty$ and 
\begin{align}
\int (|\nabla b|^2+R)w\,dV \ge \frac{C(\tau)}{\tau}.
\end{align}
But it obviously contradicts Lemma \ref{L402} if $\epsilon$ is sufficiently small.
\end{proof}

\begin{lem}\label{L403a}
For any $\theta>0$,
\begin{align} 
\int_{T-1}^{T} \int \tau^{\theta}(|\nabla b|^2+R)w\,dV\,dt  < \infty.
\end{align}
\end{lem}

\begin{proof}
It follows from Lemma \ref{L402} that
\begin{align*}
&\int_{T-1}^{T} \int \tau^{\theta}(|\nabla b|^2+R)w\,dV\,dt \\
=&\sum_{k=0}^{\infty} \int_{T-2^{-k}}^{T-2^{-k-1}} \int \tau^{\theta}(|\nabla b|^2+R)w\,dV\,dt  \notag \\
 \le & \sum_{k=0}^{\infty} 2^{-\theta k} \int_{T-1}^{T-2^{-k-1}} \int (|\nabla b|^2+R)w\,dV\,dt \notag \\
\le& \sum_{k=0}^{\infty} 2^{-\theta k}\log {2^{-k-1}} <\infty. \notag
\end{align*}
\end{proof}

Now we fix a nonnegative function $u$ on the time slice $T-1$ such that $\sqrt{u}$ smooth and compactly supported. We denote by the same $u$ as its heat equation solution.

Then we have

\begin{lem}\label{L404}
There exists a constant $C>0$ such that 
\begin{align*} 
\frac{|\nabla u|^2}{u} \le C
\end{align*}
on $M \times [T-1,T]$.
\end{lem}

\begin{proof}
The conclusion follows directly from
\begin{align*} 
\square \frac{|\nabla u|^2}{u}=-\frac{2}{u}\left|\text{Hess}\,u-\frac{du \otimes du}{u}\right|^2
\end{align*}
and Theorem \ref{T101}. Note that the assumption in Theorem \ref{T101} can be checked similarly as Lemma \ref{L201a}
\end{proof}

We also need the following lemma, whose proof is similar to Lemma \ref{L103}.

\begin{lem}\label{L404a}
There exists a constant $C>0$ such that
\begin{align} 
\int_{T-1}^T \int |\text{Hess}\,F|^2 w \,dV\,dt \le C.
\end{align}
\end{lem}

\begin{proof}
From the evolution equation $\square |\nabla F|^2=-2|\text{Hess}\,F|^2$, we have
\begin{align*} 
\partial_t \int |\nabla F|^2 w\phi^r \,dV=&\int (\Delta |\nabla F|^2-2|\text{Hess}\,F|^2 )w\phi^r-|\nabla F|^2 \Delta w\phi^r+|\nabla F|^2w \phi^r_t \, dV.
\end{align*}
Integrate above from $T-1$ to $T$, we get
\begin{align*} 
&\int_{T-1}^T\int 2|\text{Hess}\,F|^2 w\phi^r \,dV \,dt  \\
\le &\int_{T-1}^T\int -2\langle \nabla |\nabla F|^2, \nabla \phi^r\rangle w +|\nabla F|^2w \square \phi^r\,dV \,dt-\left.\left( \int |\nabla F|^2 w\phi^r \,dV\right)\right|_{T-1}^{T}\notag \\
\le& \int_{T-1}^T\int |\text{Hess}\,F|^2 w\phi^r+4|\nabla F|^2\frac{|\nabla \phi^r|^2}{\phi^r}w+|\nabla F|^2w \square \phi^r\,dV \,dt -\left.\left( \int |\nabla F|^2 w\phi^r \,dV\right)\right|_{T-1}^{T}.
\end{align*}
From \eqref{E113} and \eqref{E115a}, there exists a constant $C$ independent of $r$ such that
\begin{align*}
|\nabla F|^2\frac{|\nabla \phi^r|^2}{\phi^r}+|\nabla F|^2 |\square \phi^r| \le C.
\end{align*}
Therefore,
\begin{align*} 
\int_{T-1}^T\int |\text{Hess}\,F|^2 w\phi^r \,dV \,dt \le C.
\end{align*}
 Now the lemma follows by taking $r \to \infty$.
\end{proof}

With the same proof, we have

\begin{lem}\label{L404b}
There exists a constant $C>0$ such that
\begin{align*} 
\int_{T-1}^T \int |\text{Hess}\,u|^2 w \,dV\,dt \le C.
\end{align*}
\end{lem}

As before, we set 
\begin{align*} 
v=\left(\tau(2\Delta b-|\nabla b|^2+R)+b-n\right)w,
\end{align*}
and therefore
\begin{align} 
\partial_tv=-\Delta v+Rv+2\tau\left|Rc+\text{Hess}\,b-\frac{g}{2\tau}\right|^2w.
\label{E410}
\end{align}

Now we prove
\begin{lem}\label{L405}
There exist a sequence $\tau_i \to 0$ and a constant $C>0$ independent of $r$ and $i$ such that
\begin{align*} 
\left(\int vu\phi^r \,dV\right)(\tau_i) \le C.
\end{align*}
\end{lem}

\begin{proof}
From integration by parts, we have
\begin{align*} 
\int vu\phi^r \,dV=&\int \left(\tau(2\Delta b-|\nabla b|^2+R)+b-n\right)wu\phi^r\,dV \notag \\
=&\int -2\tau \langle \nabla b,\nabla w\rangle u\phi^r-2\tau \langle \nabla b, \nabla u \rangle w\phi^r-2\tau \langle \nabla b, \nabla \phi^r \rangle wu\,dV \notag \\
&+\int \left(\tau(R-|\nabla b|^2)+b-n\right)wu\phi^r\,dV \notag \\
=&\int \left(\tau(|\nabla b|^2+R)+b-n\right)wu\phi^r-2\tau \langle \nabla b, \nabla u \rangle w\phi^r-2\tau \langle \nabla b, \nabla \phi^r \rangle wu\,dV.
\end{align*}

In addition,
\begin{align*} 
&\int -2\tau \langle \nabla b, \nabla u \rangle w\phi^r-2\tau \langle \nabla b, \nabla \phi^r \rangle wu\,dV\\
 \le& \int 2\tau |\nabla b||\nabla u|w\phi^r+2\tau|\nabla b||\nabla \phi^r|wu\,dV \notag\\
\le& \tau\int 2|\nabla b|^2wu\phi^r+\frac{|\nabla u|^2}{u}w\phi^r+\frac{|\nabla \phi^r|^2}{\phi^r}wu\,dV
\end{align*}

Now the conclusion follows immediately from Lemma \ref{L401}, Lemma \ref{L403} and Lemma \ref{L404}.
\end{proof}

We are now ready to estimate the squared term in \eqref{E410}.

\begin{lem}\label{L406}
\begin{align*} 
\int_{T-1}^T \int \tau\left|Rc+\text{Hess}\,b-\frac{g}{2\tau}\right|^2wu \,dV\,dt <\infty.
\end{align*}
\end{lem}

\begin{proof}
We denote $A= 2\tau\left|Rc+\text{Hess}\,b-\frac{g}{2\tau}\right|^2w$. By computations,
\begin{align*} 
\partial_t\int vu\phi^r \,dV=& \int v_tu\phi^r+vu_t\phi^r+uv\phi^r_t-Ruv\phi^r\,dV \notag \\
=& \int -\Delta vu\phi^r+Au\phi^r+\Delta u v \phi^r+uv\phi^r_t\,dV \notag \\
=& \int uv\square\phi^r-2v\langle \nabla \phi^r,\nabla u \rangle+Au\phi^r\,dV.
\end{align*}

Now we have
\begin{align} 
&\int uv\square\phi^r\,dV \notag \\
=&\int \left(\tau(2\Delta b-|\nabla b|^2+R)+b-n\right)wu\square\phi^r\,dV                          \notag \\
=& \int -2\tau \langle \nabla b,\nabla w\rangle u\square \phi^r-2\tau \langle \nabla b,\nabla u\rangle w\square \phi^r-2\tau \langle \nabla b,\nabla \square \phi^r\rangle uw \,dV \notag \\
&+\int \left(\tau(R-|\nabla b|^2)+b-n\right)wu\square\phi^r\,dV     \notag \\
=&\int \left(\tau(|\nabla b|^2+R)+b-n\right)wu\square\phi^r\,dV -2\tau\int  \langle \nabla b,\nabla u\rangle w\square \phi^r+ \langle \nabla b,\nabla \square \phi^r\rangle uw \,dV
\label{E412}
\end{align}

For the last integral, 
\begin{align*} 
2\int  \langle \nabla b,\nabla u\rangle w\square \phi^r \,dV \le 2\int  | \nabla b||\nabla u|w|\square \phi^r| \,dV\le \int |\nabla b|^2wu|\square \phi^r|+\frac{|\nabla u|^2}{u}w|\square \phi^r| \,dV
\end{align*}
and 
\begin{align*} 
2\int  \langle \nabla b,\nabla \square \phi^r\rangle uw \,dV \le 2\int  | \nabla b||\nabla \square \phi^r|uw \,dV \le \int_{K^r_t} |\nabla b|^2wu+|\nabla \square \phi^r|^2uw \,dV.
\end{align*}

By the explicit expression $\square \phi^r=-nr^{-1}\eta'/2-r^{-2}\eta''|\nabla F|^2$, we have
\begin{align*} 
|\nabla\square \phi^r|=&\left|-nr^{-2}\nabla F\eta''/2-r^{-3}\eta'''\nabla F|\nabla F|^2-2r^{-2}\eta''\text{Hess}\,F(\nabla F)\right| \notag \\
\le & Cr^{-2}|\nabla F|\left(1+|\text{Hess}\,F|+r^{-1}|\nabla F|^2 \right).
\end{align*}

In addition, 
\begin{align*} 
 &\int v\langle \nabla \phi^r,\nabla u \rangle\,dV \notag \\
=&\int \left(\tau(2\Delta b-|\nabla b|^2+R)+b-n\right)w\langle \nabla \phi^r,\nabla u \rangle\,dV                          \notag \\
=& \int -2\tau \langle \nabla b,\nabla w\rangle \langle \nabla \phi^r,\nabla u \rangle-2\tau \text{Hess}\,\phi^r(\nabla b,\nabla u)w-2\tau \text{Hess}\,u(\nabla b,\nabla \phi^r)w \,dV \notag \\
&+\int \left(\tau(R-|\nabla b|^2)+b-n\right)w\langle \nabla \phi^r,\nabla u \rangle\,dV     \notag \\
=&\int \left(\tau(|\nabla b|^2+R)+b-n\right)w\langle \nabla \phi^r,\nabla u \rangle\,dV -2\tau \text{Hess}\,\phi^r(\nabla b,\nabla u)w-2\tau \text{Hess}\,u(\nabla b,\nabla \phi^r)w \,dV.
\end{align*}

To estimate the last two terms, since $|\nabla u|$ is uniformly bounded,
\begin{align*} 
\int \text{Hess}\,\phi^r(\nabla b,\nabla u)w \,dV \le \int |\text{Hess}\,\phi^r| | \nabla b||\nabla u| w\,dV \le C\int_{K^r_t}  |\nabla b|^2w+ |\text{Hess}\,\phi^r|^2w  \,dV.
\end{align*}

Note that we have
\begin{align*} 
|\text{Hess}\,\phi^r|=\left|r^{-2}\eta'' F_iF_j+r^{-1}\eta' \text{Hess}\,F\right| \le Cr^{-1}\left(|\text{Hess}\,F|+r^{-1}|\nabla F|^2\right)
\end{align*}
and 
\begin{align*} 
\int \text{Hess}\,u(\nabla b,\nabla \phi^r)w \,dV \le \int |\text{Hess}\,u| | \nabla b||\nabla \phi^r| w\,dV \le Cr^{-\frac{1}{2}}\int_{K^r_t} |\nabla b|^2w+ |\text{Hess}\,u|^2w  \,dV.
\end{align*}

Now we integrate \eqref{E412} from $T-1$ to $T-\tau_i$, 
\begin{align*} 
&\int_{T-1}^{T-\tau_i} \int Au\phi^r\,dV \notag\\
\le & \left.\left(\int vu\phi^r \,dV\right) \right|_{T-1}^{T-\tau_i}+\int_{T-1}^{T-\tau_i}\int \left(\tau(|\nabla b|^2+R)+b-n\right)w(2\langle \nabla \phi^r,\nabla u \rangle-u\square\phi^r)\,dV\,dt \notag\\
&+\int_{T-1}^{T-\tau_i}\tau \int_{K^r_t} |\nabla b|^2wu|\square \phi^r|+\frac{|\nabla u|^2}{u}w|\square \phi^r| + |\nabla b|^2wu\,dV\,dt \notag\\
&+Cr^{-4}\int_{T-1}^{T-\tau_i}\tau \int_{K^r_t}|\nabla F|^2\left(1+|\text{Hess}\,F|^2+r^{-2}|\nabla F|^4 \right)uw\,dV\,dt \notag\\
&+C\int_{T-1}^{T-\tau_i}\tau \int_{K^r_t} |\nabla b|^2w+|\text{Hess}\,u|^2w\,dV\,dt \notag\\
&+Cr^{-2}\int_{T-1}^{T-\tau_i}\tau \int_{K^r_t}\left(|\text{Hess}\,F|^2+r^{-2}|\nabla F|^4 \right)w\,dV\,dt. \notag\\
\end{align*}

Therefore,
\begin{align} 
&\int_{T-1}^{T-\tau_i} \int Au\phi^r\,dV \notag\\
\le & \left.\left(\int vu\phi^r \,dV\right) \right|_{T-1}^{T-\tau_i}+Cr^{-\frac{1}{2}}\int_{T-1}^{T}\int_{K^r_t} \left(\tau(|\nabla b|^2+R)+|b|+n\right)w\,dV\,dt \notag\\
&+\int_{T-1}^{T}\tau \int_{K^r_t} Cr^{-1}(|\nabla b|^2w+w) + |\nabla b|^2w\,dV\,dt \notag\\
&+Cr^{-2}\int_{T-1}^{T}\tau \int_{K^r_t}\left(1+|\text{Hess}\,F|^2 \right)uw\,dV\,dt \notag\\
&+C\int_{T-1}^{T}\tau \int_{K^r_t} |\nabla b|^2w+|\text{Hess}\,u|^2w\,dV\,dt 
\label{E418a}
\end{align}

For a fixed $i$, from Theorem \ref{T304}, Lemma \ref{L103}, Lemma \ref{L403a}, Lemma \ref{L404} and Lemma \ref{L404a}, we have by taking $r \to \infty$ that
\begin{align} 
\int_{T-1}^{T-\tau_i} \int Au \,dV =&\lim_{r \to \infty}\int_{T-1}^{T-\tau_i} \int Au\phi^r \,dV =\lim_{r \to \infty}\left.\left(\int vu\phi^r \,dV\right) \right|_{T-1}^{T-\tau_i} \le C,
\label{E419}
\end{align}
where the last inequality follows from Lemma \ref{L405}.

Now the lemma follows from \eqref{E419} by taking $i \to \infty$.
\end{proof}

A consequence of Lemma \ref{L406} is 

 \begin{lem}\label{L407}
There exists a sequence $\tau_j \to 0$ such that
\begin{align*} 
\lim_{j \to \infty} \int \tau_j^2\left|Rc+\text{Hess}\,b-\frac{g}{2\tau}\right|^2wu+\tau_j^{\frac{3}{2}}|\nabla b|^2w \,dV =0.
\end{align*}
\end{lem}

\begin{proof}
It follows from Lemma \ref{L406} and Lemma \ref{L403a} that
\begin{align*} 
\int_{T-1}^{T} \int \tau\left|Rc+\text{Hess}\,b-\frac{g}{2\tau}\right|^2wu+\tau^{\frac{1}{2}}|\nabla b|^2w \,dV < \infty.
\end{align*}
Now the conclusion is obvious.
\end{proof}

Note that the sequence $\tau_j$ may not be the same sequence $\tau_i$ in Lemma \ref{L403}.

Finally, we can prove Perelman's differential Harnack inequality.

\begin{thm}
\label{T401}
\begin{align*}
\tau(2\Delta b-|\nabla b|^2+R)+b-n \le 0.
\end{align*}
\end{thm}

\begin{proof}
As $T-1$ can be any time $S<T$, we just need to prove $v \le 0$ on $T-1$.

For the chosen $\tau_j$ obtained in Lemma \ref{L407}, we have
\begin{align} 
&\left(\int vu\phi^r \,dV\right)(\tau_j) \notag \\
=&\int \left(\tau_j(2\Delta b-|\nabla b|^2+R)+b-n\right)wu\phi^r\,dV            \notag \\
=&\int \tau_j(\Delta b+R-\frac{n}{2\tau_j})wu\phi^r-\tau_j\langle \nabla b,\nabla u \rangle w\phi^r-\tau_j\langle \nabla b,\nabla \phi^r \rangle uw+(b-\frac{n}{2})wu\phi^r\,dV
\label{E420}
\end{align}

On the one hand,
\begin{align} 
\int \tau_j(\Delta b+R-\frac{n}{2\tau_j})wu\phi^r \,dV \le& \tau_j \left( \int \left|Rc+\text{Hess}\,b-\frac{g}{2\tau}\right|^2wu \,dV \right)^{\frac{1}{2}}\left(\int wu \, dV \right)^{\frac{1}{2}} \notag \\
\le& C\left( \int \tau_j^2\left|Rc+\text{Hess}\,b-\frac{g}{2\tau}\right|^2wu \,dV \right)^{\frac{1}{2}}.
\label{E421}
\end{align}

On the other hand,
\begin{align} 
&\int -\tau_j\langle \nabla b,\nabla u \rangle w\phi^r-\tau_j\langle \nabla b,\nabla \phi^r \rangle uw \, dV \notag \\
\le& C\tau_j\int |\nabla b|w \,dV \le C\left( \int \tau_j^{\frac{3}{2}}|\nabla b|^2 \,dV \right)^{\frac{1}{2}}\left(\int \tau_j^{\frac{1}{2}}w \, dV \right)^{\frac{1}{2}} \notag \\
=&C\tau_j^{\frac{1}{4}}\left( \int \tau_j^{\frac{3}{2}}|\nabla b|^2 \,dV \right)^{\frac{1}{2}}.
\label{E422}
\end{align}

In addition, it follows from Lemma \ref{L401} and Remark \ref{rem:R401} that
\begin{align} 
\int (b-\frac{n}{2})wu\phi^r\,dV=0.
\label{E422a}
\end{align}

Combining \eqref{E421}, \eqref{E422} and \eqref{E422a}, it follows immediately from Lemma \ref{L407} and Lemma \ref{L401} that
\begin{align*} 
\lim_{j \to \infty}\left(\int vu\phi^r \,dV\right)(\tau_j)=0.
\end{align*}

Now we consider \eqref{E418a}, with $\tau_i$ replaced by $\tau_j$, and let $j \to \infty$. 
\begin{align*} 
 &\left(\int vu\phi^r \,dV\right)(T-1) \notag \\
\le& Cr^{-\frac{1}{2}}\int_{T-1}^{T}\int_{K^r_t} \left(\tau(|\nabla b|^2+R)+|b|+n\right)w\,dV\,dt \notag\\
&+\int_{T-1}^{T}\tau \int_{K^r_t} Cr^{-1}(|\nabla b|^2w+w) + |\nabla b|^2w\,dV\,dt \notag\\
&+Cr^{-2}\int_{T-1}^{T}\tau \int_{K^r_t}\left(1+|\text{Hess}\,F|^2 \right)uw\,dV\,dt \notag\\
&+C\int_{T-1}^{T}\tau \int_{K^r_t} |\nabla b|^2w+|\text{Hess}\,u|^2w\,dV\,dt 
\end{align*}

It is easy to see all integrals above are finite, by Lemma \ref{L401a}, Lemma \ref{L404}, Lemma \ref{L404a} and Lemma \ref{L404b}.

By the definition of $K^r_t$, if we let $r \to \infty$, then
\begin{align*} 
\left(\int vu \,dV\right)(T-1) \le 0.
\end{align*}

By the arbitrary choice of $u$ at $T-1$, we have proved that $v \le 0$.
\end{proof}

\begin{rem}
Note that as in Perelman's paper \cite{Pe1}, Theorem \ref{T302} is a corollary of Theorem \ref{T401}. Our proof of Theorem \ref{T401} is different from most literature, for instance \cite{Ni06} and \cite{CTY}, in that we do not need a pointwise gradient estimate of the conjugate heat kernel, see \cite[Lemma $2.2$]{Ni06}.
\end{rem}

\begin{rem}
The proof of Theorem \ref{T401} shows the following identity. For any $S<T<1$,
\begin{align*} 
\left(\int vu \,dV\right)(S)=-\int_{S}^T \int 2\tau\left|Rc+\text{Hess}\,b-\frac{g}{2\tau}\right|^2wu \,dV\,dt.
\end{align*}
\end{rem}


\section{The no-local-collapsing theorems}
\label{sec:nolocal}

We need to use the local entropy in~\cite{BW17A}. Let us first recall some notations. 
Let $\Omega$ be a domain in $M$. Then we define (cf. \eqref{eqn:PL21_4} and \eqref{eqn:PL21_5} and Section 2 of~\cite{BW17A}):
\begin{align}
   &\boldsymbol{\mu}(\Omega, g, \tau) \coloneqq  \inf\left\{\overline{\mathcal W}(g,u,\tau) \left| u \in \mathcal \mathcal W_{*}^{1,2}(M), \; u \; \textrm{is supported on} \; \Omega \right. \right\}, \label{eqn:PL24_2}\\
   &\boldsymbol{\nu}(\Omega, g, \tau) \coloneqq  \inf_{s \in (0, \tau)} \boldsymbol{\mu}(\Omega, g, s).   \label{eqn:PL24_3}
\end{align}
When the meaning is clear in the context,  the metric $g$ may be dropped.  Note that if $\Omega$ does not appear, it means the default set is $M$. 
We shall exploit the argument in~\cite{BW17A} to obtain volume ratio estimate. 

\begin{theorem}
Suppose $(M^n, g, f)$ is a Ricci shrinker and $B=B(x,r) \subset M$ is a geodesic ball with $R \leq \Lambda$, then we have
\begin{align}
 &r^{-n} |B| \geq   c \cdot e^{\boldsymbol{\mu} -\Lambda r^2},
 \label{eqn:PL06_1}   
\end{align}
for some $c=c(n)$. 
If $r \in (0, 1)$, then (\ref{eqn:PL06_1}) can be improved to
\begin{align}
 &r^{-n} |B| \geq   c \cdot e^{\boldsymbol{\mu}(g, r^2) -\Lambda r^2}. 
 \label{eqn:PL06_2}   
\end{align}
\label{thm:PL06_1}
\end{theorem}

\begin{proof}
  We first show (\ref{eqn:PL06_1}). 
  By Theorem 3.3 of~\cite{BW17A}, we know that
  \begin{align}
 &r^{-n} |B| \geq   c(n) e^{\boldsymbol{\nu}(B, r^2)} e^{-\Lambda r^2},
 \label{eqn:PL06_4}   
 \end{align}
 where $\boldsymbol{\nu}(B,r^2)$ is the local $\boldsymbol{\nu}$-functional of $B$ on the scale $r^2$. 
 Since $(M, g)$ is a Ricci shrinker, it follows from (\ref{eqn:PK20_1}) in Theorem~\ref{thm:A} that
 \begin{align}
   \boldsymbol{\nu}(B, r^2) \geq \boldsymbol{\nu}(M, r^2)=\inf_{\tau \in (0,r^2)} \boldsymbol{\mu}(M, g, \tau) \geq \boldsymbol{\mu}.    \label{eqn:PL06_5}
 \end{align}
 If $r \in (0, 1)$, then $r^2 \in (0, 1)$. By the monotonicity in Theorem~\ref{thm:A}, the above inequality can be written as
 \begin{align}
   \boldsymbol{\nu}(B, r^2) \geq  \boldsymbol{\mu}(g, r^2).    \label{eqn:PL06_6}
 \end{align}
 Therefore, we obtain (\ref{eqn:PL06_1}) and (\ref{eqn:PL06_2}), after we plugging (\ref{eqn:PL06_5}) and (\ref{eqn:PL06_6}) into (\ref{eqn:PL06_4}) respectively. 

\end{proof}

\begin{theorem}
Suppose $(M^n, g, f)$ is a Ricci shrinker and $B=B(x,r) \subset M$ is a geodesic ball with $R \leq \Lambda$, then we have
\begin{align}
 r^{-n}|B| \ge c \cdot e^{\boldsymbol{\mu}}(1+\Lambda r^2)^{-\frac{n}{2}}.     \label{eqn:PL06_3}
\end{align}
\label{thm:PL06_2}
\end{theorem}

\begin{proof}

Choose $\rho_0 \in [0,r]$ such that $\displaystyle \inf_{s \in [0, r]}s^{-n}|B(q,s)|$  is achieved  at $\rho_0$. There are two cases $\rho_0=0$ and $\rho_0>0$, which we shall discuss separately.

 \textit{Case 1. $\rho_0=0$.}
 
 In this case, we have
 \begin{align}
     |B(q,r)|\ge \omega_n r^n,   \label{eqn:PK29_5}
 \end{align}
 where $\omega_n$ is the volume of the unit Euclidean ball.  Actually,  it is not hard to observe that 
 \begin{align}
    \boldsymbol{\mu} \leq 0.   \label{eqn:PK29_6}
 \end{align}
 Let $\tau \to 0^{+}$, it is clear that $(M^n,p,\tau^{-1}g)$ converges to $(\R^n,0,g_E)$ in the Cheeger-Gromov sense.
  By Lemma 3.2 of ~\cite{Li17}, we have
\begin{align} 
\limsup_{\tau \to 0^+}\boldsymbol{\mu}(g,\tau)=\limsup_{\tau \to 0^+}\boldsymbol{\mu}(\tau^{-1}g,1) \le \boldsymbol{\mu}(g_E,1)=0.    
\label{eqn:PL21_10}
\end{align}
As $\boldsymbol{\mu}(g,\tau)$ is decreasing  on $(0, 1)$ by Lemma~\ref{lma:PK10_3}, then (\ref{eqn:PK29_6}) follows from the above inequality. 
Consequently, (\ref{eqn:PL06_3}) follows from the combination of (\ref{eqn:PK29_5}) and (\ref{eqn:PK29_6}). 

  \textit{Case 2. $\rho_0>0$.}
 
We choose a nonincreasing smooth function $\eta$ on $\R$ such that $\eta=1$ on $(\infty,1/2]$ and $0$ on $[1,\infty)$. We also define $u(x)=\eta(\frac{d(q,x)}{\rho_0})$. From \eqref{E234} in Corollary~\ref{cor234},
we obtain
\begin{align*}
|B(q,\rho_0/2)|^{\frac{n-2}{n}} \le& Ce^{-\frac{2 \boldsymbol{\mu}}{n}}\int \left\{ 4|\nabla u|^2+Ru^2 \right\} dV \\
\le& Ce^{-\frac{2 \boldsymbol{\mu}}{n}}\lc \rho_0^{-2}|B(q,r)|+\int Ru^2 \,dV \rc  \\
\le& Ce^{-\frac{2 \boldsymbol{\mu}}{n}} \rho_0^{-2}(1+\Lambda r^2)|B(q,\rho_0)| 
\end{align*}
where the last inequality follows from $R \le \Lambda \le \Lambda r^2 \rho_0^{-2}$.
According to the choice of $\rho_0$, we obtain
\begin{align*}
|B(q,\rho_0/2)|  \ge 2^{-n}|B(q,\rho_0)|.
\end{align*}
Combining the previous two steps yields  that
\begin{align*}
|B(q,\rho_0)| \ge 2^n |B(q,\rho_0/2)| \geq Ce^{\boldsymbol{\mu}}(1+\Lambda r^2)^{-\frac{n}{2}}\rho_0^n.
\end{align*}
Recall that $r^{-n}|B(q,r)| \ge \rho_0^{-n} |B(q,\rho_0)|$ by our choice of $\rho_0$. Therefore, (\ref{eqn:PL06_3}) follows directly from the above inequality. 
\end{proof}

Note that Theorem~\ref{thm:PL06_1} is based on the Logarithmic Sobolev inequality, and Theorem~\ref{thm:PL06_2} relies on the Sobolev inequality. 
Each of Theorem~\ref{thm:PL06_1} and Theorem~\ref{thm:PL06_2} has its own advantage and will be used in the remainder of the section.
Bascially, Theorem~\ref{thm:PL06_1} is sharper when $r$ is very small and Theorem~\ref{thm:PL06_2} is more accurate in the situation when $\Lambda r^2$ is large.

Using the Sobolev constant estimate in Corollary \ref{cor234}, we can further improve   Theorem 6.1 of ~\cite{MW12} stating that for any noncompact Ricci shrinker, the volume increases at least linearly. 

\begin{proposition}
For any noncompact Ricci shrinker $(M^n,p,g,f)$, there exist big positive constant $r_0=r_0(n)$ and small positive constant $\epsilon_0=\epsilon_0(n)$ such that
\begin{align}
|B(p,r)| \ge \epsilon_0 e^{\boldsymbol{\mu}} r, \quad \forall \; r \geq r_0.   \label{eqn:PK29_2}
\end{align}
\label{prn:PK29_1}
\end{proposition}

\begin{proof}
Similar to the proof of Lemma~\ref{L101}, we follow the notation of~\cite{MW12} to denote
\begin{align*}
   &\rho \coloneqq 2\sqrt{f}, \quad D(r)  \coloneqq \{x\in M \mid \rho \le r\}, \quad A(s,r) \coloneqq D(r)\backslash D(s);\\
   &V(r) \coloneqq |D(r)|, \quad \chi(r) \coloneqq \int_{D(r)}R \,dV. 
\end{align*}
From Lemma~\ref{L100},  $V(r)$ is almost the volume of geodesic ball $B(p,r)$, with the advantage that the estimate of $V(r)$ is relatively easier than the estimate of $|B(p, r)|$. 
Actually, by equations (6.24) and (6.25) of~\cite{MW12},  we know that 
\begin{align}
&V(t+1) \le 2V(t), \label{T206e1}\\
&V(t+1)-V(t) \le C_1\frac{V(t)}{t},  \label{T206e2}
\end{align}
whenever $t \geq C_1$ for some dimensional constant $C_1=C_1(n)$. 
Now we define
\begin{align}
   r_0 \coloneqq \max \{100n,10C_1\}.   \label{eqn:PK29_7}
\end{align}
Therefore,  in order to prove (\ref{eqn:PK29_2}), it suffices to show that 
\begin{align}
       V(r) \geq \epsilon_0 e^{\boldsymbol{\mu}} r, \quad \forall \; r \geq r_0,       \label{eqn:PK29_3}
\end{align}
where $\epsilon_0=\epsilon_0(n)$ will be determined later.

We shall prove (\ref{eqn:PK29_3}) by a contradiction argument.
If (\ref{eqn:PK29_3}) were wrong,  then there exists an $r \ge 2r_0$ such that $V(r) \le \ep_0 e^{\boldsymbol{\mu}}r$ for $\ep_0$ to be determined later, we claim that 
\begin{align}
   V(t_m)\le 2\ep_0e^{\boldsymbol{\mu}}t_m, \quad t_m=r+m, \quad \forall m \in \mathbb N.   \label{eqn:PK29_4}
\end{align}
Indeed, by our assumption the case $m=0$ is true. We assume that 
the conclusion is true for all $m=0,1,2,\cdots,k$ and proceed to show it holds for $m=k+1$. 

For any $t \ge r_0$, we define
\begin{align*}
u(x) \coloneqq
\begin{cases}
1 \quad\quad &\text{on} \quad A(t,t+1),\\
t+2-\rho(x) \quad\quad &\text{on} \quad A(t+1,t+2),\\
\rho(x)-(t-1) \quad\quad &\text{on} \quad A(t-1,t),\\
0 \quad\quad&\text{otherwise}.
\end{cases}
\end{align*}
Let $t=t_m$ and plug the above $u$ into the Sobolev inequality~\eqref{E234}.
We obtain 
\begin{align}
\label{T206e3}
|A(t_m,t_{m+1})|^{\frac{n-2}{n}} \le C_3e^{-\frac{2\boldsymbol{\mu}}{n}} \lc |A(t_{m-1},t_m)|+|A(t_{m+1},t_{m+2})|+\chi(t_{m+2})-\chi(t_{m-1}) \rc
\end{align}
for some $C_3=C_3(n)$.  For $0\le m\le k$, it follows from our induction assumption and \eqref{T206e2} that
\begin{align}
\label{T206e4}
   |A(t_m,t_{m+1})|=V(t_{m+1})-V(t_m) \le C_1\frac{V(t_m)}{t_m} \le 2C_1\ep_0e^{\boldsymbol{\mu}}. 
\end{align}
Summing \eqref{T206e3} from $m=0$ to $m=k$, we have
\begin{align*}
 \sum_{m=0}^k|A(t_m,t_{m+1})|^{\frac{n-2}{n}}\le& C_3e^{-\frac{2\boldsymbol{\mu}}{n}} \sum_{m=0}^k\lc |A(t_{m-1},t_m)|+|A(t_{m+1},t_{m+2})|+\chi(t_{m+2})-\chi(t_{m-1}) \rc \\
 \le& 3C_3e^{-\frac{2\boldsymbol{\mu}}{n}} \lc |A(t_{-1},t_{k+2})|+\chi(t_{k+2})\rc.
\end{align*}
Recall that $\chi(t) \le \frac{n}{2}V(t)$ by (3.4) of~\cite{CZ10}.  Plugging this fact into the above inequality yields that
\begin{align}
  \sum_{m=0}^k|A(t_m,t_{m+1})|^{\frac{n-2}{n}} \le 3C_3e^{-\frac{2\boldsymbol{\mu}}{n}} \lc V(t_{k+2})+\frac{n}{2}V(t_{k+2})\rc \leq C_4e^{-\frac{2\boldsymbol{\mu}}{n}} V(t_{k+1}),   \label{T206e5}
\end{align}
where $C_4=(6+3n)C_3=C_4(n)$. 
Now we choose 
\begin{align}
  \ep_0 \coloneqq (2C_1)^{-1}(2C_4)^{-\frac{n}{2}}.     \label{eqn:PK29_8}
\end{align}
Clearly, $\epsilon_0=\epsilon_0(n)$. Then it follows from \eqref{T206e4} that 
\begin{align*}
2C_4e^{-\frac{2\boldsymbol{\mu}}{n}}|A(t_m,t_{m+1})| \le |A(t_m,t_{m+1})|^{\frac{n-2}{n}}, \quad \forall m \in \{ 1, 2, \cdots, k\}. 
\end{align*}
It is clear from~\eqref{T206e5} that 
\begin{align*}
2C_4e^{-\frac{2\boldsymbol{\mu}}{n}}(V(t_{k+1})-V(r)) \le C_4e^{-\frac{2\boldsymbol{\mu}}{n}} V(t_{k+1})
\end{align*}
and hence
\begin{align*}
V(t_{k+1}) \le 2V(r) \le 2\ep_0e^{\boldsymbol{\mu}}r \le 2\ep_0e^{\boldsymbol{\mu}}t_{k+1}.
\end{align*}
Therefore, the induction is complete and (\ref{eqn:PK29_4}) is proved.  
By the arbitrary choice of $m$, the total volume of the Ricci shrinker is finite, which contradicts Lemma 6.2 of~\cite{MW12}(See also Theorem 3.1 of~\cite{Cao092} by Cao-Zhu). 
Therefore, the proof of (\ref{eqn:PK29_3})  is established by this contradiction. 
Consequently, (\ref{eqn:PK29_2}) holds by Lemma~\ref{L100}.   
Note that $r_0$ and $\epsilon_0$ are defined in (\ref{eqn:PK29_7}) and (\ref{eqn:PK29_8}). 
Both of them can be calculated explicitly. 
\end{proof}

\begin{remark}
In \cite[Theorem $6.1$]{MW12}, the authors have obtained a weaker lower bound
\begin{align*}
|B(p,r)| \ge Ce^{c \boldsymbol{\mu}}r
\end{align*}
for two constants $C>0$ and $c>1$ depending only on $n$. 

\label{rmk:PK29_1}
\end{remark}

We are now ready to prove the improved no-local-collapsing, i.e., Theorem~\ref{thm:B}. 

\begin{proof}[Proof of Theorem~\ref{thm:B}]
It follows from Lemma~\ref{L101} and Proposition~\ref{prn:PK29_1} that
\begin{align*}
      \epsilon_0 r \leq   |B(p,r)|e^{-\boldsymbol{\mu}} \leq C r^n.        
\end{align*}
By Lemma~\ref{lem:unitvol}, we know $e^{\boldsymbol{\mu}}|B(p,1)|^{-1}$ is uniformly bounded from above and from below.  
Multiplying each term of the above inequality by $e^{\boldsymbol{\mu}}|B(p,1)|^{-1}$ and adjusting $C$ if necessary, we arrive at 
\begin{align*}
      \frac{1}{C}r  \leq \frac{|B(p,r)|}{|B(p,q)|} \leq Cr^n,
\end{align*}
which is nothing but (\ref{eqn:PL05_2a}).   We proceed to prove (\ref{eqn:PL05_2b}).  Recall that $q \in \partial B(p,r)$ and $\rho \in (0, r^{-1})$ for some $r>1$.  Triangle inequality implies that
\begin{align*}
   B(q,\rho) \subset B(p, 2r). 
\end{align*}
It follows from Lemma~\ref{L100} that $f \leq C r^2$ for some $C=C(n)$ on $B(p, 2r)$.   Since $R+|\nabla f|^2=f$ and $R \geq 0$, it follows that $R \leq Cr^2$ on $B(p, 2r)$.
In particular, we have $R\rho^2 \leq Rr^{-2} \leq C(n)$ on $B(q,\rho)$.  Consequently, we can apply Theorem~\ref{thm:PL06_2} on the ball $B(q,\rho)$ to obtain (\ref{eqn:PL05_2b}). 
\end{proof}

\section{The pseudolocality theorems}
\label{sec:pseudo}

In this section, we prove the pseudo-locality theorems on Ricci shrinker and discuss their applications.

Based on the Harnack estimate, following a classical point-picking, or maximum principle argument, we are able to obtain the following pseudo-locality theorem.

\begin{theorem}
There exist positive numbers $\epsilon_0=\epsilon_0(n)$ and $\delta_0=\delta_0(n)$ with the following properties. 

Let $\{(M^n, g(t)), -\infty < t < 1\}$ be the Ricci flow induced from a Ricci shrinker $(M^n, p, g)$.   Suppose  $t_0 \in (-\infty, 1)$ and $B_{g(t_0)}(x,r) \subset M$ is a geodesic ball satisfying
  \begin{align}
       \boldsymbol{\nu} (B_{g(t_0)}(x,r), g(t_0), r^2) > -\delta_0.   \label{eqn:PL03_0}
  \end{align}
Then for each $t \in (t_0,  \min\{t_0+\epsilon_0^2 r^2, 1\})$ and $y \in B_{g(t)}(x, 0.5r)$, we have
\begin{subequations}
\begin{align}[left = \empheqlbrace \,]
&|Rm|(y, t) \leq  (t-t_0)^{-1};  \label{eqn:PL03_1a}\\
&\inf_{0<\rho<\sqrt{t-t_0}} |B_{g(t)}(y, \rho)|\rho^{-n} \geq  \frac{1}{2} \omega_n.    \label{eqn:PL03_1b}
\end{align}
\label{eqn:PL03_1}
\end{subequations}
\\
\noindent
\label{thm:PL03_1}  
\end{theorem}

The statement in Theorem~\ref{thm:PL03_1} is a slight improvement of Theorem 10.1 of~\cite{Pe1}.   The basic idea of the proof is already contained in  Proposition 3.1 and Proposition 3.2 of Tian-Wang~\cite{TiWa}.
Note that the isoperimetric constant estimate in Peleman's statement is only used to (cf. Lemma 3.5 of~\cite{BW17A}) estimate the local entropy (i.e., (\ref{eqn:PL24_2}) and (\ref{eqn:PL24_3})) $\boldsymbol{\nu} (B_{g(t_0)}(x,r), g(t_0), r^2)$.
The statement (\ref{eqn:PL03_0}) seems to be more straightforward.  The conclusion (\ref{eqn:PL03_1}) follows from a standard point-picking argument,  
whenever the differential Harnack estimate,  i.e., Theorem \ref{T401} holds. 
More details can be found in~\cite[Section $30$]{KL08}, \cite[Section $8$]{CTY}, ~\cite[Chapter $21$]{CCGG}, or~\cite{BW17B}.

As Ricci shrinker Ricci flows are self-similar, we can improve the estimate (\ref{eqn:PL03_1}) by the following property. 

\begin{theorem}
Suppose $(M^n, p, g, f)$ is a Ricci shrinker, $B=B(q,r) \subset M$ is a geodesic ball satisfying
  \begin{align}
       \boldsymbol{\nu} (B, g, r^2) > -\delta_0.   \label{eqn:PL06_7}
  \end{align}
Then we have 
\begin{align}
   \sup_{x \in B(q, 0.5\epsilon_0 r)} |Rm|(x)  \leq \max\{ 1, \epsilon_0 D r\} \cdot (\epsilon_0 r)^{-2},   \label{eqn:PL15_0}
\end{align}
where $D=d(p,q)+\sqrt{2n}$. 

\label{thm:PL16_1}  
\end{theorem}

\begin{proof}[Proof of Theorem~\ref{thm:PL16_1}]

We fix $\xi \leq \epsilon_0$ a small positive number, whose value will be determined later (i.e., (\ref{eqn:PL20_6})). 
We set 
\begin{align}
  t \coloneqq -(\xi r)^2, \quad \tilde q=(\psi^{t})^{-1}(q), \quad D \coloneqq d(p,q)+\sqrt{2n},   \label{eqn:PL20_2}
\end{align}
where $\psi^{s}$ is the diffeormorphism (i.e., (\ref{E102})) generated by $\frac{\nabla f}{1-s}$.

\begin{claim}
By choosing $\xi$ properly, we have 
\begin{align}
      d(q, \tilde{q}) \leq  \frac{ \epsilon_0 r}{2}.    \label{eqn:PL20_5}
\end{align}
\label{clm:PL20_1}
\end{claim}

By (\ref{E102}) and (\ref{E101}),  along the flow line $\psi^s(\tilde q)$ where $s$ goes from $t$ to $0$, we compute
\begin{align}\label{eqTI52}
d(q,\tilde q) \le \int_t^0 \frac{|\na f|(\psi^s(\tilde q))}{1-s}\,ds \le \int_t^0 \frac{\sqrt{f(\psi^s(\tilde q))}}{1-s}\,ds.
\end{align}
From the definition of $\psi^s$, we have
\begin{align*}
\frac{d}{ds} f(\psi^s(\tilde q))=\frac{|\na f|^2(\psi^s(\tilde q))}{1-s} \le \frac{f(\psi^s(\tilde q))}{1-s}. 
\end{align*}
For each $s \ge t=-(\xi r)^2$, the integration of the above inequality yields that
\begin{align*}
f(\psi^s(\tilde q)) \le \frac{1-t}{1-s}f(\psi^t(\tilde q))=\frac{1-t}{1-s}f(q)  \leq  \frac{1-t}{1-s} \cdot \frac{D^2}{4}, 
\end{align*}
where we applied (\ref{eqn:PL20_1}) in the last step. 
Therefore, it follows from (\ref{eqTI52}) that
\begin{align*}
 d(q,\tilde q)\le& D\sqrt{1-t}\int_t^0 \frac{1}{2}(1-s)^{-3/2}\,ds=D \left( \sqrt{1-t}-1\right). 
\end{align*}
Plugging the fact that $t=-(\xi r)^2$ into the above inequality, we arrive at
\begin{align}
\label{eqTI53}
  d(q,\tilde q) \leq D \left( \sqrt{1+(\xi r)^2}-1\right). 
\end{align}
Now we define $\xi$ as follows.
\begin{align}
\xi \coloneqq
\begin{cases}
\epsilon_0, & \textrm{if} \; Dr \leq \epsilon_0^{-1};\\
\sqrt{\frac{\epsilon_0}{Dr}} & \textrm{if} \; Dr > \epsilon_0^{-1}. 
\end{cases}
\label{eqn:PL20_6}
\end{align}
Therefore, if $Dr \leq \epsilon_0^{-1}$, it follows from (\ref{eqTI53}) that
\begin{align*}
d(q,\tilde q)\le D \left( \sqrt{1+(\epsilon_0 r)^2}-1\right)
\le \frac{D(\epsilon_0r)^2}{2}\le \frac{\epsilon_0 r}{2}.
\end{align*}
If $Dr >\epsilon_0^{-1}$, it also follows from (\ref{eqTI53}) that
\begin{align*}
d(q,\tilde q)\le D \left( \sqrt{1+(\xi r)^2}-1\right) \leq \frac{D}{2} \cdot (\xi r)^2=\frac{\epsilon_0 r}{2}. 
\end{align*}
Therefore, no matter what the value of $r$ is, we always have (\ref{eqn:PL20_5}). 
The proof of the Claim is complete.

We proceed to prove (\ref{eqn:PL15_0}). 
Since $g(t)=(1-t)(\psi^t)^*g$, it is clear that
\begin{align*}
\psi^t\lc B_t \left(\tilde q,\sqrt{1-t}\,r \right)\rc=B(q,r).
\end{align*}
It follows from the scaling property of $\boldsymbol{\nu}$ that
\begin{align*}
   \boldsymbol{\nu} \left(B_{g(t)} \left(\tilde q,\sqrt{1-t}\,r \right), g(t), (1-t)r^2 \right)=\boldsymbol{\nu} (B, g, r^2)>-\delta_0. 
\end{align*}
Therefore,  we can apply Theorem~\ref{thm:PL03_1}.   For each $s \in (t, \min\{t+(\epsilon_0r)^2, 1\}]$ and $x \in B_{g(s)}(\tilde q, 0.5 r)$, we have
\begin{subequations}
\begin{align}[left = \empheqlbrace \,]
&|Rm|(x, s) \leq  (s-t)^{-1};  \label{eqn:PL20_3a}\\
&\inf_{0<\rho<\sqrt{s-t}} |B_{g(s)}(x, \rho)|\rho^{-n} \geq  \frac{1}{2} \omega_n.    \label{eqn:PL20_3b}
\end{align}
\label{eqn:PL20_3}
\end{subequations}
\\
\noindent
In particular, we can choose $s=0$. 
Since $g=g(0)$,  for each $x \in B(\tilde{q}, 0.5r)$,  we obtain
\begin{subequations}
\begin{align}[left = \empheqlbrace \,]
&|Rm|(x) \leq  (\xi r)^{-2};  \label{eqn:PL20_4a}\\
&\inf_{0<\rho<\xi r} |B_{g}(x, \rho)|\rho^{-n} \geq  \frac{1}{2} \omega_n.    \label{eqn:PL20_4b}
\end{align}
\label{eqn:PL20_4}
\end{subequations}
\\
\noindent
Note that $B(q, 0.5 \epsilon_0 r) \in B(\tilde{q}, \epsilon_0 r) \subset B(\tilde{q}, 0.5r)$ by (\ref{eqn:PL20_5}).
Plugging (\ref{eqn:PL20_6}) into (\ref{eqn:PL20_4a}), we obtain (\ref{eqn:PL15_0}). 
\end{proof}

 Now we apply Theorem~\ref{thm:PL03_1} and Theorem~\ref{thm:PL16_1} to study the geometric properties of $(M, g)$ in terms of $\boldsymbol{\mu}$.
 In particular, we are ready to finish the proof of Theorem~\ref{thm:C}. 
 
 \begin{proof}[Proof of Theorem~\ref{thm:C}]
 We divide the proof into several steps.
 
 \textit{Step 1. The gap property (\ref{eqn:PK20_4}) holds.}

 It suffices to show that $\boldsymbol{\mu} \geq -\delta_0$ implies that $(M, g)$ is isometric to the Euclidean space.

 Following directly from its definition,  as $B(x,r) \subset M$, it is clear that
\begin{align*}
  \boldsymbol{\nu} (B(x,r), g, r^2)  \geq \boldsymbol{\nu} (M, g, r^2) =\boldsymbol{\nu} (g, r^2). 
\end{align*}
Combining the above inequality with the optimal Logarithmic Sobolev inequality, we obtain
\begin{align}
  \boldsymbol{\nu} (B(x,r), g, r^2)  \geq \boldsymbol{\mu}.   \label{eqn:PL03_3} 
\end{align}
Therefore, if $\boldsymbol{\mu} \geq -\delta_0$, then each ball $B(x,r)$ will satisfy the condition (\ref{eqn:PL06_7}).   
By choosing $r>>D$,  we can apply (\ref{eqn:PL15_0}) to obtain that
\begin{align*}
    |Rm|(x)  \leq \epsilon_0 D r \cdot (\epsilon_0 r)^{-2}=D \epsilon_0^{-1} r^{-1}.   
\end{align*}
Let $r \to \infty$, we obtain that $|Rm|(x) \equiv 0$.  By the arbitrary choice of $x$, we obtain that $|Rm| \equiv 0$. 
In particular, $Rc \equiv 0$.  Then the Ricci shrinker equation implies that $f_{ij}=\frac{g_{ij}}{2}$.  
Therefore,  $(M, g)$ is isometric to a metric cone which is also a smooth manifold.
This forces that $(M, g)$ is isometric to the standard Euclidean space $(\R^n, g_{E})$. Thus, the proof of (\ref{eqn:PK20_4}) is complete. 
 
 \textit{Step 2.  The inequality (\ref{eqn:PL05_4}) and (\ref{eqn:PL05_0}) imply the curvature and injectivity radius bound (\ref{eqn:PL05_1}).}
 
 Recall that  (\ref{eqn:PK20_4}) means $\boldsymbol{\mu}(g, 1)<-\delta_0$.  If (\ref{eqn:PL05_4}) holds, by continuity and monotonicity of $\boldsymbol{\mu}(g, \tau)$, it is clear that there exists some $\tau \in (0,1)$ such that
 \begin{align*}
     \boldsymbol{\mu}(g,\tau)=-\delta_0. 
 \end{align*} 
 Then the $\tau_0$ in (\ref{eqn:PL05_1}) is well defined. Namely, $\tau_0$ is the largest $\tau \in (0, 1)$ such that the above equality holds. 
 It follows from the definition of $\tau_0$ and $\boldsymbol{\nu}$ that
 \begin{align}
  \boldsymbol{\nu}(g, \tau_0) = \boldsymbol{\mu}(g, \tau_0)=-\delta_0.     \label{eqn:PL05_5}
 \end{align}
 For each ball $B_{g(0)}(x,r) \subset M$, we know $ \boldsymbol{\nu}(B_{g(0)}(x,r), g, \tau_0) \geq -\delta_0$. 
 In particular, we can choose $r=\sqrt{\tau_0}$.  
Now we apply Theorem~\ref{thm:PL03_1} on the time slice $t_0=0$, with scale $\sqrt{\tau_0}$, to obtain that
\begin{align*}
     |Rm|(x,t)  \leq  t^{-1},  \quad \forall \; x \in M, \quad \forall \; t \in (0, \epsilon_0^2 r^2]. 
\end{align*}
In particular, we have
\begin{align*}
   \sup_{x \in M} |Rm|(x, \epsilon_0^2 \tau_0) \leq \epsilon_0^{-2} \tau_0^{-1}. 
\end{align*}
Up to rescaling, since $g(0)=g$, we arrive atl
\begin{align*}
  \sup_{x \in M} |Rm|_{g}(x) \leq \epsilon_0^{-2} \tau_0^{-1} (1-\epsilon_0^2 \tau_0)=\epsilon_0^{-2} \tau_0^{-1}-1< C(n) \tau_0^{-1},
\end{align*}
which is nothing but (\ref{eqn:PL05_1a}).  Plugging (\ref{eqn:PL05_5}) into (\ref{eqn:PL06_2}) of Theorem~\ref{thm:PL06_1}, we obtain that each geodesic ball $B(\cdot, \sqrt{\tau_0})$
has volume bounded below by $c(n) \tau_0^{\frac{n}{2}}$.  Therefore, the injectivity radius estimate of Cheeger-Gromov-Taylor~\cite{CGT} applies and we arrive at (\ref{eqn:PL05_1b}). 
The proof of (\ref{eqn:PL05_1}) is complete.

 \textit{Step 3. The bounded geometry estimate (\ref{eqn:PL05_1}) implies the equality (\ref{eqn:PL05_3}), i.e., $\displaystyle \lim_{\tau \to 0^+}\boldsymbol{\mu}(g,\tau) =0$. }
 
We shall argue in the way similar to that in Theorem 1.1 of~\cite{Z11}, with more details on the regularity estimate.  

Assume otherwise that there exists a sequence $\tau_i \to 0^+$ such that
\begin{align}
   \lim_{i \to \infty}\boldsymbol{\mu}(g,\tau_i)=\boldsymbol{\mu}_{\infty}<0. \label{eqn:PL21_8}
\end{align}
If we set $g_i=\tau_i^{-1}g$, then all metrics $g_i$ have uniformly bounded geometry.
More precisely, there exist positive constants $K$ and $v_0$ such that
\begin{subequations}
\begin{align}[left = \empheqlbrace \,]
& |Rm_i| \le K\tau_i, \label{eqn:PL20_7a}\\
& |B(q,r)|_{g_i} \ge v_0 r^n(1+\tau_iKr^2)^{-\frac{n}{2}}.    \label{eqn:PL20_7b}
\end{align}
\label{eqn:PL20_7}
\end{subequations}
\\
\noindent
Notice that for any $i$, there exists a large domain 
\begin{align}
  B_i \coloneqq  \left \{x \left|2\sqrt{f} \le r_i \right. \right\}    \label{eqn:PL20_9}
\end{align}
for some large $r_i>>1$ such that 
 \begin{align}\label{eq:equality4}
\boldsymbol{\mu}(B_i, g_i,1)-\boldsymbol{\mu}(g_i,1)=\boldsymbol{\mu}(B_i, g_i,1)-\boldsymbol{\mu}(g,\tau_i)<i^{-1}.
\end{align}
The geometry bound (\ref{eqn:PL20_7}) actually implies higher order derivatives of curvatures and $\sqrt{f}$ are also uniformly bounded (cf. Section 4 of~\cite{LLW18}).
Therefore,  it is not hard to see that $\partial B_i$ is smooth. 
All the covariant derivatives of second fundamental forms of $\partial B_i$ are bounded independent of $i$.

It follows from \cite{Ro81} that a minimizer $u_i$ of $\boldsymbol{\mu}(B_i, g_i,1)$ exists. More precisely, $u_i \in W_0^{1,2}(B_i)$ is a positive smooth function on $B_i$  satisfying the normalization condition
 \begin{align}\label{eq:equality5}
\int_{B_i} u_i^2\,dV_i=1
\end{align}
and solve the Dirichelet problem
\begin{subequations}
\begin{align}[left = \empheqlbrace \,]
& -4\Delta_i u_i+R_iu_i-2u_i\log u_i-\lambda_i u_i=0, \quad \text{in}\quad B_i; \label{eqn:PL20_8a}\\
&u_i=0, \quad \text{on}\quad  \partial B_i.   \label{eqn:PL20_8b}
\end{align}
\label{eqn:PL20_8}
\end{subequations}
\\
\noindent
Here $dV_i$, $\Delta_i$ and $R_i$ denote the volume form, Laplacian operator and scalar curvature with respect to $g_i$ respectively.  
The number $\lambda_i$ is defined by 
\begin{align*}
\lambda_i \coloneqq n+\frac{n}{2}\log(4\pi)+\boldsymbol{\mu}(B_i, g_i,1).
\end{align*}
Recall that $\displaystyle \lim_{\tau \to 0^{+}} \boldsymbol{\mu}(g,\tau) \leq 0$ by (\ref{eqn:PL21_10}). 
Then it follows from (\ref{eq:equality4})  that $\lambda_i$ is uniformly bounded.  
Since curvature is uniformly bounded, the classical $L^2$-Sobolev constant of $(B_i, g_i)$ is uniformly bounded. 
In light of (\ref{eqn:PL20_8}), the Moser iteration then implies $\norm{u_i}{C^0}$ is uniformly bounded, see \cite[Lemma $2.1$(a)]{Z11} or the proof of Proposition 3.1 of~\cite{TiWa}. Then it follows from \cite[Corollary $8.36$]{GT01} that $\|u_i\|_{C^{1,\frac{1}{2}}(\bar B_i)}$ are uniformly bounded. Since all $\partial B_i$ have uniformly higher regularities, the bootstrapping, see \cite[Theorem $6.19$]{GT01},  shows that $\|u_i\|_{C^{k,\frac{1}{2}}(\bar B_i)}$ are uniformly bounded for any $k \ge 2$.

Let $q_i$ be a point where $u_i$ achieves maximum value in $B_i$.  
By \eqref{eqn:PL20_8},  at $q_i$ we have
 \begin{align*}
R_iu_i-2u_i\log u_i-\lambda_i u_i \le 0, 
\end{align*}
whence we derive
\begin{align}\label{eq:equality8}
u_i(q_i) \ge \exp\lc\frac{R_i-\lambda_i}{2} \rc \ge c_0
\end{align}
for some uniform constant $c_0$. 

In light of (\ref{eqn:PL20_7}) and the discussion below (\ref{eq:equality4}), we know that $(M^n, q_i, g_i)$ subconverges to Euclidean space $(\R^n, 0, g_{E})$ in
$C^{\infty}$-Cheeger-Gromov topology.   The set $B_i$ converges to a limit set $B_{\infty}$.
If $d(q_i, \partial B_i) \to \infty$, then $B_{\infty}=\R^n$.   Otherwise, by the estimate of second fundamental form and its covariant derivatives, $\partial B_i$ converge to a smooth
$(n-1)$-dimensional set $\partial B_{\infty}$.  In light of the uniform bound of $\norm{u_i}{C^{k,\frac{1}{2}}}$ and the uniform regularity of $\partial B_i$, by taking subsequence if necessary, we can assume that $u_i$ converges in smooth topology to a smooth function $u_{\infty} \in C^{\infty}(\bar B_{\infty})$. Furthermore, $u_{\infty} \equiv 0$ on $\partial B_{\infty}$.

In view of (\ref{eq:equality8}), the convergence process implies that
\begin{align}
  0<c^2 =  \int_{B_{\infty}} u_{\infty}^2 dV_{\infty} \leq 1.     \label{eqn:PL21_9}
\end{align}
Furthermore, we have on $B_{\infty}$ that 
 \begin{align}\label{eq:equality10}
-4\Delta_{g_E} u_{\infty}-2u_{\infty}\log u_{\infty}-\lambda_{\infty} u_{\infty} =0,
\end{align}
where $\lambda_{\infty}=n+\frac{n}{2}\log(4\pi)+\boldsymbol{\mu}_{\infty}$.  Let $\tilde{u}=c^{-1}u_{\infty}$. Then $\int_{B_{\infty}} \tilde{u}^2 dV_{\infty}=1$.  The above equation becomes
\begin{align*}
  -4\Delta_{g_E} \tilde{u}-2\tilde{u} \log \tilde{u} -\left(n+\frac{n}{2}\log(4\pi)+\boldsymbol{\mu}_{\infty}+2\log c\right)\tilde{u}=0.
\end{align*}
Since $c \in (0, 1)$ by (\ref{eqn:PL21_9}) and $\boldsymbol{\mu}_{\infty}<0$ by (\ref{eqn:PL21_8}), then an integration by parts shows that
\begin{align*}
\boldsymbol{\mu}(g_E,1) \le  \overline{\mathcal{W}}(g_{E}, \tilde{u}, 1)=\boldsymbol{\mu}_{\infty}+2\log c<0,
\end{align*}
which is a contradiction.
So we finish the proof of Step 3.

 \textit{Step 4.  The three properties are equivalent.}
 
 By Step 2, it is clear that $(c) \Rightarrow (a)$.  Then Step 3 means that $(a) \Rightarrow (b)$. 
 It is obvious that $(b) \Rightarrow (c)$.   Therefore, we obtained the equivalence of properties (a), (b) and (c) in Theorem~\ref{thm:C}.
 The proof of the Theorem is complete.
 \end{proof}

\begin{cor}\label{T601}
There exists a small positive number $\epsilon=\epsilon(n)>0$ such that for any nonflat Ricci shrinker $(M^n, p, g, f)$, we have
\begin{align}
  d_{PGH} \left\{ (M^n,p,g), (\R^n, 0,g_{E})\right\}>\epsilon.    \label{eqn:PL06_8}
\end{align} 
\end{cor}

\begin{proof}
We argue by contradiction. 

If (\ref{eqn:PL06_8}) were wrong,  then we can have a sequence of nonflat Ricci shrinkers $(M_i, p_i, g_i)$ such that
\begin{align*}
   d_{PGH} \left( (M_i, p_i, g_i), \left(\R^n,0, g_{E} \right)\right) \to 0. 
\end{align*}
By Proposition 5.8 of~\cite{LLW18},   it is clear that  $\boldsymbol{\mu}_i = \boldsymbol{\mu} (M_i, p_i, g_i)$ is uniformly bounded from below.
Using Theorem 1.1 of~\cite{LLW18}, the above convergence can be improved to be in the smooth topology
\begin{align*}
    (M_i, p_i, g_i)  \longright{C^{\infty}-Cheeger-Gromov} \left(\R^n, 0, g_{E}\right). 
\end{align*}
It is not hard to see that $\boldsymbol{\mu}$ is continuous with respect to the above convergence (cf. Theorem 1.2(c) of~\cite{LLW18}).
Therefore,  we have
\begin{align*}
   \boldsymbol{\mu}_i = \boldsymbol{\mu} (M_i, p_i, g_i)  \to \boldsymbol{\mu} (\R^n, 0, g_{E}) =0. 
\end{align*}
It follows that $\boldsymbol{\mu}_i >-\delta_0$ for large $i$.  
Therefore, each $(M_i, g_i)$ is isometric to Euclidean space by Theorem~\ref{thm:C}. 
This contradicts our choice of $(M_i, g_i)$.  The proof of (\ref{eqn:PL06_8}) is established by this contradiction. 
\end{proof}

\begin{cor}
\label{C501b}
Let $(M^n,g,f)$ be a Ricci shrinker and let $q \in M$ be a point such that 
\begin{align*}
\nu(B(q,\ep_0^{-1}),g,\ep_0^{-2})>-\delta_0.
\end{align*}
Then there exist a positive constant $C=C(n)$ such that
\begin{align*}
|Rm|(\psi^t(x)) \le CD(1-t)\le CD\frac{f(x)}{f(\psi^t(x))}
\end{align*}
for any $x \in B(q,\frac{1}{2}e^{-CD}D^{-\frac{1}{2}})$ and $t \in [0,1)$, where $D=d(p,q)+\sqrt{2n}$.
\end{cor}

\begin{proof}
By the assumption, it follows from Theorem \ref{thm:PL03_1} by choosing $r=\ep_0^{-1}$ that
\begin{align}\label{eqt5011}
|Rm|(x,t) \le \frac{1}{t}
\end{align}
for any $t \in (0,1)$ and $d_{g(t)}(q,x) \le \frac{1}{2}\ep_0^{-1}$. In addition, from Theorem \ref{thm:PL16_1} we have
\begin{align}\label{eqt5012}
|Rm|(x) \le D
\end{align}
for any $x \in B(q,\frac{1}{2})$. From \eqref{eqt5011}, \eqref{eqt5012} and \cite[Theorem $3.1$]{CBL07} that there exist a positive constant $C=C(n)$ such that for any $x \in B_t(q,\frac{1}{2}D^{-\frac{1}{2}})$,
\begin{align}\label{eqt5012a}
|Rm|(x,t) \le CD.
\end{align}
From \eqref{eqt5012a}, it is easy to see by comparing the distances that
\begin{align}\label{eqt5012b}
B \lc q, \frac{1}{2}e^{-CD}D^{-\frac{1}{2}} \rc \subset B_t \lc q,\frac{1}{2}D^{-\frac{1}{2}} \rc
\end{align}
for any $t \in [0,1)$. 

Therefore, for any $x \in B(q,\frac{1}{2}e^{-CD}D^{-\frac{1}{2}})$,
\begin{align}
|Rm|(\psi^t(x))=(1-t)|Rm|(x,t)\le CD(1-t).\label{eqt5013}
\end{align}

Along the flow line of $\psi^t(x)$,
\begin{align}
    \frac{d}{dt}f(\psi^t(x))=\frac{|\nabla f|^2(\psi^t(x))}{1-t}\le \frac{f(\psi^t(x))}{1-t},
\end{align}
and hence by solving the corresponding ODE,
\begin{align}
f(\psi^t(x)) \le \frac{f(x)}{1-t}.\label{E501b}
\end{align}
Combining \eqref{eqt5013} and \eqref{E501b}, the conclusion follows.
\end{proof}

Since $f$ is almost $\frac{d^2}{4}$ by Lemma \ref{L100}, Corollary \ref{C501b} shows that the curvature is quadratically decaying along the flow line. Next we prove that if there exists a tubular neighborhoold of some level set of $f$ whose isoperimetric constant is almost Euclidean, then globally the curvature is quadratically decaying.

\begin{cor}\label{C501c}
For any Ricc shrinker $(M^n,g,f)$, if there exists an $a>0$ such that for any $x \in f^{-1}(a)$, 
\begin{align*}
\nu(B(x,\ep_0^{-1}),g,\ep_0^{-2})>-\delta_0,
\end{align*}
then the curvature is quadratically decaying and each end has a unique smooth tangent cone at infinity.
\end{cor}

\begin{proof}
We can assume that $(M,f)$ is nonflat, otherwise there is nothing to prove. Now we reparametrize $\psi^t$ by defining for any $s \in (-\infty,\infty)$
\begin{align*}
\tilde \psi^s=\psi^{1-e^{-s}}.
\end{align*}
It is clear from the definition of $\psi^t$ that 
\begin{align*}
\frac{d}{ds}\tilde \psi^s(x)=\na f(\tilde \psi^s(x)).
\end{align*}
In other words, $\tilde \psi^s$ is the one-parameter group of diffeomorphisms generated by $\na f$. Now we set
\begin{align*}
\ep_1=\ep_1(a,n)=\frac{1}{2}e^{-CD_1}D_1^{-\frac{1}{2}},
\end{align*}
where $D_1=2\sqrt a+5n+\sqrt{2n}+4$.

We claim that any $x \in T_{\ep_1}(f^{-1}(a))\coloneqq \bigcup_{q\in f^{-1}(a)}B(q,\ep_1)$ is not a stationary point of $\tilde \psi^s$. Otherwise, it follows from Corollary \ref{C501b} that
\begin{align*}
|Rm|(x)=|Rm|(\tilde \psi^s(x))\le CD_1e^{-s}
\end{align*}
for any $s \ge 0$. However, when $s \to \infty$, $|Rm|(x)=0$ and this contradicts our nonflatness assumption. 

Now we choose $c<a<d$ such that for any $x \in \partial T_{\frac{\ep_1}{2}}(f^{-1}(a))$, either $f(x) \le c$ or $f(x) \ge d$. By continuity, there exists a positive constant $\ep \ll \ep_1$ such that for any $x \in T_{\ep}(f^{-1}(a))$, $f(x) \in (c+\ep,d-\ep)$. We set $U \coloneqq T_{\ep}(f^{-1}(a))$ and claim that for any $y \in U$, there exists an $x \in f^{-1}(a)$ such that $\tilde \psi^{s}(x)=y$ for some $s$. If $f(y)=a$, then the claim is obvious. If $f(y) <a$, we consider the flow line $\tilde \psi^{s}(y)$ for $s \ge 0$. Notice that by the definition of $\tilde \psi^{s}$,
\begin{align*}
\frac{d}{ds}f(\tilde \psi^{s}(y))=|\na f|^2(\tilde \psi^{s}) \ge 0.
\end{align*}
Therefore, by the local compactness and our previous no stationary argument, the flow will continue and along the flow $f$ is strictly increasing as long as $\tilde \psi^{s}(y)$ stays in $T_{\ep_1}(f^{-1}(a))$. We set $s_0$ to be the first time such that $\tilde \psi^{s}(y)$ reaches $\partial T_{\frac{\ep_1}{2}}(f^{-1}(a))$. In particular, $f(\tilde \psi^{s}(y)) \le c$ or $f(\tilde \psi^{s}(y)) \ge d$. Since $f(y) \in (c+\ep,d-\ep)$, it must be $f(\tilde \psi^{s}(y)) \ge d$. As $f(y) <a<f(\tilde \psi^{s_0}(y))$, there exists an $s \in (0,s_0)$ such that $f(\tilde \psi^{s}(y))=a$ by continuity. Therefore, if we set $x=\tilde \psi^{s}(y) \in f^{-1}(a)$, then $\tilde \psi^{-s}(x)=y$ and the claim follows. Similarly, for the case $f(y)>a$, the claim is also true.

Next we prove that for any $y$ such that $f(y)>a$, there exists an $x \in U$ such that $\tilde \psi^{s}(x)=y$ for some $s$. Fix such $y$, we choose any curve $\{\gamma (z):\,z\in [0,1]\}$ such that $\gamma (0)=p$ and $\gamma (1)=y$. In particular, since $p$ is the minimum point of $f$, there exists a $z_0 \in [0,1)$ such that $\gamma (z_0) \in f^{-1}(a)$ and for all $z \in (z_0,1]$, $f(\gamma(z))>a$. Now we define $I \subset [z_0,1]$ such that $z \in I$ if and only if there exists an $x \in U$ such that $\tilde \psi^{s}(x)=\gamma(z)$ for some $s$. In particular, $I$ is not empty as $z_0 \in I$. It is clear that $I$ is open, since $U$ is open. Now we prove the closedness of $I$. For a sequence $z_i \in I$ such that $z_i \to z_{\infty} \in [z_0,1]$, $f(z_i)>a$ if $i$ is sufficiently large. By our definition of $I$ and the claim with its proof, there exists $x_i \in f^{-1}(a)$ and $s_i >0$ such that $\tilde \psi^{s_i}(x_i)=\gamma(z_i)$. Note that $s_i$ must be bounded. Indeed, by Corollary \ref{C501b},
\begin{align*}
|Rm|(\gamma(z_i))=|Rm|(\tilde \psi^{s_i}(x_i))\le CD_1e^{-s_i}.
\end{align*}
If $s_i \to \infty$, then it forces $|Rm|(\gamma(z_{\infty}))=0$ and this is a contradiction. By compactness and taking the subsequence, there exist $x_{\infty} \in f^{-1}(a)$ and $s_{\infty} \ge 0$ such that $x_i \to x_{\infty}$ and $s_i \to s_{\infty}$. By continuity, $\tilde \psi^{s_{\infty}}(x_i)=\gamma(z_{\infty})$. To summarize, $I=[z_0,1]$ and in particular, $\tilde \psi^{s}(x)=\gamma(1)=y$ for some $x \in U$ and $s \in \R$. By the claim again, we have proved that for any $y$ with $f(y) \ge a$, there exists an $x \in f^{-1}(a)$ such that $\psi^s(x)=y$ for some $s \ge 0$. 

Therefore, for any point $y$ outside the compact set $\{f \le a\}$, it follows from Corollary \ref{C501b} that
\begin{align} \label{eq:quad}
|Rm|(y) \le \frac{CD_1a}{f(y)} \le C\frac{ \max\{1,a^{\frac{3}{2}}\}}{f(y)}.
\end{align}

See Figure~\ref{fig:flowline} for intuition in the case $a>1$. 

 \begin{figure}[h]
    \begin{center}
     \psfrag{A}[c][c]{$\textcolor{red}{f=a}$}
      \psfrag{B}[c][c]{$\textcolor{green}{f=b}$}
      \psfrag{C}[c][c]{\textcolor{blue}{flow line of $\nabla f$}}
      \psfrag{D}[c][c]{$B(x, \epsilon_0^{-1})$}
      \psfrag{E}[c][c]{$\textcolor{green}{|Rm| \leq Ca^{\frac32} f^{-1}}$}
      \psfrag{F}[c][c]{$x$}
      \psfrag{G}[c][c]{$\boldsymbol{\nu}(B(x, \epsilon_0^{-1}), g, \epsilon_0^{-2})>-\delta_0$}
      \includegraphics[width=0.5\columnwidth]{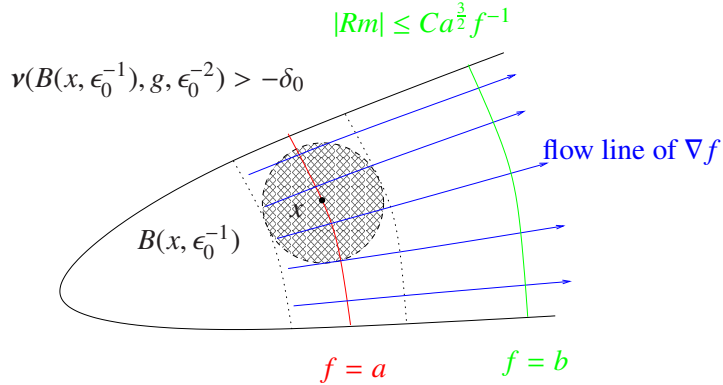}
    \end{center}
    \caption{The quadratic decay of curvature}
    \label{fig:flowline}
  \end{figure}

 In other words, the curvature is quadratically decaying. Since a Ricci shrinker can be regarded as an ancient Ricci flow, it follows from Shi's local estimates \cite{Shi89B} that 
\begin{align*}
|\na ^kRm|(y) \le \frac{C_k}{d^{k+2}(p,y)}
\end{align*}
for all $k=1,2,\cdots$. It follows immediately that any tangent cone at infinity must be smooth. Finally, the uniqueness follows from \cite[Theorem $2$]{CL15}, see also \cite[Lemma $A.3$]{KW15}.
\end{proof}

\begin{rem}
The proof of Corollary \ref{C501c} shows that the manifold $\{x\in M \mid\,f(x)  \ge a\}$ is diffeomorphic to $f^{-1}(a) \times [0,1)$.
\end{rem}

\section{Strong maximum principle for curvature operator}
\label{sec:L2}

The purpose of this section is to prove Theorem \ref{thm:D}.  We remind the readers that all constants $C$'s in this section depend only on the dimension $n$.

We first show an $L^2$-integral estimate of Riemannian curvature. 

\begin{theorem}
 Suppose $(M^n,p,g,f)$ is a Ricci shrinker satisfying $\boldsymbol{\mu} \geq -A$, and $\lambda$ is a positive number. 
 Then we have 
\begin{align}
  \int_{M} |Rm|^2e^{-\lambda f}\,dV \le I    \label{eqn:PK20_8}
\end{align}
for some $I=I(n,A,\lambda)<\infty$.  
\label{thm:PL21_1}
\end{theorem}

Theorem~\ref{thm:PL21_1} is the consequence of the improved no-local-collapsing theorem (i.e., Theorem~\ref{thm:B}), the local conformal transformation technique (cf. section 3 of~\cite{LLW18}), and the curvature estimate
of Jiang-Naber (i.e., ~\cite{JN16}).

\begin{lemma} 
For any Ricci shrinker $(M^n,p,g,f)$ and any constant $D >100n$, we have
\begin{align}
\int_{A(D,2D)} |Rc|^2e^{-f}\,dV \le   Ce^{\boldsymbol{\mu}}D^{n+2}e^{-D^2/5}       \label{eqn:PL21_1}
\end{align}
where $A(D, 2D)$ is the annulus $B(p,2D) \backslash B(p, D)$.  
\label{lma:PL21_0}
\end{lemma}

\begin{proof}
Fix a cutoff function $\psi$ on $\R$ such that $\psi=1$ on $[1,2]$ and $\psi=0$ outside $[\frac{1}{2},3]$. By defining $\eta(x)=\psi(\frac{d(p,x)}{D})$, we compute
\begin{align*} 
\int \eta^2|Rc|^2e^{-f} \,dV=&\int \eta^2\la \frac{g}{2}-\text{Hess}\,f,Rc \ra e^{-f} \,dV \\
=& \int \lc \frac{1}{2}\eta^2 R+2\eta Rc(\na \eta,\na f) \rc e^{-f} \,dV \\
\le & \int \lc \frac{1}{2}\eta^2 R+\frac{1}{2}\eta^2|Rc|^2+2|\na \eta|^2|\na f|^2 \rc e^{-f} \,dV 
\end{align*}
where for the second line we have used $\text{div}(Rc\,e^{-f})=0$. Consequently, by Lemma~\ref{L100} and Lemma~\ref{L101}, we have
\begin{align*} 
\int \eta^2|Rc|^2e^{-f} \,dV \le& \int \lc\eta^2 R+4|\na \eta|^2|\na f|^2\rc e^{-f} \,dV 
\le C\int_{A(D/2,3D)} fe^{-f}\,dV. 
\end{align*}
Plugging the estimates in Lemma~\ref{L100} and Lemma~\ref{L101} into the above inequality, we arrive at (\ref{eqn:PL21_1}). 
\end{proof}

In the proof of Lemma~\ref{lma:PL21_0}, if we choose $\psi$ such that $\psi=1$ on $(-\infty,1]$ and $\psi=0$ on $[2,\infty)$,  then a similar argument shows the following Lemma.
\begin{lemma} 
For any Ricci shrinker $(M^n,p,g,f)$, we have
\begin{align}
\int |Rc|^2e^{-f} \,dV \le Ce^{\boldsymbol{\mu}}.
\label{eqn:PL21_2}
\end{align} 
\label{lma:PL21_1}
\end{lemma}

The details of the proof of Lemma~\ref{lma:PL21_1} is almost identical to that of Lemma~\ref{lma:PL21_0}. So we leave it to interested readers. 
Note that  Lemma~\ref{lma:PL21_1} provides an explicit upper bound of~\cite[Theorem 1.1]{MS13}. 
Starting from Lemma~\ref{lma:PL21_0} and Lemma~\ref{lma:PL21_1}, we are ready to prove Theorem~\ref{thm:PL21_1}.

\begin{proof}[Proof of Theorem~\ref{thm:PL21_1}:]

 We only prove the case when $\lambda=1$.
 The general case is similar and is left to interested readers. 
 
 For any point $q \in M$ such that $d(p,q)=D>100n$, we set $r=\frac{1}{D}$, $\bar{f}=f-f(q)$, then under the conformal transformation $\bar{g}\coloneqq e^{-\frac{2\bar{f}}{n-2}}g$, we have
\begin{align}
    \overline{Rc}&=\frac{1}{n-2} \left\{ df \otimes df + (n-1-f) e^{\frac{2\bar{f}}{n-2}} \bar{g} \right\},  \label{eqn:con1} \\
\overline{Rm}&=e^{-\frac{2\bar{f}}{n-2}}\left[Rm+\frac{1}{n-2}\left(\frac{df \otimes df}{n-2}+\frac{g}{2}\lc1-\frac{|\nabla f|^2}{n-2}\rc-\frac{Rc}{n-2}\right)\KN g \right],\label{eqn:con2}
\end{align}
where the proof and the definition of the Kulkarni-Nomizu product $\KN$ can be found in \cite[Theorem 1.165]{Besse}. 
It follows from \cite[Lemma $3.5$]{LLW18} that
\begin{align}
   B_{\bar{g}}\left(q,   e^{-\frac{1}{n-2}} r \right) \subset B(q, r) \subset B_{\bar{g}}\left(q,  e^{\frac{1}{n-2}}r \right). \label{eqn:PE18_4}
\end{align}
Therefore, by the same proof as in \cite[Lemma $3.7$]{LLW18}, we have
  \begin{align}
   |\bar{f}| \le C \quad \text{and} \quad  \snorm{\overline{Rc}}{\bar{g}} \le CD^2 \quad \textrm{on} \quad B_{\bar{g}}\left(q, e^{\frac{1}{n-2}}r \right).  \label{eqn:PE06_7}
  \end{align}
Since $R \le CD^2$ on $B(q,r)$, it follows from Theorem \ref{thm:PL06_2} that $|B(q,r)| \ge Ce^{\boldsymbol{\mu}}r^n$ and hence
  \begin{align}
\left|B_{\bar{g}}\left(q,  e^{\frac{1}{n-2}}r \right)\right |_{\bar{g}} \ge Ce^{\boldsymbol{\mu}}r^n.   \label{eqn:PL21_3}
  \end{align}
One can also use Theorem~\ref{thm:B} to obtain the above estimate directly. 

By defining $\tilde{g}\coloneqq r^{-2}\bar{g}$, we have $|\widetilde{Rc}|_{\tilde{g}} \le C$ on $B_{\tilde{g}}(q,e^{\frac{1}{n-2}})$ and $|B_{\tilde{g}}(q,e^{\frac{1}{n-2}})|_{\tilde{g}} \ge Ce^{\boldsymbol{\mu}}$.
By shrinking balls to its half size if necessary,  it follows from~\cite[Theorem $1.6$]{JN16} that
\begin{align}
r^{4-n}\int_{B_{\bar{g}}(q,e^{\frac{1}{n-2}}r)} |\overline{Rm}|^2 \,dV_{\bar g}=\int_{B_{\tilde{g}}(q,e^{\frac{1}{n-2}})} |\widetilde{Rm}|^2 \,dV_{\tilde g} \le I_0 \label{eq:int1}
 \end{align}
for some constant $I_0=I_0(n,A)$.

From \eqref{eqn:con2}, we have on $B(q,r)$,
  \begin{align*}
|Rm|^2 \le C\lc |\overline{Rm}|^2+|\na f|^4+|Rc|^2 \rc \le C\lc |\overline{Rm}|^2+f^2+|Rc|^2 \rc.
  \end{align*}
Therefore, we have
 \begin{align*}
&\quad \int_{B(q,r)} |Rm|^2 e^{-f}\,dV\\
&\leq C \lc \int_{B_{\bar{g}}(q,e^{\frac{1}{n-2}}r)} |\overline{Rm}|^2 e^{-f}\,dV_{\bar g}+\int_{B(q,r)} f^2 e^{-f}\,dV+\int_{B(q,r)} |Rc|^2e^{-f}\,dV   \rc \\
&\leq Ce^{-\frac{D^2}{5}} \left(D^{4-n}I_0+D^{n+2}e^{\boldsymbol{\mu}} \right)
  \end{align*}
where we have used Lemma \ref{lma:PL21_0} and \eqref{eq:int1}.
Consequently,  there exists $I_1=I_1(n,A)$ such that
\begin{align}
\int_{B(q,r)} |Rm|^2 e^{-f}\,dV \le I_1D^{n+2}e^{-\frac{D^2}{5}} \label{eq:int2}.
 \end{align}

For any constant $D >100n$, we apply Vitali's lemma for the covering $\{B(q,\frac{1}{4D})\}_{q\in A(D,2D)}$. 
If we assume that $\{B(q_i,\frac{1}{4D})\}_{1 \le i \le k}$ is a maximal collection of mutually disjoint sets, 
then $\{B(q_i,\frac{1}{2D})\}_{1 \le i \le k}$ cover $A(D,2D)$.
It is clear from definition that
  \begin{align*}
\sum_{i=1}^k \left|B\left(q_i,\frac{1}{4D} \right) \right| \le |A(D,2D)| \le |B(p,2D)|.
  \end{align*}
By Lemma \ref{L101} and (\ref{eqn:PL21_3}), we obtain $k \le CD^{2n}$.
Combining \eqref{eq:int2} with the above inequality implies that
  \begin{align}
\int_{A(D,2D)} |Rm|^2 e^{-f}\,dV \le \sum_{i=1}^k \int_{B(q_i,\frac{1}{2D})} |Rm|^2 e^{-f}\,dV \le kI_1D^{n+2}e^{-\frac{D^2}{5}} \le CI_1D^{3n+2}e^{-\frac{D^2}{5}}.\label{eq:int3}
  \end{align}
Similarly, by exploiting Lemma~\ref{lma:PL21_1}, we have 
  \begin{align}
\int_{B(p,D_0)} |Rm|^2 e^{-f}\,dV \le I_2\label{eq:int4}
  \end{align}
where $D_0=100n$ and $I_2=I_2(n,A)$.

Now we set $D_i=2^i D_0$ and decompose the integral as
\begin{align*}
\int_{M} |Rm|^2 e^{-f}\,dV=\int_{B(p,D_0)} |Rm|^2 e^{-f}\,dV+\sum_{i \ge 0}\int_{A(D_i,2D_i)} |Rm|^2 e^{-f}\,dV
\end{align*}
Plugging (\ref{eq:int3}) and (\ref{eq:int4}) into the above equation, we arrive at  
  \begin{align*}
   \int_{M} |Rm|^2 e^{-f}\,dV \le & I_2+\sum_{i \ge 0}CI_1D_i^{3n+2}e^{-\frac{D_i^2}{5}}=I_2+ CI_1 \sum_{i \ge 0} 2^{i(3n+2)}D_0^{3n+2}e^{-\frac{4^i D_0^2}{5}} \coloneqq I.
  \end{align*}
Since both $I_1$ and $I_2$ depend only on $n$ and $A$, it is clear that $I$ relies only on $n$ and $A$ and we arrive at (\ref{eqn:PK20_8}). 
The proof of Theorem~\ref{thm:PL21_1} is complete. 
\end{proof}

From (\ref{eqn:PK20_8}) and \cite[Theorem $1.2$]{MS13}, a direct corollary of Theorem~\ref{thm:PL21_1} is the following estimate. 

\begin{cor}
For any Ricci shrinker $(M^n, g, f)\in \mathcal M_n(A)$, there exists a constant $I=I(n,A)<\infty$ such that
  \begin{align*}
\int |\na Rc|^2 e^{-f}\,dV = \int |div(Rm)|^2 e^{-f}\,dV \le I.
  \end{align*}
\end{cor}

Theorem~\ref{thm:PL21_1} is an important step for verifying maximum principle on curvature operators. 
The curvature operator on two-forms are defined as $\mathcal{R}: \Lambda^2 \to \Lambda^2: \mathcal R(e^i \wedge e^j,e^k \wedge e^l)=R_{ijkl}$. The two-form $e^i \wedge e^j \coloneqq e^i \otimes e^j-e^j\otimes e^i$ and the inner product on $\Lambda^2$ is defined as $\la A,B\ra \coloneqq -\frac{1}{2}tr(AB)$ for $A,B\in \Lambda^2=\mathfrak{so}(n)$. In other words, for $w=\frac{1}{2}\sum_{i,j}w_{ij}e^i \wedge e^j$, we have
\begin{align*}
\mathcal R(w)_{ij}=\frac{1}{2}R_{ijkl}w_{kl}.
\end{align*}
In the setting of Ricci shrinker $(M^n,g,f)$, the following equation (see ~\cite{Ham86}) holds:
\begin{align*}
\Delta_f \mathcal R=\mathcal R-2Q(\mathcal R).
\end{align*}
Here $Q(\mathcal R) \coloneqq \mathcal R^2+\mathcal R^{\#}$ and $\mathcal R^{\#}$ is defined as
\begin{align*}
\mathcal R^{\#}(u,v)=-\frac{1}{2}tr(ad_u \,\mathcal R\, ad_v\, \mathcal R)
\end{align*}
for any $u,v \in \Lambda^2$. If we choose an orthonormal basis $\{\phi_i\}$ of $\Lambda^2$, then
\begin{align*}
\mathcal R^{\#}(u,v)=-\frac{1}{2}\sum_{i,j}\la [\mathcal R(\phi_i),\phi_j],u \ra \la [\mathcal R(\phi_j),\phi_i],v \ra.
\end{align*}

If we assume $\lambda_1 \le \lambda_2 \le \cdots$ are all eigenvalues of $\mathcal R$ on $\Lambda^2$, then we have the following rigidity theorem.

\begin{thm}
\label{T:rigidity}
There exists a constant $\ep=\ep(n)>0$ such that for any Ricci shrinkers $(M^n,g,f)$, if $\lambda_2 \ge -\ep \dfrac{\lambda_1^2}{|R-2\lambda_1|}$, then $\lambda_1 \geq 0$.
Consequently, $(M^n,g)$ is isometric to a quotient of $N^k \times \R^{n-k}$ for some $0 \le k \le n$, where $N^k$ is a closed symmetric space.
\end{thm}

\begin{proof}
It suffices to prove $\lambda_1 \ge 0$. Namely,  $(M^n,g)$ has nonnegative curvature operator. The further conclusion follows from~\cite[Corollary $4$]{MW17}.

We fix a point $q$ and assume that $\phi_1$ is an eigenvector of $\lambda_1$. Extending $\phi_1$ by parallel transport on a small neighborhood of $q$, we have
\begin{align*}
\Delta_f \mathcal R(\phi_1,\phi_1)=\mathcal R(\phi_1,\phi_1)-2Q(\mathcal R)(\phi_1,\phi_1).
\end{align*}
Therefore if we assume that $\phi_i$ are eigenvectors of $\lambda_i$, then in the barrier sense,
\begin{align}
\Delta_f \lambda_1 \le& \lambda_1-\lc2\lambda_1^2-\sum_{i,j}\la [\mathcal R(\phi_i),\phi_j],\phi_1 \ra\la [\mathcal R(\phi_j),\phi_i],\phi_1 \ra\rc \notag \\
=&\lambda_1-\lc 2\lambda_1^2+\sum_{i,j}C_{ij}^2 \lambda_i\lambda_j \rc \label{eq:rig1}
\end{align}
where $C_{i,j}=\la [\phi_i,\phi_j],\phi_1 \ra$. Notice that $C_{i,j}=0$ if $i=1$ or $j=1$.

We claim that $|C_{i,j}| \le 2$. Indeed, if we assume that $\phi_i,\phi_j$ and $\phi_1$ are represented by the antisymmetric matrices $A,B$ and $C$ respectively, then $C_{i,j}=-\frac{1}{2}tr((AB-BA)C)=-tr(ABC)$.
By choosing a basis such that $A_{2k-1,2k}=a_k=-A_{2k,2k-1}$ for $k \le [n/2]$ and $0$ otherwise, we have
\begin{align*}
|tr(ABC)|\le& \sum_{k,l}|a_k||B_{2k,l}C_{l,2k-1}-B_{2k-1,l}C_{l,2k}| \\
\le& \frac{1}{2}\sum_{k,l} (B^2_{2k,l}+C^2_{l,2k-1}+B^2_{2k-1,l}+C^2_{l,2k}) \\
\le& \frac{1}{2}(|B|^2+|C|^2)=2.
\end{align*}
Here we have used the fact that $|A|^2=|B|^2=|C|^2=2$.

Next we prove that if $\ep$ is properly chosen, then we have
\begin{align*}
  P \coloneqq 2\lambda_1^2+\sum_{i,j}C_{ij}^2 \lambda_i\lambda_j \ge 0. 
\end{align*}
 From the definition of $\lambda_i$, we notice that $\sum \lambda_i=R/2$. Therefore, we fix $\lambda_1$ and $\lambda_2$ and minimize $P$ under the restriction $\sum \lambda_i=R/2$. We can assume that $\lambda_2<0$, otherwise $P \ge 0$ from its definition. We also set $c_n=n(n-1)/2$ and assume that $\lambda_1 \le \lambda_2 \le \cdots \le \lambda_{s+1}$ are all eigenvalues smaller than $0$. Therefore,
\begin{align*}
P\ge P_1 \coloneqq 2\lambda_1^2+2\sum_{\substack{2\le i\le s+1 \\s+2 \le j \le c_n}}C_{ij}^2 \lambda_i\lambda_j.
\end{align*}
It is easy to show that $P_1$ is minimized when $\lambda_2=\lambda_3=\cdots=\lambda_{s+1}$ and $\lambda_{s+2}=\cdots=\lambda_{c_n}$. It follows that
\begin{align*}
\frac{P_1}{2}\ge& \lambda_1^2+\sum_{\substack{2\le i\le s+1 \\s+2 \le j \le c_n}}\frac{1}{c_n-s-1}C^2_{i,j}\lambda_2(R/2-\lambda_1-s\lambda_2) \\
\ge&\lambda_1^2+4s\lambda_2(R/2-\lambda_1-s\lambda_2).
\end{align*}
By solving the above quadratic inequality, we obtain that $P_1$ and hence $P$ are nonnegative if
\begin{align*}
\lambda_2 \ge \frac{\frac{R}{2}-\lambda_1-\sqrt{(\frac{R}{2}-\lambda_1)^2+\lambda_1^2}}{2s}.
\end{align*}
If we choose $\ep=\frac{1}{(1+\sqrt 2)(c_n-2)}$, then it is clear that for any $1\le s \le c_n-2$,
\begin{align*}
\lambda_2 \ge -\ep \frac{\lambda_1^2}{R-2\lambda_1} \ge \frac{\frac{R}{2}-\lambda_1-\sqrt{(\frac{R}{2}-\lambda_1)^2+\lambda_1^2}}{2(c_n-2)}\ge \frac{\frac{R}{2}-\lambda_1-\sqrt{(\frac{R}{2}-\lambda_1)^2+\lambda_1^2}}{2s}.
\end{align*}
Therefore, from \eqref{eq:rig1} we obtain $\Delta_f \lambda_1 \le \lambda_1$. Since $\lambda_1 \in L^2(e^{-f}\,dV)$ by (\ref{eqn:PK20_8}), then it follows from \cite[Theorem $4.4$]{PW10} that $\lambda_1 \ge 0$.
\end{proof}

We conclude this section by the proof of Theorem \ref{thm:D}.

\emph{Proof of Theorem \ref{thm:D}}: Since $\lambda_2 \geq 0$, we can apply Theorem~\ref{T:rigidity} to obtain $\lambda_1 \geq 0$.
Therefore, $M^n$ is a finite quotient of $N^k \times \R^{n-k}$. Note that only the case $k=n$ is possible. For otherwise the second smallest eigenvalue must be $0$. 
Since $N^n$ is a compact Einstein manifold such that the curvature operator is $2$-positive, it follows from \cite{BW08} that its universal covering must be $S^n$.

  

\vskip10pt

Yu Li, Department of Mathematics, Stony Brook University, Stony Brook, NY 11794, USA; yu.li.4@stonybrook.edu.\\

Bing Wang, School of Mathematical Sciences, University of Science and Technology of China, No. 96 Jinzhai Road, Hefei, Anhui Province, 230026, China; Department of Mathematics, University of Wisconsin-Madison, Madison, WI 53706, USA; bwang@math.wisc.edu.\\

\end{document}